\documentclass{article}

\usepackage{authblk}

\usepackage[utf8]{inputenc}
\usepackage[T1]{fontenc}
\usepackage{textcomp}
\usepackage{amsmath,amssymb,amsthm}
\usepackage{lmodern}
\usepackage[a4paper]{geometry}
\usepackage{xcolor, pict2e}
\usepackage{microtype}
\usepackage{caption}
\usepackage{listings}
\usepackage{multicol}
\usepackage{moreverb}
\usepackage{hyperref}
\hypersetup{pdfstartview=XYZ}
\usepackage{wrapfig}
\usepackage[sans]{dsfont}

\usepackage{hyperref}
\hypersetup{
    colorlinks=true,
    linkcolor=blue,
    citecolor=magenta,
    urlcolor=blue,
    pdfborder={0 0 0}
}

\usepackage[font=sf, labelfont={sf,bf}, margin=1cm]{caption}

\usepackage{cite}

\usepackage{float}
\usepackage[caption = false]{subfig}
\usepackage{graphicx}

\newtheorem*{theorem*}{Theorem}
\newtheorem{theorem}{Theorem}[section]
\newtheorem{proposition}{Proposition}[section]
\newtheorem{lemma}{Lemma}[section]
\newtheorem{corollary}{Corollary}[section]

\newtheorem{theoremm}{Theorem}

\theoremstyle{remark}
\newtheorem{remark}{Remark}[section]

\newcommand{\R}{\mathbb{R}}
\newcommand{\N}{\mathbb{N}}
\newcommand{\Z}{\mathbb{Z}}
\newcommand{\Q}{\mathbb{Q}}

\newcommand{\D}{\mathbb{D}}

\newcommand{\Bc}{\mathcal{B}}
\newcommand{\Cc}{\mathcal{C}}

\newcommand{\Mc}{\mathcal{M}}
\newcommand{\Nc}{\mathcal{N}}

\newcommand{\Tc}{\mathcal{T}}

\newcommand{\expect}{\mathbb{E}}
\newcommand{\Expect}[1]{\mathbb{E} \left[ #1 \right] }
\newcommand{\EXPECT}[2]{\mathbb{E}_{#1} \left[ #2 \right] }
\newcommand{\prob}{\mathbb{P}}
\newcommand{\Prob}[1]{\mathbb{P} \left( #1 \right) }
\newcommand{\PROB}[2]{\mathbb{P}_{#1} \left( #2 \right) }

\newcommand{\brob}{\mathds{P}}
\newcommand{\Brob}[1]{\mathds{P} \left( #1 \right) }
\newcommand{\BROB}[2]{\mathds{P}_{#1} \left( #2 \right) }
\newcommand{\expectB}{\mathds{E}}
\newcommand{\ExpectB}[1]{\mathds{E} \left[ #1 \right] }
\newcommand{\EXPECTB}[2]{\mathds{E}_{#1} \left[ #2 \right] }

\newcommand{\abs}[1]{\left\vert #1 \right\vert}
\newcommand{\norme}[1]{\left\| #1 \right\| }
\newcommand{\scalar}[1]{\left\langle #1 \right\rangle }
\newcommand{\floor}[1]{\left\lfloor #1 \right\rfloor}
\newcommand{\indic}[1]{ \mathbf{1}_{ \left\{ #1 \right\} } }
\newcommand{\eps}{\varepsilon}

\usepackage{tikz}

\title{Planar Brownian motion and Gaussian multiplicative chaos}
\author{Antoine Jego\thanks{On leave from the University of Cambridge, partly supported by the EPSRC grant
EP/L016516/1 for the University of Cambridge Centre for Doctoral Training, the Cambridge Centre for Analysis. E-mail address: \href{mailto:apfj2@cam.ac.uk}{apfj2@cam.ac.uk}}
}
\affil{University of Vienna}
\date {}

\numberwithin{equation}{section}

\begin{document}

\renewcommand{\theparagraph}{\thesubsection.\arabic{paragraph}} 

\maketitle

\begin{abstract}
We construct the analogue of Gaussian multiplicative chaos measures for the local times of planar Brownian motion by exponentiating the square root of the local times of small circles. We also consider a flat measure supported on points whose local time is within a constant of the desired thickness level and show a simple relation between the two objects. Our results extend those of \cite{bass1994} and in particular cover the entire $L^1$-phase or subcritical regime. These results allow us to obtain a nondegenerate limit for the appropriately rescaled size of thick points, thereby considerably refining estimates of \cite{dembo2001}.
\end{abstract}

\tableofcontents

\section{Introduction}

\subsection{Main results}

Gaussian multiplicative chaos (GMC) introduced by Kahane \cite{kahane} consists in defining and studying the properties of random measures formally defined as the exponential of a log-correlated Gaussian field, such as the two-dimensional Gaussian free field (GFF). Since such a field is not defined pointwise but is rather a random generalised function, making sense of such a measure requires some nontrivial work. The theory has expanded significantly in recent years and by now it is relatively well understood, at least in the subcritical case \cite{RobertVargas2010, DuplantierSheffieldGMC, RhodesVargasGMC, ShamovGMC, berestycki2017} and even in the critical case \cite{DRSV2014b, DRSV2014a, JunnilaSaksman2017, JSW2018, powell2018}. Furthermore, Gaussian multiplicative chaos appears to be a universal feature of log-correlated fields going beyond the Gaussian theory discussed in these papers. Establishing universality for naturally arising models is a very active and important area of research. We mention the work of \cite{SaksmanWebb2016} on the Riemann $\zeta$ function on the critical line and the work of \cite{FyodorovKeating2014, Webb2015, NikulaSaksmanWebb2018, Lambert2018, BerestyckiWebWong2018 } on large random matrices.

The goal of this paper is to study Gaussian multiplicative chaos for another natural non-Gaussian log-correlated field: (the square root of) the local times of two-dimensional Brownian motion.
\medskip

Before stating our main results, we start by introducing a few notations.
Let $\prob_x$ be the law under which $(B_t)_{t \geq 0}$ is a planar Brownian motion starting from $x \in \R^2$. Let $D \subset \R^2$ be an open bounded simply connected domain, $x_0 \in D$ a starting point and $\tau$ be the first exit time of $D$:
\[
\tau := \inf \{ t \geq 0: B_t \notin D \}.
\]
For all $x \in \R^2, t>0, \eps >0,$ define $L_{x,\eps}(t)$ the local time of $\left( \abs{B_s - x}, s \geq 0 \right)$ at $\eps$ up to time $t$ (here $\abs{\cdot}$ stands for the Euclidean norm):
\[
L_{x,\eps}(t) := \lim_{\substack{r \rightarrow 0 \\ r >0}} \frac{1}{2 r} \int_0^t \indic{ \eps - r \leq \abs{B_s - x} \leq \eps + r} ds.
\]
One can use classical theory of one-dimensional semimartingales to get existence for a fixed $x$ of $\{ L_{x,\eps}(\tau), \eps > 0 \}$ as a process. In this article, we need to make sense of $L_{x,\eps}(\tau)$ jointly in $x$ and in $\eps$. It is provided by Proposition \ref{prop:continuity} that we state at the end of this section. If the circle $\partial D(x,\eps)$ is not entirely included in $D$, we will use the convention $L_{x,\eps}(\tau) = 0$.
For all $\gamma \in (0,2)$ we consider the sequence of random measures $\mu_\eps^\gamma(dx)$ on $D$ defined by: for all Borel sets $A \subset D$,
\begin{align}
\label{eq:def mu_eps^gamma}
\mu_\eps^\gamma(A) & := \sqrt{\abs{\log \eps}} \eps^{\gamma^2/2} \int_A e^{\gamma \sqrt{\frac{1}{\eps}L_{x,\eps}(\tau)} } dx.
\end{align}
The presence of the square root in the exponential may appear surprising at first glance, but it is natural nevertheless in view of Dynkin-type isomorphisms (see \cite{rosen_2014}).

To capture the fractal geometrical properties of a log-correlated field, another natural approach consists in encoding the so-called thick points (points where the field is unusually large) in flat measures supported on those thick points. At criticality, such measures are often called extremal processes. See for instance \cite{BiskupLouidor}, \cite{BiskupLouidor2018} in the case of discrete two-dimensional GFF, see also \cite{abe2018} in the case of simple random walk on trees. In our case, we can consider for all $\gamma \in (0,2)$ the sequence of random measures $\nu_\eps^\gamma (dx,dt)$ on $D \times \R$ defined by: for all Borel sets $A \subset D$ and $T \subset \R$,
\begin{equation}
\label{eq:def nu_eps^gamma}
\nu_\eps^\gamma(A \times T) := \abs{ \log \eps } \eps^{-\gamma^2/2} \int_A \indic{ \sqrt{\frac{1}{\eps}L_{x,\eps}(\tau)} - \gamma \log \frac{1}{\eps} \in T} dx.
\end{equation}

\begin{theorem}\label{th:convergence measures}
For all $\gamma \in (0,2)$, the sequences of random measures $\nu_\eps^\gamma$ and $\mu_\eps^\gamma$ converge as $\eps \to 0$ in probability for the topology of vague convergence on $D \times ( \R \cup \{ + \infty \} )$ and on $D$ respectively towards Borel measures $\nu^\gamma$ and $\mu^\gamma$.
\end{theorem}

The measure $\nu^\gamma$ can be decomposed as a product of a measure on $D$ and a measure on $\R$. Moreover, the component on $D$ agrees with $\mu^\gamma$ and the component on $\R$ is exponential:

\begin{theorem}\label{th:identification limits}
For all $\gamma \in (0,2)$, we have $\prob_{x_0}$-a.s.,
\[
\nu^\gamma(dx,dt) = (2\pi)^{-1/2} \mu^\gamma(dx) e^{- \gamma t} dt.
\]
Moreover, by denoting $R(x,D)$ the conformal radius of $D$ seen from $x$ and $G_D(x_0,x)$ the Green function of $D$ in $x_0$, $x$ (see \eqref{eq:def Green function}), we have for all Borel set $A \subset D$,
\begin{equation}
\label{eq:th formula first moment}
\EXPECT{x_0}{ \mu^\gamma(A) } =
\sqrt{2 \pi} \gamma \int_A R(x,D)^{\gamma^2/2} G_D(x_0,x) dx \in (0,\infty).
\end{equation}
\end{theorem}

The decomposition of $\nu^\gamma$ and \eqref{eq:th formula first moment} justify that the square root of the local times is the right object to consider. These two properties are very similar to the case of the two-dimensional GFF (see \cite{BiskupLouidor} and \cite{berestycki}, Theorem 2.1 for instance).

Simulations of $\mu^\gamma$ can be seen in Figure \ref{figure}. They have been performed using simple random walk on the square lattice killed when it exits a square composed of $401 \times 401$ vertices. 

\begin{figure}
\subfloat[$\gamma=0.3$]{\includegraphics[scale=0.45,trim={8cm 4cm 6.9cm 7cm},clip]{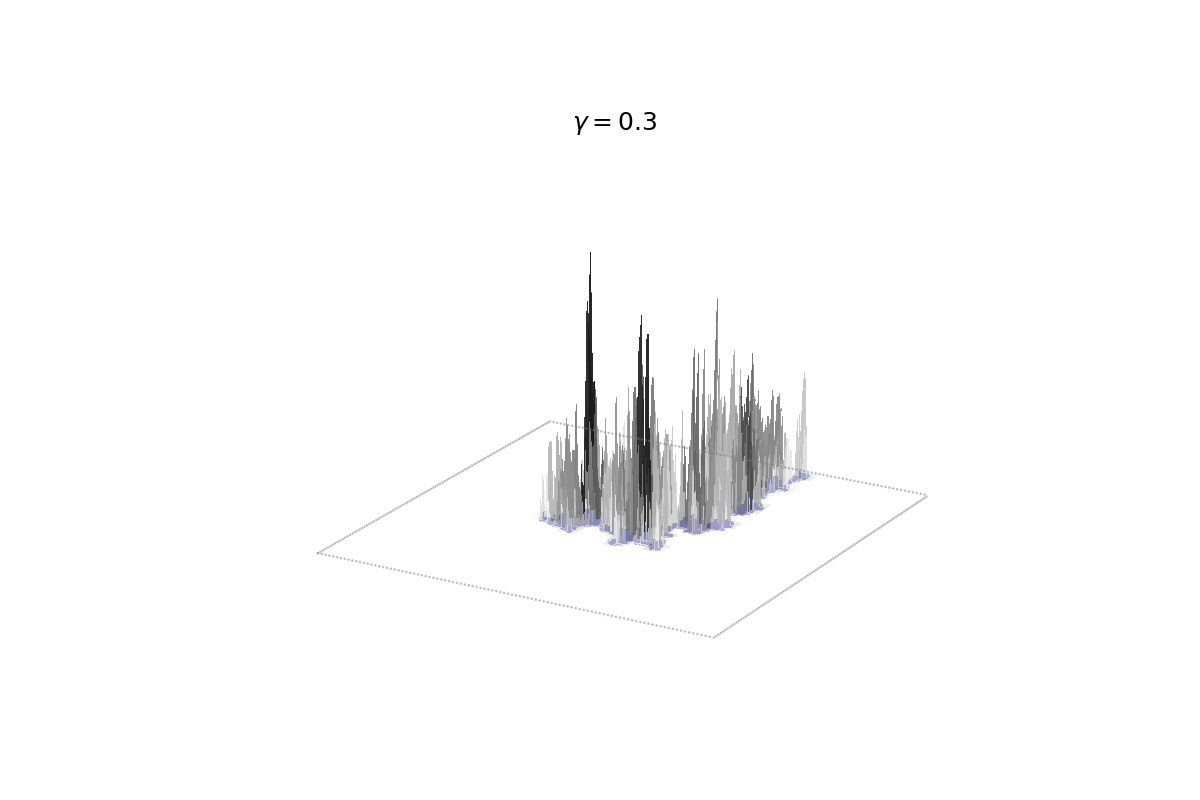}} 
\subfloat[$\gamma=0.8$]{\includegraphics[scale=0.45,trim={8cm 4cm 6.9cm 8cm},clip]{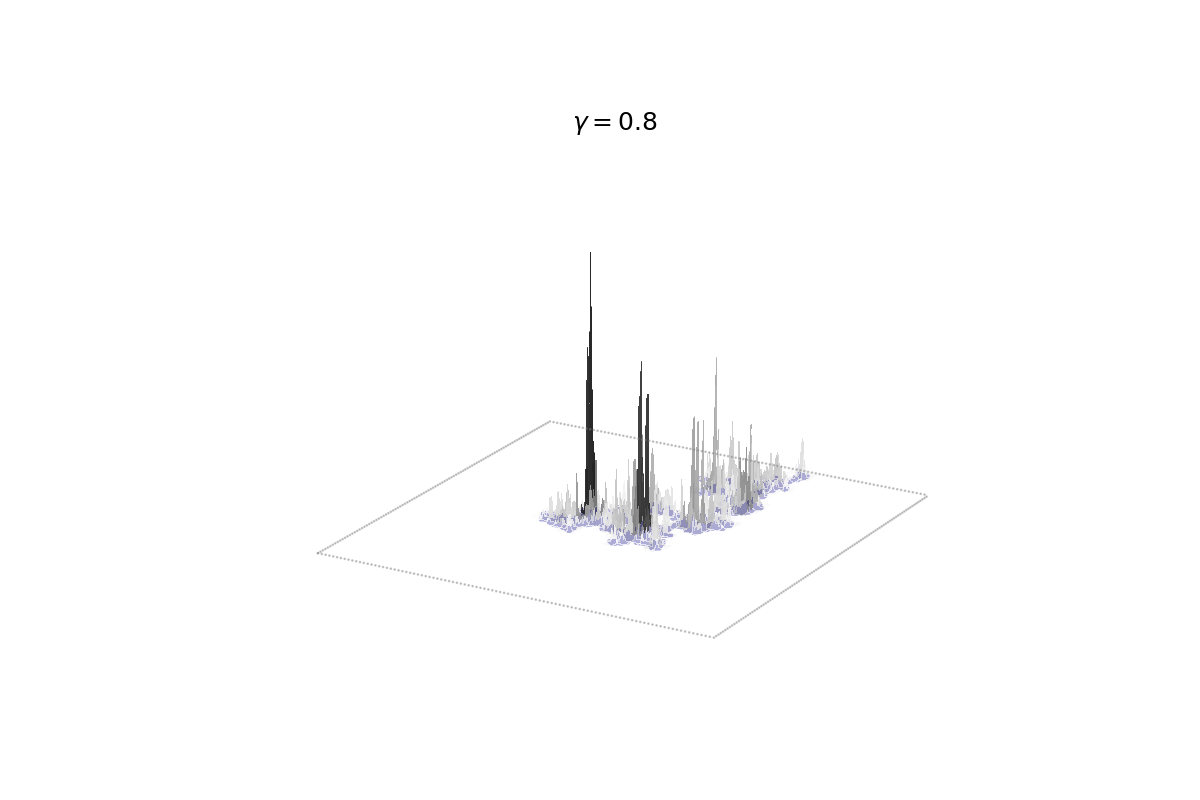}}\\
\subfloat[$\gamma=1.3$]{\includegraphics[scale=0.45,trim={8cm 4cm 6.9cm 8.5cm},clip]{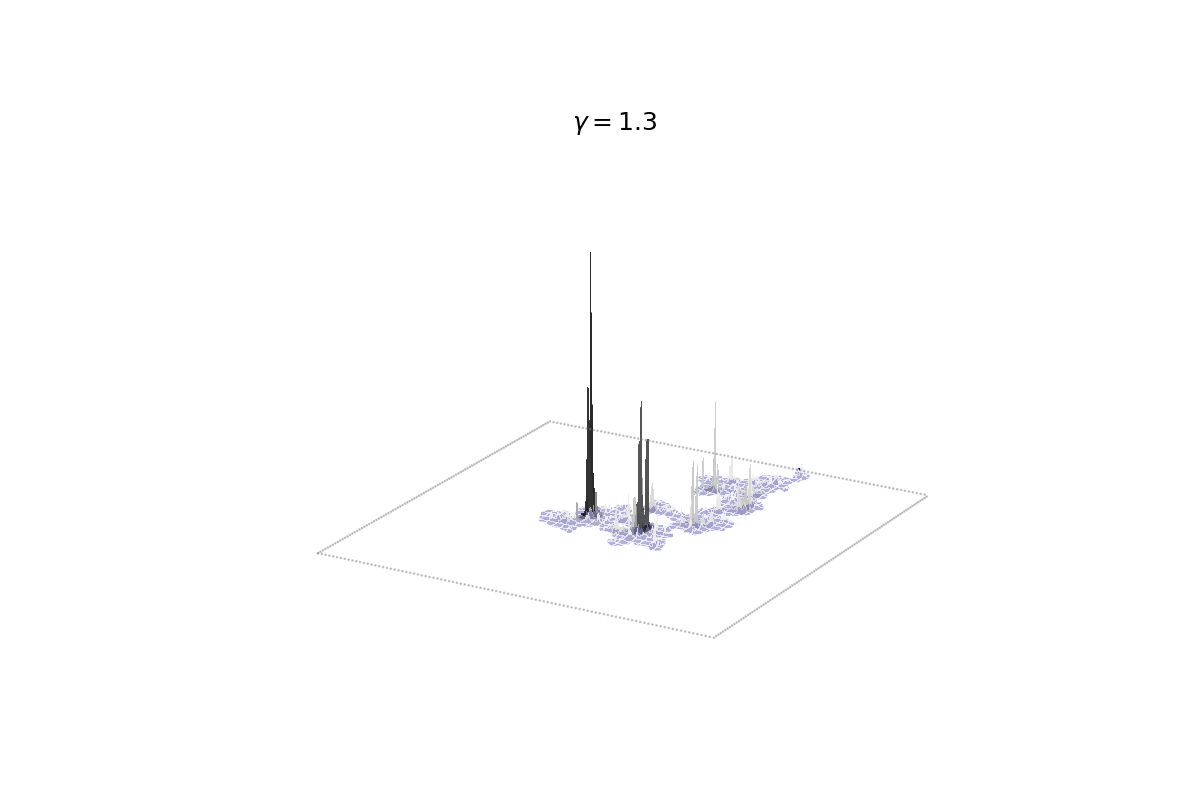}}
\subfloat[$\gamma=1.8$]{\includegraphics[scale=0.45,trim={8cm 4cm 6.9cm 8.5cm},clip]{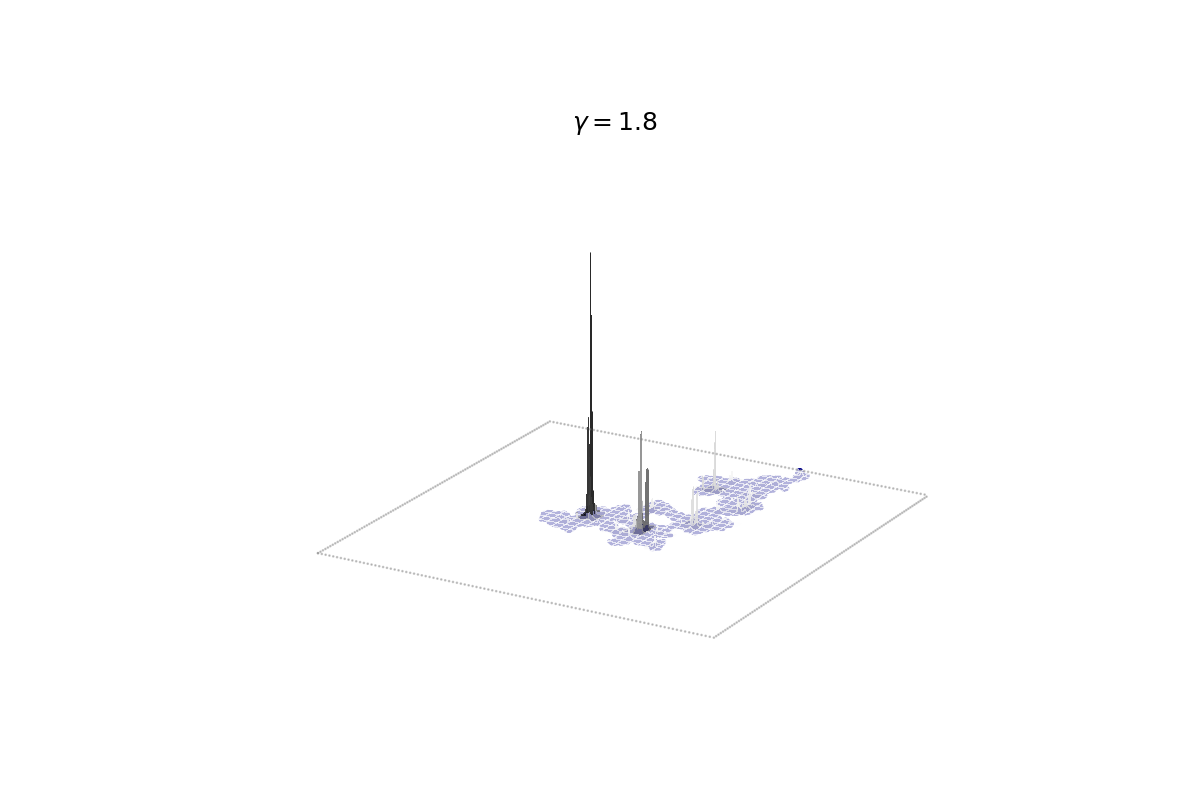}}
\caption{Simulation of $\mu^\gamma$ for $\gamma = 0.3$, $0.8$, $1.3$ and $1.8$, for the same underlying sample of Brownian path which is drawn in blue. The domain $D$ is a square and the starting point $x_0$ is its middle}
\label{figure}
\end{figure}

In \cite{bass1994}, a slight modification of $\nu_\eps^\gamma(dx,(0,\infty))$ was shown to converge for $\gamma \in (0,1)$ and the authors conjectured that the convergence should hold for the whole range $\gamma \in (0,2)$. One part of Theorem \ref{th:convergence measures} settles this question.
Let us also mention that the random measure $\mu^\gamma$ has been constructed very recently in \cite{AidekonHuShi2018} through a very different method. In Section \ref{subsec:relation with other works} below, we explain carefully the relation between the articles \cite{bass1994}, \cite{AidekonHuShi2018} and the current paper.

\begin{remark}
We decided to not include the case $\gamma = 0$ to ease the exposition, but notice that $\nu_\eps^\gamma$ is also a sensible measure in this case. By modifying very few arguments in the proofs of Theorems \ref{th:convergence measures} and \ref{th:identification limits}, one can show that this sequence of random measures converges for the topology of vague convergence on $D \times (0,\infty]$ towards a measure $\nu^0$ which can be decomposed as
\[
\nu^0(dx,dt) = \mu^0(dx)\indic{t = \infty}
\]
for some random Borel measure $\mu^0$ on $D$. With the help of \eqref{eq:prop formula} in Proposition \ref{prop:formula} characterising the measure $\mu^\gamma$, it can be shown that $\mu^0$ is actually $\prob_{x_0}$-a.s. absolutely continuous with respect to the occupation measure of Brownian motion, with a deterministic density. This last observation was already made in \cite{AidekonHuShi2018}, Section 7.
\end{remark}

Define the set of $\gamma$-thick points at level $\eps$ by
\begin{equation}
\label{eq:def thick points level epsilon}
\Tc_\eps^\gamma := \left\{ x \in D: \frac{L_{x,\eps}(\tau)}{\eps (\log \eps)^2} \geq \gamma^2 \right\}.
\end{equation}
This is similar to the notion of thick points in \cite{dembo2001}, except that they look at the occupation measure of small discs rather than small circles.
In \cite{jego2018}, the question to show the convergence of the rescaled number of thick points for the simple random walk on the two-dimensional square lattice was raised. As a direct corollary of Theorems \ref{th:convergence measures} and \ref{th:identification limits}, we answer the analogue of this question in the continuum:

\begin{corollary}\label{cor:convergence thick points}
For all $\gamma \in (0,2)$, we have the following convergence in $L^1$:
\[
\lim_{\eps \to 0} \abs{ \log \eps } \eps^{-\gamma^2/2} \mathrm{Leb} \left( \Tc_\eps^\gamma \right) = \frac{1}{\sqrt{2\pi}\gamma} \mu^\gamma(D)
\]
where $\mathrm{Leb} \left( \Tc_\eps^\gamma \right)$ denotes the Lebesgue measure of $\Tc_\eps^\gamma$.
\end{corollary}

Despite the strong links between the GFF and the local times, this shows a difference in the structure of the thick points of GFF compare to those of planar Brownian motion which cannot be observed through rougher estimates such as the fractal dimension.
Indeed, for the GFF, the normalisation is $\sqrt{\abs{ \log \eps }} \eps^{-\gamma^2/2}$ instead of $\abs{ \log \eps } \eps^{-\gamma^2/2}$. See \cite{jego2018} for more about this.

As announced earlier, in order to define the measures in \eqref{eq:def mu_eps^gamma} and \eqref{eq:def nu_eps^gamma}, we establish:

\begin{proposition}\label{prop:continuity}
The local time process $L_{x,\eps}(\tau)$, $x \in D$, $0 < \eps < \mathrm{d}(x,\partial D)$, possesses a jointly continuous modification $\tilde{L}_{x,\eps}(\tau)$. In fact, this modification is $\alpha$-H\"older for all $\alpha < 1/3$.
\end{proposition}

The proof of this proposition will be given in Appendix \ref{sec:Appendix continuity local times}. In the rest of the paper, when we write $L_{x,\eps}(\tau)$ we actually always mean its continuous modification $\tilde{L}_{x,\eps}(\tau)$.

\subsection{Relation with other works and further results}\label{subsec:relation with other works}

The construction of measures supported on the thick points of planar Brownian motion was initiated by the work of Bass, Burdzy and Khoshnevisan \cite{bass1994}. The notion of thick points therein is defined through the number of excursions $N_\eps^x$ from $x$ which hit the circle $\partial D(x,\eps)$, before the Brownian motion exits the domain $D$: more precisely, for $a \in (0,2)$, they define the set
\begin{equation}
\label{eq:def thick points BBK}
A_a := \left\{ x \in D: \lim_{\eps \to 0} \frac{N_\eps^x}{\abs{\log \eps}} = a \right\}.
\end{equation}
Note that our parametrisation is somewhat different; it is chosen to match the GMC theory. Informally, the relation between the two is given by $a = \gamma^2/2$.
Next, we recall that the carrying dimension of a measure $\beta$ is the infimum of $d>0$ for which there exists a set $A$ such that $\beta(A^c)=0$ and the Hausdorff dimension of $A$ is equal to $d$. They showed:

\begin{theoremm}[Theorem 1.1 of \cite{bass1994}]\label{thmm:BBK}
Assume that the domain $D$ is the unit disc of $\R^2$ and that the starting point $x_0$ is the origin. For all $a \in (0,1/2)$, with probability one there exists a random measure $\beta_a$, which is carried by $A_a$ and whose carrying dimension is equal to $2 - a$.
\end{theoremm}

In \cite{bass1994}, the measure $\beta_a$ is constructed as the limit of measures $\beta_a^\eps$ as $\eps \to 0$ which are defined in a very similar manner as our measures $\nu_\eps^\gamma(dx, (0,\infty))$ using local times of circles (see the beginning of Section 3 of \cite{bass1994}). We emphasise here the difference of renormalisation: the local times they consider are half of our local times. We also mention that the range $\{ a \in (0,1/2) \}$ for which they were able to show the convergence of $\beta_a^\eps$ is a strict subset of the so-called $L^2$-phase of the GMC, which would correspond to $\{ a \in (0,1) \}$ or  $\{ \gamma \in (0,\sqrt{2}) \}$. This is the region where $\beta_a^\eps(D)$ is bounded in $L^2$, see Theorem 3.2 of \cite{bass1994}.

Bass, Burdzy and Khoshnevisan also gave an effective description of their measure $\beta_a$ in terms of a Poisson point process of excursions. More precisely, they define a probability distribution $\Q_{x,D}^{x_0,a}$ (written $\Q_a^x$ in \cite{bass1994}, defined just before Proposition 5.1 of \cite{bass1994}) on continuous trajectories which can be understood heuristically as follows. The trajectory of a process under $\Q_{x,D}^{x_0,a}$ is composed of three independent parts. The first one is a Brownian motion starting from $x_0$ conditioned to visit $x$ before exiting $D$ and killed at the hitting time of $x$. The third part is a Brownian motion starting from $x$ and killed when it exits for the first time $D$. The second part is composed of an infinite number of excursions from $x$ generated by a Poisson point process with the intensity measure being the product of the Lebesgue measure on $[0,a]$ and an excursion law. In Proposition 5.1 of \cite{bass1994}, they roughly speaking show that the law of the Brownian motion conditioned on the fact that $x$ has been sampled according to $\beta_a$ is $\Q_{x,D}^{x_0,a}$. This characterises their measure $\beta_a$ (Theorem 5.2 of \cite{bass1994}). Once Theorems \ref{th:convergence measures} and \ref{th:identification limits} above are established, we can adapt their arguments for the proof of characterisation to conclude the same thing for our measure $\mu^\gamma$: see Proposition \ref{prop:formula} for a precise statement.
A consequence of Proposition \ref{prop:formula} is the identification of our measure $\mu^\gamma$ with their measure $\beta_a$:

\begin{corollary}\label{cor:mu=beta}
If the domain $D$ is the unit disc, $x_0$ the origin, $\gamma \in (0,1)$ and $a = \gamma^2/2$, we have $\prob_{x_0}$-a.s. $\mu^\gamma = \sqrt{2\pi} \gamma \beta_a$.
\end{corollary}

A consequence of Theorem \ref{thmm:BBK} is a lower bound on the Hausdorff dimension of the set of thick points $A_a$: for all $a \in (0,1/2)$, a.s. $\dim(A_a) \geq 2-a$. The upper bound they obtained (\cite{bass1994}, Theorem 1.1 (ii)) is: for all $a > 0$, a.s. $\dim(A_a) \leq \max(0,2 - a/e)$. They conjectured that the lower bound is sharp and holds for all $a \in (0,2)$. In 2001, Dembo, Peres, Rosen and Zeitouni \cite{dembo2001} answered positively the analogue of this question for thick points defined through the occupation measure of small discs:
\begin{equation}
\label{eq:def thick points DPRZ}
\Tc_a := \left\{ x \in D: \lim_{\eps \to 0} \frac{1}{\eps^2 (\log \eps)^2} \int_0^{\tau} \indic{B_t \in D(x,\eps)} dt = a \right\}.
\end{equation}
In particular, their result went beyond the $L^2$-phase to cover the entire $L^1$-phase. This allowed them to solve a conjecture by Erd\H{o}s and Taylor \cite{erdos_taylor1960}. 
\medskip

Very recently, Aïdékon, Hu and Shi \cite{AidekonHuShi2018} made a link between the definitions of thick points of \cite{bass1994} and \cite{dembo2001} (defined in \eqref{eq:def thick points BBK} and \eqref{eq:def thick points DPRZ} respectively) by constructing measures supported on these two sets of thick points. Their approach is superficially very different from ours but we will see that the measure $\mu^\gamma$ we obtained is, perhaps surprisingly, related to theirs in a strong way (Corollary \ref{cor:mu=Mc} below). Their measure is defined through a martingale approach for which the interpretation of the approximation is not immediately transparent (see \cite{AidekonHuShi2018} (4.1), (4.2) and Corollary 3.6).

Let us describe this relation in more details. 
For technical reasons, in \cite{AidekonHuShi2018}, the boundary $\partial D$ of $D$ is assumed to be a finite union of analytic curves. To compare our results with theirs, we will also make this extra assumption in the following and we will call such a domain a nice domain.
Consider $z \in \partial D$ a boundary point such that the boundary of $D$ is analytic locally around $z$; we will call such a point a nice point. They denote by $\prob_D^{x_0,z}$ the law of a Brownian motion starting from $x_0$ and conditioned to exit $D$ through $z$. They showed:
 
\begin{theoremm}[Theorem 1.1 of \cite{AidekonHuShi2018}]
For all $a \in (0,2)$, with $\prob_D^{x_0,z}$-probability one there exists a random measure $\Mc_\infty^a$ which is carried by $A_a$ and by $\Tc_a$ and whose carrying dimension is equal to $2-a$.
\end{theoremm}

Their starting point is the interpretation of the measure $\beta_a$ of \cite{bass1994} described above in terms of Poisson point process of excursions. For $x \in D$, they define a measure $\Q_{x,D}^{x_0,z,a}$ on trajectories similar to $\Q_{x,D}^{x_0,a}$ mentioned above: the only difference is that the last part of the trajectory is a Brownian motion conditioned to exit the domain through $z$. In a nutshell, they show the absolute continuity of $\Q_{x,D}^{x_0,z,a}$ with respect to $\prob_D^{x_0,z}$ (restricted to the event that the Brownian path stays away from $x$) and define a sequence of measures using the Radon-Nikodym derivative. Their convergence relies on martingales argument rather than on computations on moments. As in \cite{bass1994}, they obtain a characterisation of their measure in terms of $\Q_{x,D}^{x_0,z,a}$ (\cite{AidekonHuShi2018}, Proposition 5.1) matching with ours (Proposition \ref{prop:formula}). As a consequence, we are able to compare their measure with ours.

Before stating this comparison, let us notice that we can also make sense of our measure $\mu^\gamma$ for the Brownian motion conditioned to exit $D$ through $z$. Indeed, as noticed in \cite{bass1994}, Remark 5.1 (i), our measure $\mu^\gamma$ is measurable with respect to the Brownian path and defined locally. $\mu^\gamma$ is thus well defined for any process which is locally mutually absolutely continuous with respect to the two dimensional Brownian motion killed when it exits for the first time the domain $D$. The Brownian motion conditioned to exit $D$ through $z$ being such a process, $\mu^\gamma$ makes sense under $\prob_D^{x_0,z}$ as a measure on $D$.

\begin{corollary}\label{cor:mu=Mc}
Let $z \in \partial D$ be a nice point and denote by $H_D(x,z)$ the Poisson kernel of $D$ from $x$ at $z$, that is the density of the harmonic measure $\PROB{x}{B_\tau \in \cdot}$ with respect to the Lebesgue measure of $\partial D$ at $z$.
For all $\gamma \in (0,2)$, if $a = \gamma^2/2$, we have $\prob_D^{x_0,z}$-a.s.,
\[
\mu^\gamma(dx) = \sqrt{2\pi}\gamma \frac{H_D(x_0,z)}{H_D(x,z)} \Mc_\infty^a(dx).
\]
\end{corollary}

In particular, our measure $\mu^\gamma$ inherits some properties of the measure $\Mc_\infty^a$ obtained in \cite{AidekonHuShi2018}. Recalling the definitions \eqref{eq:def thick points BBK} and \eqref{eq:def thick points DPRZ} of the two sets of thick points $A_a$ and $\Tc_a$, we have:

\begin{corollary}\label{cor:properties mu}
For all $\gamma \in (0,2)$, the following properties hold:

(i) Non-degeneracy: with $\prob_{x_0}$-probability one, $\mu^\gamma(D) >0$.

(ii) Thick points: with $\prob_{x_0}$-probability one, $\mu^\gamma$ is carried by $A_{\gamma^2/2}$ and by $\Tc_{\gamma^2/2}$.

(iii) Hausdorff dimension: with $\prob_{x_0}$-probability one, the carrying dimension of $\mu^\gamma$ is $2 - \gamma^2/2$.

(iv) Conformal invariance: if $\phi : D \to D'$ is a conformal map between two nice domains, $x_0 \in D$, and if we denote by $\mu^{\gamma,D}$ and $\mu^{\gamma,D'}$ the measures built in Theorem 1 for the domains $(D,x_0)$ and $(D',\phi(x_0))$ respectively, we have
\[
\left( \mu^{\gamma,D} \circ \phi^{-1} \right) (dx)
\overset{\mathrm{law}}{=}
\abs{ \phi'(\phi^{-1}(x)) }^{2 + \gamma^2/2} \mu^{\gamma,D'}(dx).
\]
\end{corollary}

Let us mention that we present the previous properties \textit{(i)-(iii)} as a consequence of Corollary \ref{cor:mu=Mc} to avoid to repeat the arguments, but we could have obtained them without the help of \cite{AidekonHuShi2018}: as in \cite{bass1994}, \textit{(i)} and \textit{(ii)} follow from the Poisson point process interpretation of the measure $\mu^\gamma$ (Proposition \ref{prop:formula}) whereas \textit{(iii)} follows from our second moment computations (Proposition \ref{prop:bdd in L2}). On the other hand, it is not clear that our approach yields the conformal invariance of the measure without the use of \cite{AidekonHuShi2018}.

Finally, while there are strong similarities between $\mu^\gamma$ and the GMC measure associated to a GFF (indeed, our construction is motivated by this analogy), there are also essential differences. In fact, from the point of view of GMC theory, the measure $\mu^\gamma$ is rather unusual in that it is carried by the random fractal set $\{B_t, t \leq \tau \}$ and does not need extra randomness to be constructed, unlike say Liouville Brownian motion or other instances of GMC on random fractals.

\subsection{Organisation of the paper}

We now explain the main ideas of our proofs of Theorems \ref{th:convergence measures} and \ref{th:identification limits} and how the paper is organised. The overall strategy of the proof is inspired by \cite{berestycki2017}. To prove the convergence of the measures $\nu_\eps^\gamma$ and $\mu_\eps^\gamma$, it is enough to show that for any suitable $A \subset D$ and $T \subset \R$, the real valued random variables $\nu_\eps^\gamma(A \times T)$ and $\mu_\eps^\gamma(A)$ converge in probability which is the content of Proposition \ref{prop:convergence L1} (we actually show that they converge in $L^1$). As in \cite{berestycki2017}, we will consider modified versions $\widetilde{\nu}_\eps^\gamma$ and $\widetilde{\mu}_\eps^\gamma$ of $\nu_\eps^\gamma$ and $\mu_\eps^\gamma$ by introducing good events (see \eqref{eq:def good event} and \eqref{eq:def good event exponential}): at a given $x \in D$, the local times are required to be never too thick around $x$ at every scale. We will show that introducing these good events does not change the behaviour of the first moment (Propositions \ref{prop:first moment estimates} and \ref{prop:first moment estimates exponential}, Section \ref{sec:first moment estimates}) and it makes the sequences $\widetilde{\nu}_\eps^\gamma(A \times T)$ and $\widetilde{\mu}_\eps^\gamma(A)$ bounded in $L^2$ (Propositions \ref{prop:bdd in L2} and \ref{prop:bdd in L2 exponential}, Section \ref{sec:uniform integrability}). Furthermore, we will see that these two sequences are Cauchy sequences in $L^2$ (Proposition \ref{prop:Cauchy}, Section \ref{sec:Cauchy in L2}) implying in particular that they converge in $L^1$. Section \ref{sec:proof of theorems} finishes the proof of Theorems \ref{th:convergence measures} and \ref{th:identification limits} and demonstrates the links of our work with the ones of \cite{bass1994} and \cite{AidekonHuShi2018} (Corollaries \ref{cor:mu=beta}, \ref{cor:mu=Mc} and \ref{cor:properties mu}).

We now explain a few ideas underlying the proof. If the domain $D$ is a disc $D = D(x,\eta)$ centred at $x$, then it is easy to check (by rotational invariance of Brownian motion and second Ray-Knight isomorphism for local times of one-dimensional Brownian motion) that the local times $L_{x,r}(\tau), r >0,$ have a Markovian structure. More precisely, for all $\eta' \in (0,\eta)$ and all $z \in D(0,\eta) \backslash D(0,\eta')$, under $\prob_z$ and conditioned on $L_{x,\eta'}(\tau)$,
\begin{equation}\label{eq:prop local times and Bessel process}
\left( \frac{ L_{x,r}(\tau) }{ r }, r = \eta' e^{-s}, s \geq 0 \right)
\overset{\mathrm{law}}{=}
\left( R_s^2, s \geq 0 \right)
\end{equation}
with $(R_s,s \geq 0)$ being a zero-dimensional Bessel process starting from $\sqrt{ L_{x,\eta'}(\tau) / {\eta'}}$. This is an other clue that exponentiating the square root of the local times should yield an interesting object.

In the case of a general domain $D$, such an exact description is of course not possible, yet for small enough radii, the behaviour of $L_{x,r}(\tau)$ can be seen to be approximatively given by the one in \eqref{eq:prop local times and Bessel process}. If we assume \eqref{eq:prop local times and Bessel process} then the construction of $\mu^\gamma$ is similar to the GMC construction for GFF, with the Brownian motions describing circle averages replaced by Bessel processes of suitable dimension. It seems intuitive that the presence of the drift term in a Bessel process should not affect significantly the picture in \cite{berestycki2017}.

To implement our strategy and use \eqref{eq:prop local times and Bessel process}, we need an argument. In the first moment computations (Propositions \ref{prop:first moment estimates} and \ref{prop:first moment estimates exponential}), we will need a rough upper bound on the local times; an obvious strategy consists in stopping the Brownian motion when it exits a large disc containing the domain. For the second moment (Proposition \ref{prop:Cauchy}), we will need a much more precise estimate.
Let us assume for instance that $D(x,2) \subset D$. We can decompose the local times $(L_{x,r}(\tau), r < 1)$ according to the different macroscopic excursions from $\partial D(x,1)$ to $\partial D(x,2)$ before exiting the domain $D$. To keep track of the overall number of excursions, we will condition on their initial and final points. Because of this conditioning, the local times of a specific excursion are no longer related to a zero-dimensional Bessel process. But if we now condition further on the fact that the excursion went deep inside $D(x,1)$, it will have forgotten its initial point and those local times will be again related to a zero-dimensional Bessel process: this is the content of Lemma \ref{lem:independence local times and exit point} and Appendix \ref{sec:proof lemma independence} is dedicated to its proof. Let us mention that the spirit of Lemma \ref{lem:independence local times and exit point} can be tracked back to Lemma 7.4 of \cite{dembo2001}.

As we have just explained, we will use \eqref{eq:prop local times and Bessel process} to transfer some computations from the local times to the zero-dimensional Bessel process. Throughout the paper, we will thus collect lemmas about this process (Lemmas \ref{lem:Bessel first moment}, \ref{lem:Bessel tail asymptotic} and \ref{lem:Bessel Cauchy}) that will be proven in Appendix \ref{sec:Appendix}.
Of course, we will not be able to transfer all the computations to the zero-dimensional Bessel process, for instance when we consider two circles which are not concentric. But we will be able to treat the local times as if they were the local times of a continuous time random walk:
for a continuous time random walk starting at a given vertex $x$ and killed when it hits for the first time a given set $A$, the time spent by the walk in $x$ is exactly an exponential variable which is independent of the hitting point of $A$. We will show that it is also approximatively true for the local times of Brownian motion. This is the content of Section \ref{sec:preliminaries}.

\medskip
We end this introduction with some notations which will be used throughout the paper.

\textbf{Notations:} If $A,B \subset \R^2$, $x, y \in \R^2$, $\eps > 0$, and $i,j \in \Z$, we will denote by:
\begin{enumerate}
\item[--]
$\tau_A := \inf \{ t \geq 0: B_t \in A \}$ the first hitting time of $A$. In particular, $\tau = \tau_{\R^2 \backslash D}$;
\item[--]
$D(x,\eps)$ (resp. $\bar{D}(x,\eps)$, $\partial D(x,\eps)$) the open disc (resp. closed disc, circle) with centre $x$ and radius $\eps$;
\item[--]
$\mathrm{d}(A,B)$ the Euclidean distance between $A$ and $B$. If $A = \{x\}$, we will simply write $\mathrm{d}(x,B)$ instead of $\mathrm{d}(\{x\},B)$;
\item[--]
$R(x,D)$ the conformal radius of $D$ seen from $x$;
\item[--]
$G_D(x,y)$ the Green function in $x,y$:
\begin{equation}
\label{eq:def Green function}
G_D(x,y) := \pi \int_0^\infty p_s(x,y) ds,
\end{equation}
where $p_s(x,y)$ is the transition probability of Brownian motion killed at $\tau$. We recall its behaviour close to the diagonal (see Equation (1.2) of \cite{berestycki} for instance):
\begin{equation}
\label{eq:Green function asymptotic}
G_D(x,y) = - \log \abs{x-y} + \log R(x,D) + u(x,y)
\end{equation}
where $u(x,y) \to 0$ as $y \to x$;
\item[--]
$\brob_r$ the law under which $(R_s,s \geq 0)$ is a zero-dimensional Bessel process starting from $r>0$;
\item[--]
$[ \vert i, j \vert ]$ the set of integers $\{i, \dots, j\}$.
\end{enumerate}

Finally, we will write $C,C', \tilde{C}$, etc, positive constants which may vary from one line to another. We will also write $o(1)$ (resp. $O(1)$) real-valued sequences which go to zero as $\eps \to 0$ (resp. which are bounded). If we want to emphasise that such a sequence may depend on a parameter $\alpha$, we will write $o_\alpha(1)$ (resp. $O_\alpha(1)$).

\section{Preliminaries}\label{sec:preliminaries}

We start off with some preliminary results that will be used throughout the paper. 

\subsection{Green's function}

\begin{lemma}\label{lem:Green function estimate}
For all $x \in D$, $r > \eps >0$ so that $D(x,\eps) \subset D$ and $y \in \partial D(x, \eps)$, we have:
\begin{align}
\EXPECT{y}{L_{x,\eps}(\tau_{\partial D(x,r)})} & = 2 \eps \log \frac{r}{\eps}, \label{eq:Green estimate circle} \\
\EXPECT{y}{L_{x,\eps}(\tau)} & = 2 \eps \left( \log \frac{1}{\eps} + \log R(x,D) + o(1) \right). \label{eq:Green estimate domain}
\end{align}
\end{lemma}

\begin{proof}
We start by proving \eqref{eq:Green estimate circle}. By denoting $p_s(y,z)$ the transition probability of Brownian motion killed at $\tau_{\partial D(x,r)}$, we have:
\[
\EXPECT{y}{L_{x,\eps}(\tau_{\partial D(x,r)})}
=
\int_{\partial D(x, \eps)} dz \int_0^\infty ds ~p_s(y,z)
= \frac{1}{\pi} \int_{\partial D(x, \eps)} dz~ G_{D(x,r)}(y,z).
\]
But the Green function of the disc $D(x,r)$ is equal to (see \cite{lawler2005conformally}, Section 2.4):
\[
G_{D(x,r)}(y,z) = \log \frac{\abs{1-(\bar{y} - \bar{x})(z-x)/r^2}}{\abs{y-z}/r}.
\]
Hence
\[
\EXPECT{y}{L_{x,\eps}(\tau_{\partial D(x,r)})}
= 2 \eps \log \frac{r}{\eps} + \frac{1}{\pi} \int_{\partial D(x,\eps)} \log \frac{\eps}{\abs{y-z}} dz + \frac{1}{\pi} \int_{\partial D(x,\eps)} \log \abs{1-\frac{(\bar{y} - \bar{x})(z-x)}{r^2}} dz.
\]
Because the last two integrals vanish, this gives \eqref{eq:Green estimate circle}.
The proof of \eqref{eq:Green estimate domain} is very similar. The only difference is that we consider the Green function of the general domain $D$. Using the asymptotic \eqref{eq:Green function asymptotic}, we conclude in the same way.
\end{proof}

\subsection{Hitting probabilities}\label{subsec:hitting probabilities}

We now turn to the study of hitting probabilities. The following lemma gives estimates on the probability to hit a small circle before exiting the domain $D$, whereas the next one gives estimates on the probability to hit a small circle before hitting another circle and before exiting the domain $D$.

\begin{lemma}\label{lem:hitting probability}
Let $\eta >0$. For all $\eps >0$ small enough, for all $x \in D$ such that $\mathrm{d}(x, \partial D) > \eta$ and for all $y \in D \backslash D(x, \eps)$, we have:
\begin{equation}\label{eq:hitting probability accurate}
\PROB{y}{ \tau_{ \partial D(x,\eps) } < \tau }
=
\left( 1 + O_\eta \left( \frac{\eps}{\log \eps} \right) \right) G_D(x,y) \left/ \log \left( \frac{R(x,D)}{\eps} \right) \right. .
\end{equation}
\end{lemma}

\begin{proof}
A similar but weaker statement can be found in \cite{bass1994} (Lemma 2.1) and our proof is really close to theirs.
We will take $\eps$ smaller than $\eta/2$ to ensure that the circle $\partial D(x,\eps)$ stays far away from $\partial D$.
If the domain $D$ were the unit disc $\D$ and $x$ the origin, then the probability we are interested in is the probability to hit a small circle before hitting the unit circle. The two circles being concentric, we can use the fact that $(\log \abs{B_t},t \geq 0)$ is a martingale to find that this probability is equal to:
\begin{equation}\label{eq:proof hitting probabilities}
\PROB{y}{\tau_{\partial D(0,\eps)} < \tau_{\partial \D}} = \log \abs{y} / \log \eps.
\end{equation}
In general, we come back to the previous situation by mapping $D$ onto the unit disc $\D$ and $x$ to the origin with a conformal map $f_x$. By conformal invariance of Brownian motion,
\[
\PROB{y}{\tau_{\partial D(x,\eps)} < \tau_D}
=
\PROB{f_x(y)}{\tau_{f_x \left( \partial D(x,\eps) \right)} < \tau_{\D}}.
\]
As $\partial D(x,\eps)$ is far away from the boundary of $D$, the contour $f_x \left( \partial D(x,\eps) \right)$ is included into a narrow annulus
\[
D \left(0, \abs{f_x'(x)} \eps + c \eps^2 \right) \backslash D \left(0, \abs{f_x'(x)} \eps - c \eps^2 \right)
\]
for some $c>0$ depending on $\eta$. In particular, using \eqref{eq:proof hitting probabilities},
\begin{align*}
\PROB{y}{\tau_{\partial D(x,\eps)} < \tau_D}
& \leq
\PROB{f_x(y)}{\tau_{\partial D \left(0, \abs{f_x'(x)} \eps + c \eps^2 \right) } < \tau_{\D}} \\
& =
\frac{\log \abs{f_x(y)}}{\log \left( \abs{f_x'(x)} \eps + c \eps^2 \right) }
=
\frac{\log \abs{f_x(y)}}{\log \left( \abs{f_x'(x)} \eps \right) } \left( 1 + O_\eta \left( \frac{\eps}{\log \eps} \right) \right).
\end{align*}
The lower bound is obtained is a similar manner which yields the stated claim \eqref{eq:hitting probability accurate} noticing that $R(x,D) = 1/ \abs{f'_x(x)}$ and that $G_D(x,y) = - \log \abs{ f_x(y) }$ (see \cite{lawler2005conformally}, Section 2.4).
\end{proof}

\begin{remark}
If $x,y \in D$ are at least at a distance $\eta$ from the boundary of $D$, the quantities
\[
\frac{ G_D(x,y) }{ - \log \abs{x-y} }, R(x,D) \mathrm{~and~} R(y,D)
\]
are bounded away from 0 and from infinity uniformly in $x,y$ (depending on $\eta$). We thus obtain the simpler estimate:
\begin{equation}\label{eq:hitting probabilities rough}
\PROB{y}{ \tau_{ \partial D(x,\eps) } < \tau }, \PROB{x}{ \tau_{ \partial D(y,\eps) } < \tau }
=
\left( 1 + O_\eta \left( \frac{1}{\log \eps} \right) \right) \frac{\log \abs{x-y}}{\log \eps}.
\end{equation}
Depending on the level of accuracy we need, we will use either \eqref{eq:hitting probability accurate} or its rougher version \eqref{eq:hitting probabilities rough}.
\end{remark}

For $x,y \in D$ and $\eps>0$ define
\begin{align*}
p_{xy}^- := \min_{z \in \partial D(x,\eps)} \PROB{z}{\tau_{\partial D(y,\eps)} < \tau}
&\mathrm{~and~}
p_{xy}^+ := \max_{z \in \partial D(x,\eps)} \PROB{z}{\tau_{\partial D(y,\eps)} < \tau},\\
p_{yx}^- := \min_{z \in \partial D(y,\eps)} \PROB{z}{\tau_{\partial D(x,\eps)} < \tau}
&\mathrm{~and~}
p_{yx}^+ := \max_{z \in \partial D(y,\eps)} \PROB{z}{\tau_{\partial D(x,\eps)} < \tau}.
\end{align*}

\begin{lemma}\label{lem:hitting probability 2 circles}
For all $x,y \in D$, $\eps>0$ so that $D(x,\eps)$ and $D(y,\eps)$ are disjoint and included in $D$, for all $z \in D \backslash \left( D(x,\eps) \cup D(y,\eps) \right)$,
\begin{align}
\label{eq:hitting probability 2 circles}
\frac{\PROB{z}{\tau_{\partial D(x,\eps)} < \tau} - p_{yx}^+ \PROB{z}{\tau_{\partial D(y, \eps)} < \tau} }{1 - p_{yx}^+ p_{xy}^-}
& \leq
\PROB{z}{\tau_{\partial D(x,\eps)} < \tau \wedge \tau_{\partial D(y,\eps)} } \\
& ~~~~~~~~~~ \leq
\frac{\PROB{z}{\tau_{\partial D(x,\eps)} < \tau} - p_{yx}^- \PROB{z}{\tau_{\partial D(y, \eps)} < \tau} }{1 - p_{yx}^- p_{xy}^+} \nonumber.
\end{align}
\end{lemma}

\begin{proof}
By Markov property and by definition of $p_{yx}^+$, we have:
\begin{align*}
\PROB{z}{\tau_{\partial D(x,\eps)} < \tau}
& = \PROB{z}{\tau_{\partial D(x,\eps)} < \tau \wedge \tau_{\partial D(y,\eps)}}
+ \PROB{z}{\tau_{\partial D(y,\eps)} < \tau_{\partial D(x,\eps)} < \tau} \\
& \leq \PROB{z}{\tau_{\partial D(x,\eps)} < \tau \wedge \tau_{\partial D(y,\eps)}}
+ \PROB{z}{\tau_{\partial D(y,\eps)} < \tau \wedge \tau_{\partial D(x,\eps)}} p_{yx}^+.
\end{align*}
Similarly,
\[
\PROB{z}{\tau_{\partial D(y,\eps)} < \tau} \geq \PROB{z}{\tau_{\partial D(y,\eps)} < \tau \wedge \tau_{\partial D(x,\eps)}}
+ \PROB{z}{\tau_{\partial D(x,\eps)} < \tau \wedge \tau_{\partial D(y,\eps)}} p_{xy}^-.
\]
Combining those two inequalities yields
\[
\PROB{z}{\tau_{\partial D(x,\eps)} < \tau} - p_{yx}^+ \PROB{z}{\tau_{\partial D(y,\eps)} < \tau}
\leq (1-p_{yx}^+p_{xy}^-) \PROB{z}{\tau_{\partial D(x,\eps)} < \tau \wedge \tau_{\partial D(y,\eps)} }
\]
which is the first inequality stated in \eqref{eq:hitting probability 2 circles}. The other inequality is similar.
\end{proof}

\subsection{Approximation of local times by exponential variables}\label{subsec:approximation local times}

In this subsection, we explain how to approximate the local times $L_{x,\eps}(\tau)$ by exponential variables.
For $x \in \R^2, \eps>0, y \in \partial D(x,\eps)$ and any event $E$, define
\[
H_{x,\eps}^y(E) := \frac{1}{2} \lim_{\substack{z \in D(x,\eps)\\z \to y}} \prob_z^*(E) / \mathrm{d}(z,\partial D(x,\eps))
+ \frac{1}{2} \lim_{\substack{z \notin \bar{D}(x,\eps)\\z \to y}} \prob_z^*(E) / \mathrm{d}(z,\partial D(x,\eps))
\]
where $\prob_z^*$ is the probability measure of Brownian motion starting at $z$ and killed when it hits for the first time the circle $\partial D(x,\eps)$.
For $A \subset \R^2, x \in \R^2$, we will denote $\omega^A(x,d \xi)$ the harmonic measure of $A$ from $x$.

\begin{lemma}\label{lem:approx local times exponential}
Let $x \in \R^2, \eps >0$ and $C \subset \R^2$. Assume that $\mathrm{d}( \partial D(x,\eps), C) > 0$ and that there exists $u > 0$ such that for all $y,y' \in \partial D(x,\eps)$ and $E \subset C$,
\[
(1-u) \omega^C(y,E) \leq \omega^C(y',E) \leq (1+u) \omega^C(y,E).
\]
Then for all $y \in \partial D(x,\eps)$ and $t>0$,
\begin{align*}
(1-u) e^{- \max_{z \in \partial D(x,\eps)} H_{x,\eps}^z(\tau_C < \infty) t} \leq
\PROB{y}{L_{x,\eps}(\tau_C) > t \vert B_{\tau_C}}
\leq (1+u) e^{- \min_{z \in \partial D(x,\eps)} H_{x,\eps}^z(\tau_C < \infty) t}.
\end{align*}
\end{lemma}

\begin{remark}
The previous lemma states that we can approximate $L_{x,\eps}(\tau_C)$ by an exponential variable which is independent of $B_{\tau_C}$. This is similar to the case of random walks on discrete graphs. If we did not condition on $B_{\tau_C}$, it would not have been necessary to add the multiplicative errors $1-u$ and $1+u$. This statement without conditioning is also a consequence of Lemma 2.2 (i) of \cite{bass1994}.
\end{remark}

\begin{proof}
Since the proof is standard, we will be brief.
Take $r>0$ small enough so that the annulus $\bar{D}(x,\eps+r) \backslash D(x, \eps-r)$ does not intersect $C$. Consider the different excursions from $\partial D(x,\eps+r)$ to $\partial D(x,\eps-r)$: denote $\sigma_0^{(2)} :=0$ and for all $i \geq 1$,
\[
\sigma_i^{(1)} := \inf \left\{ t > \sigma_{i-1}^{(2)}: B_t \in \partial D(x,\eps+r) \right\}
\mathrm{~and~}
\sigma_i^{(2)} := \inf \left\{ t > \sigma_i^{(1)}: B_t \in \partial D(x,\eps-r) \right\}.
\]
The number of excursions $N_r := \max \{ i \geq 0: \sigma_i^{(2)} < \tau_C \}$ before $\tau_C$ is related to $L_{x,\eps}(\tau_C)$ by:
\[
L_{x,\eps}(\tau_C) = 4 \lim_{r \to 0} r N_r \quad \quad \prob_{y}-\mathrm{a.s.}
\]
Hence, for any $f : \R^2 \to [0,\infty)$ continuous bounded function, we have by dominated convergence theorem
\[
\EXPECT{y}{\indic{L_{x,\eps}(\tau_C) > t} f \left( B_{\tau_C} \right) }
= \lim_{r \to 0} \EXPECT{y}{\indic{N_r > \floor{t/(4 r)}} f \left( B_{\tau_C} \right) }.
\]
Because
\[
\EXPECT{B_{\sigma^{(2)}_{\floor{t/(4 r)}}}}{f(B_{\tau_C})} \leq (1+u+o_{r \to 0}(1)) \EXPECT{y}{f(B_{\tau_C})} \quad \quad \prob_y-\mathrm{a.s.},
\]
and by a repeated application of Markov property, $ \EXPECT{y}{\indic{N_r > \floor{t/(4 r)}} f \left( B_{\tau_C} \right) } $ is at most
\begin{equation}
\label{eq:proof_lemma2.4}
(1+u+o_{r \to 0}(1)) \EXPECT{y}{f(B_{\tau_C})}
\max_{z \in \partial D(x,\eps+r)} \PROB{z}{ \sigma_1^{(2)} < \sigma_2^{(1)} < \tau_C}^{\floor{\frac{t}{4 r}}}.
\end{equation}
If $z \in \partial D(x,\eps +r)$ is at distance $r$ from $z_\eps \in \partial D(x,\eps)$,
\begin{align*}
1 - \PROB{z}{ \sigma_1^{(2)} < \sigma_2^{(1)} < \tau_C} & = \PROB{z}{\tau_C < \tau_{\partial D(x,\eps-r)}} + (1+o_{r\to 0}(1)) \PROB{z}{\tau_C < \sigma_2^{(1)} \left\vert \sigma_1^{(2)} < \tau_C\right.} \\
= 2r & (1+o_{r\to 0}(1)) \left( \lim_{\substack{z' \notin \bar{D}(x,\eps)\\z' \to z_\eps}} \frac{ \prob_{z'}^*(\tau_C < \infty) }{ \mathrm{d}(z',\partial D(x,\eps)) } + \lim_{\substack{z' \in D(x,\eps)\\z' \to z_\eps}} \frac{ \prob_{z'}^*(\tau_C < \infty) }{ \mathrm{d}(z',\partial D(x,\eps)) } \right) \\
= 4r & (1+o_{r\to 0}(1)) H_{x,\eps}^{z_\eps}(\tau_C < \infty).
\end{align*}
Hence
\[
\max_{z \in \partial D(x,\eps+r)} \PROB{z}{ \sigma_1^{(2)} < \sigma_2^{(1)} < \tau_C} = 1 - 4 r \min_{z \in \partial D(x,\eps)} H^z_{x,\eps}(\tau_C < \infty) + o_{r \to 0}(r).
\]
Coming back to \eqref{eq:proof_lemma2.4}, we have obtained
\[
\EXPECT{y}{\indic{L_{x,\eps}(\tau_C) > t} f \left( B_{\tau_C} \right) }
\leq (1+u) \EXPECT{y}{f(B_{\tau_C})} e^{- \min_{z \in \partial D(x,\eps)} H_{x,\eps}^z(\tau_C < \infty) t}
\]
which is the required upper bound. The lower bound is obtained in a similar way.
\end{proof}

The next lemma explains how to compute the quantities appearing in the previous lemma. Again, particular cases of this can be found in \cite{bass1994} (Lemmas 2.3, 2.5).

\begin{lemma}\label{lem:compute excursion measure}
Let $x \in D, \eps> \delta >0$ and $A \subset D$ such that $D(x,\eps) \subset D \backslash A$ and denote $d$ the distance between $\partial D(x,\eps)$ and $A \cup \partial D$. Assume $d>0$. Let $B$ be either $A$ or $A \cup \partial D(x,\delta)$ and denote
\[
u = \begin{cases}
\frac{\eps}{\eps + d} \mathrm{~if~} B = A, \\
\frac{\eps}{\eps + d} + \frac{\delta}{\eps} \mathrm{~if~} B = A \cup \partial D(x,\delta).
\end{cases}
\]
We have for all $y,y' \in \partial D(x,\eps)$, and $E \subset B \cup \partial D,$
\begin{equation}
\label{eq:lem harmonic measure}
\omega^{B \cup \partial D}(y,E) = \left( 1 + O (u) \right) \omega^{B \cup \partial D}(y',E).
\end{equation}
Moreover, denoting $\tau_{\partial D(x,\eps)}^B := \inf \{ t > \tau_B: B_t \in \partial D(x,\eps) \} $ the first hitting time of $\partial D(x,\eps)$ after $\tau_B$, we have for any $z \in \partial D(x,\eps)$,
\begin{align}
\label{eq:lem compute excursion measure}
\frac{1}{ H_{x,\eps}^z (\tau \wedge \tau_B < \infty) }
& = (1 + O (u)) \max_{y \in \partial D(x,\eps)} \EXPECT{y}{L_{x,\eps}(\tau)} \left( 1 - \int_{\partial D(x,\eps)} \frac{dy}{2\pi \eps} \PROB{y}{\tau_{\partial D(x,\eps)}^B < \tau} \right).
\end{align}
\end{lemma}

\begin{proof}
In this proof, we will consider $\eta >0$ such that $D(x,\eps + \eta) \cap (A \cup \partial D) = \varnothing$.

Let us start by proving \eqref{eq:lem harmonic measure} for $B=A$. Let $y \in \partial D(x,\eps), E \subset A \cup \partial D$. By Markov property applied to the first hitting time of $\partial D(x,\eps+\eta)$, we have
\[
\omega^{A \cup \partial D}(y,E) = \int_{\partial D(x,\eps+\eta)} \omega^{\partial D(x,\eps+\eta)}(y,d \xi) \PROB{\xi}{B_{\tau_A \wedge \tau} \in E}.
\]
But the measure $\omega^{\partial D(x,\eps+\eta)}(y,d \xi)$ is explicit and its density with respect to the Lebesgue measure on the circle $\partial D(x,\eps+\eta)$ is equal to
\[
\frac{1}{2\pi (\eps+\eta)} \frac{(\eps+\eta)^2 - \abs{y-x}^2}{\abs{y-\xi}^2} = \frac{1}{2\pi (\eps+\eta)} \left(1 + O \left( \frac{\eps}{\eps + \eta} \right) \right).
\]
Hence, up to a multiplicative error $1 + O (\eps/(\eps + \eta))$, $\omega^{A \cup \partial D}(y,E)$ is independent of $y \in \partial D(x,\eps)$ which is the content of \eqref{eq:lem harmonic measure} for $B=A$. We now prove it for $B = A \cup \partial D(x,\delta)$. The reasoning is going to be similar. Let $y \in \partial D(x,\eps), E \subset B \cup \partial D$. We only need to treat the case of $E \subset \partial D(x,\delta)$ or $E \subset A \cup \partial D$. We will deal with the first one, as the latter is similar. By Markov property applied to $\tau_A \wedge \tau$, we have
\begin{align*}
\omega^{B \cup \partial D}(y,E) & = \PROB{y}{B_{\tau_{\partial D(x,\delta)}} \in E} - \PROB{y}{B_{\tau_{\partial D(x,\delta)}} \in E, \tau_{\partial D(x,\delta)} > \tau_A \wedge \tau}
\\
& = \omega^{\partial D(x,\delta)}(y,E) - \EXPECT{y}{ \indic{\tau_{\partial D(x,\delta)} > \tau_A \wedge \tau} \omega^{\partial D(x,\delta)}(B_{\tau_A \wedge \tau},E) }.
\end{align*}
Again the measure $\omega^{\partial D(x,\delta)}(y,d \xi)$ is explicit and its density with respect to the Lebesgue measure on the circle $\partial D(x,\delta)$ is equal to
\[
\frac{1}{2\pi\delta} \frac{\abs{y-x}^2 - \delta^2}{\abs{y-\xi}^2} = \frac{1}{2\pi\delta} \left( 1 + O \left( \frac{\delta}{\eps} \right) \right).
\]
Hence, up to a multiplicative error $1 + O (\delta/\eps)$, $\omega^{\partial D(x,\delta)}(y,d \xi)$ is uniform on $\partial D(x,\delta)$ and does not depend on $y$. As $A \cup \partial D$ is even further from $\partial D(x,\delta)$, the same is true with $\omega^{\partial D(x,\delta)}(z,d \xi)$ for any $z \in A \cup \partial D$. To conclude that $\omega^{B \cup \partial D}(y,E)$ does not depend on $y$, we observe that
\begin{align}
\label{eq:proof compute excursion2}
\PROB{y}{\tau_{\partial D(x,\delta)} > \tau_A \wedge \tau}
& =
\PROB{y}{\tau_{\partial D(x,\delta)} > \tau_{\partial D(x,\eps+\eta)}} \\
& ~~~~~~~~~\times \int_{\partial D(x,\eps+\eta)} \omega^{\partial D(x,\eps+\eta)}(y,d\xi) \PROB{\xi}{\tau_{\partial D(x,\delta)} > \tau_A \wedge \tau}. \nonumber
\end{align}
By rotational invariance of Brownian motion, the first term is independent of $y \in \partial D(x,\eps)$. We have already seen that up to a multiplicative error $1 + O(\eps / (\eps + \eta))$, $\omega^{\partial D(x,\eps+\eta)}(y,d\xi)$ is uniform on the circle and thus does not depend on $y$. In the end, it shows that up to a multiplicative error $1 + O(\delta/\eps) + O(\eps/(\eps+\eta))$, $\omega^{B \cup \partial D}(y,E)$ is independent of $y \in \partial D(x,\eps)$ which was required by the claim \eqref{eq:lem harmonic measure} in the case $B = A \cup \partial D(x,\delta)$.
\medskip

We now prove \eqref{eq:lem compute excursion measure}.
We proceed as follows: we bound from below $1 \big/ \min_{z \in \partial D(x,\eps)} H_{x,\eps}^z (\tau \wedge \tau_B < \infty)$ and we show that
\begin{equation}\label{eq:proof compute excursion1}
\min_{z \in \partial D(x,\eps)} H_{x,\eps}^z(\tau \wedge \tau_B < \infty) \geq \left( 1 + O \left( \frac{\eps}{\eps + d} \right) \right) \max_{z \in \partial D(x,\eps)} H_{x,\eps}^z(\tau \wedge \tau_B < \infty)
\end{equation}
which provides a lower bound on $1 \big/ H_{x,\eps}^z (\tau \wedge \tau_B)$ for any $z \in \partial D(x,\eps)$. The upper bound is obtained in a similar way.

Let us start by proving \eqref{eq:proof compute excursion1}. Recall that $\eta >0$ has been chosen such that $D(x,\eps + \eta) \cap (A \cup \partial D) = \varnothing$. Let $z \in D(x,\eps+\eta/2) \backslash D(x,\delta)$. We want to show that the dependence of $z$ on $\PROB{z}{\tau \wedge \tau_B < \tau_{\partial D(x,\eps)}}$ relies almost exclusively on $\abs{z-x}$.
If $z$ is inside $D(x,\eps)$ it is clear: if $B=A$ this probability is equal to zero and if $B = A \cup \partial D(x, \delta)$, it depends only on $\abs{z-x}$ by rotational invariance of Brownian motion. Whereas if $z$ is outside $\bar{D}(x,\eps)$, a similar argument as in \eqref{eq:proof compute excursion2} shows that up to a multiplicative error $1 + O \left( \abs{z-x}/(\eps + \eta) \right)$ this probability depends only on $\abs{z-x}$. It concludes the proof of \eqref{eq:proof compute excursion1}.

We now bound from below $1 \big/ \min_{z \in \partial D(x,\eps)} H_{x,\eps}^z (\tau \wedge \tau_B < \infty)$.
Take a starting point $y \in \partial D(x,\eps)$.
We decompose $L_{x,\eps}(\tau)$ according to the different excursions between $\partial D(x,\eps)$ and $B$. Denote $\sigma_0^{(1)} := 0$ and for all $i \geq 1,$
\[
\sigma_i^{(2)} := \inf \{ t \geq \sigma_{i-1}^{(1)}: B_t \in B \}
\mathrm{~and~}
\sigma_i^{(1)} := \inf \{ t \geq \sigma_i^{(2)}: B_t \in \partial D(x,\eps) \}.
\]
We also denote $N := \sup \{ i \geq 0: \sigma_i^{(1)} < \tau \} $ the number of excursions from $B$ to $\partial D(x,\eps)$ and $L_{x,\eps}^i$ the local time of $\partial D(x,\eps)$ accumulated during the interval of time $[ \sigma_i^{(1)}, \sigma_i^{(2)} ]$. Using the convention $\sigma_{N}^{(2)} := \tau$, we have
\[
L_{x,\eps}(\tau) = \sum_{i=0}^N L_{x,\eps}^i.
\]
By Lemma \ref{lem:approx local times exponential} applied to $C = \partial D \cup B$ and thanks to \eqref{eq:lem harmonic measure},
\[
\EXPECT{y}{L_{x,\eps}(\tau)}
\leq
\sum_{n=0}^\infty \PROB{y}{N=n} (1+n) (1+O(u))^{1+n} \left/ \min_{z \in \partial D(x,\eps)} H_{x,\eps}^z (\tau \wedge \tau_B) \right..
\]
As
\[
\PROB{y}{N \geq n} = \left( (1+O(u)) \int_{\partial D(x,\eps)} \frac{dz}{2\pi \eps} \PROB{z}{N \geq 1} \right)^n,
\]
it leads to
\[
\EXPECT{y}{L_{x,\eps}(\tau)} \leq (1+O(u)) \left( 1 - \int_{\partial D(x,\eps)} \frac{dz}{2\pi \eps} \PROB{z}{N \geq 1} \right)^{-1} \left( \min_{z \in \partial D(x,\eps)} H_{x,\eps}^z (\tau \wedge \tau_B) \right)^{-1}
\]
which is the required lower bound on $1 / \min_{z \in \partial D(x,\eps)} H_{x,\eps}^z (\tau \wedge \tau_B)$.
\end{proof}

In the next sections we will consider $\gamma \in (0,2)$, $A \in \Bc(D)$ and $T$ of the form $T=(b,\infty)$ with $b \in \R$. For $\widetilde{\gamma} > \gamma$, $\eps_0 \in \{e^{-p}, p \geq 1\}$ and $x \in D$, define the good event at $x$:
\begin{equation}\label{eq:def good event}
G_\eps(x,\eps_0) := \left\{ \forall r \in [\eps,\eps_0], \frac{1}{\bar{r}} L_{x,\bar{r}}(\tau) \leq \widetilde{\gamma}^2 \left( \log \bar{r} \right)^2  \right\}
\end{equation}
where for $r >0$, we denote by $\bar{r} = \inf \left( \{e^{-p}, p \geq 1\} \cap [r,\infty) \right)$.
We also define
\begin{equation}\label{eq:def nu tilde}
\widetilde{\nu}_\eps^\gamma (dx, dt) = \nu_\eps^\gamma(dx,dt) \mathbf{1}_{G_\eps(x,\eps_0)} \indic{\abs{x-x_0} > \eps_0, \mathrm{d}(x,\partial D) > \eps_0}.
\end{equation}
To ease computations, we change a bit the definition of good events that we associate to $\mu_\eps^\gamma$:
\begin{equation}\label{eq:def good event exponential}
G'_\eps(x,\eps_0, M) := G_\eps(x,\eps_0) \cap \left\{ \abs{ \sqrt{\frac{1}{\eps} L_{x,\eps}(\tau)} - \gamma \log \frac{1}{\eps} } \leq M \sqrt{\abs{ \log \eps}} \right\},
\end{equation}
and we define
\begin{equation}
\label{eq:def mu tilde}
\widetilde{\mu}_\eps^\gamma (dx) = \mu_\eps^\gamma(dx) \mathbf{1}_{G'_\eps(x,\eps_0,M)} \indic{\abs{x-x_0} > \eps_0, \mathrm{d}(x,\partial D) > \eps_0}.
\end{equation}
This change of good event is purely technical: it will allow us to easily transfer computations linked to $\widetilde{\mu}_\eps^\gamma$ (in Proposition \ref{prop:bdd in L2 exponential}) to computations linked to $\widetilde{\nu}_\eps^\gamma$ (in Proposition \ref{prop:bdd in L2}) rather than repeating arguments which are very similar.

\section{First moment estimates}\label{sec:first moment estimates}

In this section, we give estimates on the first moment of $\nu_\eps^\gamma(A \times T)$ and $\mu_\eps^\gamma(A)$ and we show that adding the good events $G_\eps(x,\eps_0)$ and $G'_\eps(x,\eps_0,M)$ does not change the behaviour of the first moment.

\begin{proposition}\label{prop:first moment estimates}
We have the following estimate
\begin{equation}\label{eq:prop first moment limit}
\lim_{\eps \to 0} \EXPECT{x_0}{\nu_\eps^\gamma (A \times T)}
=
\int_T e^{- \gamma t} \gamma dt \int_A R(x,D)^{\gamma^2/2} G_D(x_0,x) dx.
\end{equation}
Moreover, for all $\eps < \eps_0$,
\begin{equation}\label{eq:prop first moment good event}
0 \leq \EXPECT{x_0}{\nu_\eps^\gamma (A \times T)} - \EXPECT{x_0}{\widetilde{\nu}_\eps^\gamma (A \times T)} \leq p(\eps_0)
\end{equation}
with $p(\eps_0) \to 0$ as $\eps_0 \to 0$. $p(\eps_0)$ may depend on $\gamma,\widetilde{\gamma},T$.
\end{proposition}

\begin{proposition}\label{prop:first moment estimates exponential}
We have the following estimate
\begin{equation}\label{eq:prop first moment exponential limit}
\lim_{\eps \to 0} \EXPECT{x_0}{\mu_\eps^\gamma (A)}
=
\sqrt{2 \pi} \gamma \int_A R(x,D)^{\gamma^2/2} G_D(x_0,x) dx.
\end{equation}
Moreover, for all $\eps < \eps_0$,
\begin{equation}\label{eq:prop first moment exponential good event}
0 \leq \EXPECT{x_0}{\mu_\eps^\gamma (A)} - \EXPECT{x_0}{\widetilde{\mu}_\eps^\gamma (A)} \leq p'(\eps_0,M)
\end{equation}
with $p'(\eps_0,M) \to 0$ as $\eps_0 \to 0$ and $M \to \infty$. $p'(\eps_0,M)$ may depend on $\gamma,\widetilde{\gamma}$.
\end{proposition}

The estimates \eqref{eq:prop first moment limit} and \eqref{eq:prop first moment exponential limit} will be computations on the local times made possible thanks to Section \ref{sec:preliminaries}. To prove \eqref{eq:prop first moment good event} and \eqref{eq:prop first moment exponential good event}, we will be able to transfer all the computations to the zero-dimensional Bessel process. For this reason, we first start by stating the analogue of \eqref{eq:prop first moment good event} and \eqref{eq:prop first moment exponential good event} for this process (recall that we denote $\brob_r$ the law under which $(R_t, t \geq 0)$ is a zero-dimensional Bessel process starting from $r$):

\begin{lemma}\label{lem:Bessel first moment}
Let $\widetilde{\gamma} > \gamma > 0$, $b, \widetilde{b} \in \R$, $r_0,s_0 >0$ and define for all $t >s_0$ the event
\[
E_t(s_0) := \left\{ \forall s \in \N \cap [s_0,t], R_s \leq \widetilde{\gamma} s + \widetilde{b} \right\}.
\]
For all starting point $r \in (0,r_0)$, for all $t > s_0$,
\begin{align}
\label{eq:lem Bessel first moment good event}
\BROB{r}{E_t(s_0) \vert R_t \geq \gamma t+b} & \geq 1 - p(s_0), \\
\label{eq:lem Bessel first moment good event exponential}
\EXPECTB{r}{ e^{\gamma R_t} \mathbf{1}_{E_t(s_0)} } & \geq (1-p(s_0)) \EXPECTB{r}{ e^{\gamma R_t} }
\end{align}
with $p(s_0) \to 0$ as $s_0 \to \infty$. $p(s_0)$ may depend on $\gamma,\widetilde{\gamma},b,\widetilde{b},r_0$.
\end{lemma}

In the previous proposition, the starting point $r$ was required to stay bounded away from infinity. To come back to this situation, we will need the following:

\begin{lemma}\label{lem:Bessel tail asymptotic}
1) Let $a>0$. There exists $C = C(a) >0$ such that for all $t > 0 $, $\lambda \geq at$ and $r \in (1, \lambda/2)$,
\[
\BROB{r}{R_t \geq \lambda} \leq C \sqrt{r} e^{\frac{\lambda r}{t}} \frac{1}{\lambda} e^{-\frac{\lambda^2}{2t}}.
\]
2) Let $\gamma>0$. There exists $C = C(\gamma)>0$ such that for all $t >0$, for all $r \in (1,\gamma t/2)$,
\begin{equation}
\label{eq:lem Bessel expectation exponential}
\EXPECTB{r}{e^{\gamma R_t}} \leq C \sqrt{r} e^{C r} \frac{1}{\sqrt{t}} e^{\gamma^2 t/2}.
\end{equation}
\end{lemma}

The first and second points will be used to prove Propositions \ref{prop:first moment estimates} and \ref{prop:first moment estimates exponential} respectively.
The two previous lemmas will be proven in Appendix \ref{sec:Appendix} and we now prove Propositions \ref{prop:first moment estimates} and \ref{prop:first moment estimates exponential}.

\begin{proof}[Proof of Proposition \ref{prop:first moment estimates}]
We start by proving \eqref{eq:prop first moment limit}. We have
\begin{equation}
\label{eq:proof first moment1}
\EXPECT{x_0}{\nu_\eps^\gamma (A \times T)}
=
\int_A \abs{\log \eps} \eps^{-\gamma^2/2}  \PROB{x_0}{ \sqrt{\frac{1}{\eps} L_{x,\eps}(\tau)} > \gamma \log \frac{1}{\eps} + b} dx
\end{equation}
and we are going to estimate the probability
\begin{equation}
\label{eq:proof first moment2}
\PROB{x_0}{ \sqrt{\frac{1}{\eps} L_{x,\eps}(\tau)} > \gamma \log \frac{1}{\eps} + b}.
\end{equation}
Assume that $\eps > 0$ is small enough so that 
$\gamma \abs{\log \eps} + b >0$
to ensure that the probability we are interested in is not trivial. Take $a \in (\gamma^2/4,1)$. If $x$ is at distance at most $\eps^a$ from $x_0$, we bound from above the probability \eqref{eq:proof first moment2} by 1 and the contribution to the integral \eqref{eq:proof first moment1} of such points is at most $C \eps^{2a - \gamma^2 / 2} \log 1/ \eps$ which goes to zero as $\eps \to 0$.

Let $\eta >0$. We are now going to deal with points $x \in D$ at distance at least $\eps^a$ from $x_0$ and at distance at least $\eta$ from the boundary of the domain $D$. We will then explain how to deal with points close to the boundary. By Markov property, the probability \eqref{eq:proof first moment2} is equal to
\[
\PROB{x_0}{ \tau_{\partial D(x,\eps)} < \tau}
\EXPECT{x_0}{ \PROB{Y}{ \sqrt{\frac{1}{\eps} L_{x,\eps}(\tau)} > \gamma \log \frac{1}{\eps} + b} }
\]
where $Y \in \partial D(x,\eps)$ has the law of $B_{\tau_{\partial D(x,\eps)}}$ starting from $x_0$ and knowing that $\tau_{\partial D(x,\eps)} < \tau$. Take any $y \in \partial D(x,\eps)$. By Lemma \ref{lem:compute excursion measure} we have
\begin{align*}
\min_{z \in \partial D(x,\eps)} H_{x,\eps}^z(\tau < \infty) = \left( 1 + O( \eps / \eta) \right) \max_{z \in \partial D(x,\eps)} H_{x,\eps}^z(\tau < \infty)
= \left( 1 + O( \eps / \eta) \right) / \EXPECT{y}{L_{x,\eps}(\tau)}.
\end{align*}
But Lemma \ref{lem:Green function estimate} gives
\[
\EXPECT{y}{L_{x,\eps}(\tau)} = 2 \eps \left( \log 1 / \eps + \log R(x,D) + o(1) \right).
\]
Hence, with the help of Lemma \ref{lem:approx local times exponential}, starting from $y$, $L_{x,\eps}(\tau)$ is stochastically dominated and stochastically dominates exponential variables with mean equal to $2 \eps \left( \log 1 / \eps + \log R(x,D) + o_\eta(1) \right).$ It implies that
\begin{align*}
& \PROB{y}{ \sqrt{\tfrac{1}{\eps} L_{x,\eps}(\tau)} > \gamma \log \tfrac{1}{\eps} + b} \\
& = \left( 1+o_\eta(1) \right)
\Prob{ 2 \left( \log \tfrac{1}{\eps} + \log R(x,D) + o_\eta(1) \right) \mathrm{Exp}(1) > \left\{ \gamma \log \tfrac{1}{\eps} + b \right\}^2 } \\
& = \left( 1+o_\eta(1) \right) \eps^{\gamma^2 / 2} R(x,D)^{\gamma^2/2} e^{- \gamma b}.
\end{align*}
On the other hand, Lemma \ref{lem:hitting probability} shows that
\begin{equation}\label{eq:proof first momemt estimate hitting prob}
\PROB{x_0}{ \tau_{\partial D(x,\eps)} < \tau} = (1+o_\eta(1)) G_D(x_0,x) / \abs{ \log \eps}.
\end{equation}
Putting things together leads to
\begin{align*}
\int_A \log \frac{1}{\eps} \eps^{-\gamma^2/2} & \PROB{x_0}{ \sqrt{\frac{1}{\eps} L_{x,\eps}(\tau)} > \gamma \log \frac{1}{\eps} + b} \indic{\mathrm{d}(x,\partial D) > \eta} dx\\
& =
(1+o_\eta(1)) e^{- \gamma b} \int_A R(x,D)^{\gamma^2/2} G_D(x_0,x) \indic{\mathrm{d}(x,\partial D) > \eta} dx.
\end{align*}
To conclude the proof of \eqref{eq:prop first moment limit}, it is enough to show that
\begin{equation}\label{eq:proof first moment3}
\limsup_{\eps \to 0} \int_A \abs{\log \eps } \eps^{-\gamma^2/2} \PROB{x_0}{ \sqrt{\frac{1}{\eps} L_{x,\eps}(\tau)} > \gamma \log \frac{1}{\eps} + b} \indic{\mathrm{d}(x,\partial D) \leq \eta} dx = O(\eta).
\end{equation}
Consider a larger domain $\widetilde{D}$ so that $D$ is compactly included in $\widetilde{D}$. Now, all the points $x \in D$ are far away from the boundary of $\widetilde{D}$ and what we did before shows that
\[
\PROB{x_0}{ \sqrt{\frac{1}{\eps} L_{x,\eps}(\tau)} > \gamma \log \frac{1}{\eps} + b}
\leq
\PROB{x_0}{ \sqrt{\frac{1}{\eps} L_{x,\eps}(\tau_{\widetilde{D}})} > \gamma \log \frac{1}{\eps} + b}
\leq C \eps^{\gamma^2/2} / \abs{\log \eps}
\]
which shows \eqref{eq:proof first moment3}.
\medskip

We have finished to prove \eqref{eq:prop first moment limit} and we now turn to the proof of \eqref{eq:prop first moment good event}. Let $\hat{\eps}_0> \eps_0$. As we have just seen, the contribution of $\{ x \in D: \abs{x-x_0} \leq \hat{\eps}_0 \mathrm{~or~} \mathrm{d}(x,\partial D) \leq \hat{\eps}_0 \}$ to $\EXPECT{x_0}{\nu_\eps^\gamma (A \times T)}$ is $O(\hat{\eps}_0)$.
Hence $\EXPECT{x_0}{\nu_\eps^\gamma (A \times T)} - \EXPECT{x_0}{\widetilde{\nu}_\eps^\gamma (A \times T)}$ is equal to
\[
O(\hat{\eps}_0) + 
\int_A \log \frac{1}{\eps} \eps^{-\gamma^2/2}  \PROB{x_0}{ \sqrt{\frac{1}{\eps} L_{x,\eps}(\tau)} > \gamma \log \frac{1}{\eps} + b, G_\eps(x,\eps_0)^c} \indic{\abs{x-x_0} > \hat{\eps}_0, \mathrm{d}(x,\partial D) > \hat{\eps}_0} dx.
\]
Take $x \in D$ such that $\abs{x-x_0} > \hat{\eps}_0$ and $\mathrm{d}(x,\partial D)> \hat{\eps}_0$. Considering a larger domain than $D$ will increase the probability in the above integral. As we want to bound it from above, we can thus assume in the following that $D = D(x,R_0)$ where $R_0$ is the diameter of our original domain. It is convenient because we can now use \eqref{eq:prop local times and Bessel process} which relates the local times to a zero-dimensional Bessel process.

We claim that we can take $M>0$ large enough, depending only on $\hat{\eps}_0$, $R_0$ and $b$, such that
\begin{equation}
\label{eq:proof first moment4}
\PROB{x_0}{ \sqrt{\frac{1}{\eps} L_{x,\eps}(\tau)} > \gamma \log \frac{1}{\eps} + b, L_{x,\hat{\eps}_0}(\tau) \geq M} \leq \hat{\eps}_0 \frac{1}{\abs{\log \eps}} \eps^{\gamma^2/2}.
\end{equation}
Indeed, \eqref{eq:prop local times and Bessel process} and Lemma \ref{lem:Bessel tail asymptotic} imply that there exists $C= C(\hat{\eps}_0,b)>0$ such that if $\eps$ is small enough,
\begin{align*}
& \PROB{x_0}{ \left. \sqrt{\frac{1}{\eps} L_{x,\eps}(\tau)} > \gamma \log \frac{1}{\eps} + b \right\vert L_{x,\hat{\eps}_0}(\tau) = \ell} \leq
C \ell^{1/4} e^{C \sqrt{\ell}} \frac{1}{\abs{\log \eps}} \eps^{\gamma^2/2}, \mathrm{~if~} \ell \leq \frac{\hat{\eps}_0 \gamma^2}{4} \log \left( \frac{\hat{\eps}_0}{\eps} \right)^2.
\end{align*}
As, starting from any point of $\partial D(x,\hat{\eps}_0)$, $L_{x, \hat{\eps}_0}(\tau)$ is an exponential variable (with mean depending on $\hat{\eps}_0$ and $R_0$),
\[
\PROB{x_0}{L_{x,\hat{\eps}_0}(\tau) \geq \frac{\hat{\eps}_0 \gamma^2}{4} \log \left( \frac{\hat{\eps}_0}{\eps} \right)^2 }
\]
goes to zero faster than any polynomial in $\eps$ and also
\[
\EXPECT{x_0}{ L_{x,\hat{\eps}_0}(\tau)^{1/4} e^{C \sqrt{L_{x,\hat{\eps}_0}(\tau)}} \indic{L_{x,\hat{\eps}_0}(\tau) \geq M}}
\]
goes to zero as $M \to \infty$. Putting things together then leads to \eqref{eq:proof first moment4}.

On the other hand, by \eqref{eq:prop local times and Bessel process} and claim \eqref{eq:lem Bessel first moment good event} of Lemma \ref{lem:Bessel first moment} that we use with 
\[
t \leftarrow \log \frac{\hat{\eps}_0}{\eps}, s_0 \leftarrow \log \frac{\hat{\eps}_0}{\eps_0}, r_0 \leftarrow \frac{M}{\hat{\eps}_0}, b \leftarrow b + \gamma \log \frac{1}{\hat{\eps}_0} \mathrm{~and~} \widetilde{b} \leftarrow \gamma \log \frac{1}{\hat{\eps}_0},
\]
there exists $p(\eps_0)$ (which may depend on $\gamma, \widetilde{\gamma}, b, \hat{\eps}_0, M$) such that $p(\eps_0) \to 0$ as $\eps_0 \to 0$ and for all $\eps < \eps_0$,
\begin{align*}
& \PROB{x_0}{ \sqrt{\frac{1}{\eps} L_{x,\eps}(\tau)} > \gamma \log \frac{1}{\eps} + b, G_\eps(x,\eps_0)^c, L_{x,\hat{\eps}_0}(\tau) \leq M}\\
& \leq
\EXPECT{x_0}{ \indic{L_{x,\hat{\eps}_0}(\tau) \leq M} \BROB{\sqrt{L_{x,\hat{\eps}_0}(\tau)/\hat{\eps}_0}}{R_t \geq \gamma t + b + \gamma \log \frac{1}{\hat{\eps}_0} , E_t(s_0)^c }} \\
& \leq p(\eps_0) \EXPECT{x_0}{ \indic{L_{x,\hat{\eps}_0}(\tau) \leq M} \BROB{\sqrt{L_{x,\hat{\eps}_0}(\tau)/\hat{\eps}_0}}{R_t \geq \gamma t + b + \gamma \log \frac{1}{\hat{\eps}_0} }} \\
& = p(\eps_0) \PROB{x_0}{ \sqrt{\frac{1}{\eps} L_{x,\eps}(\tau)} > \gamma \log \frac{1}{\eps} + b, L_{x,\hat{\eps}_0}(\tau) \leq M}.
\end{align*}
With \eqref{eq:proof first moment4} it implies that
\[ \PROB{x_0}{ \sqrt{\frac{1}{\eps} L_{x,\eps}(\tau)} > \sqrt{\frac{g}{2}} \gamma \log \frac{1}{\eps} + b, G_\eps(x,\eps_0)^c}
\leq
q(\eps_0) \frac{1}{\log \eps} \eps^{\gamma^2/2}
\]
for some $q(\eps_0) \to 0$ as $\eps_0 \to 0$ which may depend on $\gamma, \widetilde{\gamma}, b$. It shows that $\EXPECT{x_0}{\nu_\eps^\gamma (A \times T)} - \EXPECT{x_0}{\widetilde{\nu}_\eps^\gamma (A \times T)} \leq C q(\eps_0)$ which finishes the proof of \eqref{eq:prop first moment good event}.

\end{proof}

We now turn to the proof of Proposition \ref{prop:first moment estimates exponential}. As it is similar to what we have just done we will be brief.

\begin{proof}[Proof of Proposition \ref{prop:first moment estimates exponential}]
Take $\eta > 0$ and $x \in D$ at distance at least $\eta$ from the boundary. As we saw before, conditioned on $\tau_{\partial D(x,\eps)} < \tau$, $L_{x,\eps}(\tau)$ is approximated by an exponential variable with mean $2 \eps ( \log 1 / \eps + \log R(x,D) + o_\eta(1))$. Hence, denoting $\theta = \log (R(x,D) / \eps) + o(1)$ and with the change of variable $u = \sqrt{t} - \gamma \sqrt{\theta / 2}$, we have
\begin{align*}
\EXPECT{x_0}{ \left. e^{\gamma \sqrt{\frac{1}{\eps} L_{x,\eps}(\tau)}} \right\vert \tau_{\partial D(x,\eps)} < \tau }
& = \left( 1+o_\eta(1) \right) \int_0^\infty e^{-t} e^{\gamma \sqrt{2 \theta t}} dt \\
& = \left( 1+o_\eta(1) \right) e^{\gamma^2 \theta / 2} \int_0^\infty e^{-(\sqrt{t} - \gamma \sqrt{\theta/2} )^2} dt \\
& = \left( 1+o_\eta(1) \right) \gamma \sqrt{2 \theta} e^{\gamma^2 \theta / 2} \int_{-\gamma \sqrt{\theta / 2}}^\infty e^{-u^2} \left( 1 + \frac{\sqrt{2}u}{\gamma \sqrt{\theta}} \right) du \\
& = (1 + o_\eta(1)) \gamma \sqrt{2 \pi} R(x,D)^{\gamma^2/2} \sqrt{\log (1/\eps)} \eps^{-\gamma^2/2}.
\end{align*}
In particular, the impact of points $x$ such that $\abs{x - x_0} \leq 1/\log (1/\eps)$ is negligible. For points that are far away from $x_0$, we can use \eqref{eq:proof first momemt estimate hitting prob} which then shows that
\[
\sqrt{\log \left( \frac{1}{\eps} \right) } \eps^{\gamma^2/2}
\EXPECT{x_0}{ e^{\gamma \sqrt{\frac{1}{\eps} L_{x,\eps}(\tau)}} }
= (1 + o_\eta(1)) \gamma \sqrt{2 \pi} R(x,D)^{\gamma^2/2} G_D(x_0,x).
\]
By the same reasoning as in the proof of Proposition \ref{prop:first moment estimates}, it concludes the proof of \eqref{eq:prop first moment exponential limit}. We now focus on \eqref{eq:prop first moment exponential good event}.
First of all, we notice that requiring $\sqrt{L_{x,\eps}(\tau)/\eps}$ to belong to the interval
\[
\left[ \gamma \log \frac{1}{\eps} - M \sqrt{\log \frac{1}{\eps}}, \gamma \log \frac{1}{\eps} + M \sqrt{\log \frac{1}{\eps}} \right]
\]
has the consequence of restraining the variable $t$ in the above computations to the interval
\[
\left[ \frac{\gamma^2}{2} \log \frac{1}{\eps} - M \gamma \sqrt{\log \frac{1}{\eps}} + O(1), \frac{\gamma^2}{2} \log \frac{1}{\eps} + M \gamma \sqrt{\log \frac{1}{\eps}} + O(1) \right]
\]
which then restrains the variable $u$ to the interval:
\[
\left[ -\frac{1}{\sqrt{2}} M + o(1), \frac{1}{\sqrt{2}} M + o(1) \right].
\]
Therefore, the integral over $u$ is still equal to $(1+o_{M \to \infty}(1))\sqrt{\pi}$ showing that we can safely forget the event
\[
\left\{ \abs{ \sqrt{\frac{1}{\eps} L_{x,\eps}(\tau)} - \gamma \log \frac{1}{\eps} } \leq M \sqrt{\log \frac{1}{\eps}} \right\}
\]
in the good event $G'_\eps(x,\eps_0,M)$.
To bound from above
\[
\EXPECT{x_0}{ e^{\gamma \sqrt{\frac{1}{\eps} L_{x,\eps}(\tau)}} } - \EXPECT{x_0}{ e^{\gamma \sqrt{\frac{1}{\eps} L_{x,\eps}(\tau)}} \mathbf{1}_{G_\eps(x,\eps_0)} },
\]
we proceed exactly as in the proof of Proposition \ref{prop:first moment estimates}. We notice that this quantity increases with the domain, so we can assume that $D$ is a disc centred at $x$ which allows us to use the link between the local times and the zero-dimensional Bessel process \eqref{eq:prop local times and Bessel process}. We then conclude as in the proof of Proposition \ref{prop:first moment estimates} using claim \eqref{eq:lem Bessel expectation exponential} of Lemma \ref{lem:Bessel tail asymptotic} and claim \eqref{eq:lem Bessel first moment good event exponential} of Lemma \ref{lem:Bessel first moment}.
\end{proof}

\section{Uniform integrability}\label{sec:uniform integrability}

This section is devoted to the following two propositions:

\begin{proposition}\label{prop:bdd in L2}
If $\widetilde{\gamma}$ is close enough to $\gamma$, then
\begin{equation}\label{eq:prop bdd in L2}
\sup_{\eps > 0} \EXPECT{x_0}{ \widetilde{\nu}_\eps^\gamma(A \times T)^2} < \infty.
\end{equation}
\end{proposition}

\begin{proposition}
\label{prop:bdd in L2 exponential}
If $\widetilde{\gamma}$ is close enough to $\gamma$, then
\begin{equation}\label{eq:prop bdd in L2 exponential}
\sup_{\eps > 0} \EXPECT{x_0}{ \widetilde{\mu}_\eps^\gamma(A)^2} < \infty.
\end{equation}
\end{proposition}

We start by proving Proposition \ref{prop:bdd in L2}.

\begin{proof}[Proof of Proposition \ref{prop:bdd in L2}]
The proof will be decomposed in three parts. The first part is short and lay the ground work. In particular, it shows that it is enough to control the probability (written in \eqref{eq:proof bdd in L2 THE PROBA}) that the local times are large in two circles and small in an other circle. The second part describes the joint law of the local times in those three circles whereas the third part computes the probability \eqref{eq:proof bdd in L2 THE PROBA} left in the first part.
To shorten the equations, we will denote $L_{x,\eps} := L_{x,\eps}(\tau)$ the local times up to time $\tau$ in this proof.
\medskip

\emph{Step 1.}
Denoting $A_{\eps_0} = \{ x \in A: \abs{x-x_0} > \eps_0 \mathrm{~and~} \mathrm{d}(x, \partial D) > \eps_0 \}$, by definition of $\widetilde{\nu}_\eps^\gamma$ (see \eqref{eq:def nu tilde}), $\EXPECT{x_0}{ \widetilde{\nu}_\eps^\gamma(A \times T)^2}$ is equal to 
\begin{equation*}
\left( \log \frac{1}{\eps} \right)^2 \eps^{-\gamma^2} \int_{A_{\eps_0} \times A_{\eps_0}} dx dy 
~\PROB{x_0}{ \sqrt{\frac{1}{\eps}L_{x,\eps}}, \sqrt{\frac{1}{\eps}L_{y,\eps}} \geq \gamma \log \frac{1}{\eps} + b, G_\eps(x,\eps_0), G_\eps(y,\eps_0) }.
\end{equation*}
Take $a \in (\gamma^2/4,1)$. The contribution of points $x,y$ such that $\abs{x-y} \leq \eps^a$ goes to zero as $\eps \to 0$. Indeed, this contribution is not larger than
\begin{align*}
C \left( \log \frac{1}{\eps} \right)^2 \eps^{-\gamma^2} \eps^{2a} \int_{A_{\eps_0}} dx
~\PROB{x_0}{ \sqrt{\frac{1}{\eps}L_{x,\eps}} \geq \gamma \log \frac{1}{\eps} + b}
=
C \log \frac{1}{\eps} \eps^{-\gamma^2/2 + 2a} \EXPECT{x_0}{\nu_\eps^\gamma(A_{\eps_0})}
\end{align*}
which goes to zero by the first moment estimate \eqref{eq:prop first moment limit} of Proposition \ref{prop:first moment estimates}. We take now $x,y \in A_{\eps_0}$ such that $\abs{x-y} > \eps^a$. By symmetry, it is enough to bound from above
\begin{equation}
\label{eq:proof bdd in L2 1}
\PROB{x_0}{ \sqrt{\frac{1}{\eps}L_{x,\eps}}, \sqrt{\frac{1}{\eps}L_{y,\eps}} \geq \gamma \log \frac{1}{\eps} + b, G_\eps(x,\eps_0), G_\eps(y,\eps_0), \tau_{\partial D(x,\eps)} < \tau_{\partial D(y,\eps)} }.
\end{equation}
Take $M>0$ large and $R \in (e^{-p}, p \geq 0)$ such that
\begin{equation}
\label{eq:proof bdd in L2 choice R}
\frac{\abs{x-y}}{eM} \leq R < \frac{\abs{x-y}}{M}.
\end{equation}
We ensure that $R < \eps_0$ by taking $M$ large enough, but $M$ will play another role later. The probability in \eqref{eq:proof bdd in L2 1} is at most
\begin{equation}
\label{eq:proof bdd in L2 THE PROBA}
\PROB{x_0}{ \sqrt{\frac{1}{\eps}L_{x,\eps}}, \sqrt{\frac{1}{\eps}L_{y,\eps}} \geq \gamma \log \frac{1}{\eps} + b, \sqrt{\frac{1}{R}L_{x,R}} \leq \widetilde{\gamma} \log \frac{1}{R}, \tau_{\partial D(x,\eps)} < \tau_{\partial D(y,\eps)} }.
\end{equation}
The rest of the proof is dedicated to bound from above this probability. For this purpose, the next paragraph describes the joint law of the local times in those three circles.
\medskip

\emph{Step 2.}
We are going to decompose those three local times according to the different excursions between $\partial D(x,R)$, $\partial D(x,\eps)$ and $\partial D(y,\eps)$. Denote by $A_{R \to x}$ (resp. $A_{R \to y}$) the number of excursions from $\partial D(x,R)$ to $\partial D(x,\eps)$ (resp. to $\partial D(y,\eps)$) before $\tau$, and denote by
\begin{enumerate}
\item[--]
$L_{x,\eps}^n$ the local time of $\partial D(x,\eps)$ during the $n$-th excursion from $\partial D(x,\eps)$ to $\partial D(x, R)$,
\item[--]
$L_{y,\eps}^n$ the local time of $\partial D(y,\eps)$ during the $n$-th excursion from $\partial D(y,\eps)$ to $\partial D(x, R) \cup \partial D$,
\item[--]
$L_{x,R}^n$ the local time of $\partial D(x,R)$ during the $n$-th excursion from $\partial D(x,R)$ to $\partial D(x, \eps) \cup \partial D(y,\eps) \cup \partial D$.
\end{enumerate}
For any $x' \in \partial D(x,\eps)$, we have under $\prob_{x'}$ 
\begin{equation}
\label{eq:proof_bddL2}
L_{x,\eps} = \sum_{n=1}^{1+ A_{R \to x}} L_{x,\eps}^n,
~L_{y,\eps} = \sum_{n=1}^{A_{R \to y}} L_{y,\eps}^n
\mathrm{~and~}
L_{x,R} \succ \sum_{n=1}^{A_{R \to x} + A_{R \to y}} L_{x,R}^n.
\end{equation}
The stochastic domination is not exactly an equality because if the last visited circle before $\tau$ is $\partial D(x,R)$ (it could be $\partial D(y,\eps)$), the number of excursions from $\partial D(x,R)$ to $\partial D(x,\eps) \cup \partial D(y,\eps)$ before $\tau$ is $1+A_{R \to x} + A_{R \to y}$ rather than $A_{R \to x} + A_{R \to y}$.
Lemma \ref{lem:approx local times exponential} allows us to approximate (in the precise sense stated therein) the $L_{x,\eps}^n$'s, $L_{y,\eps}^n$'s, $L_{x,R}^n$'s by exponential variables independent of $A_{R \to x}$ and $A_{R \to y}$. We are going to compute the mean of those exponential variables and the transition probabilities between the different three circles.

Let us start with the study of the transition probabilities.
We will denote
\begin{equation}
\label{eq:proof bdd in L2 def pxy}
p_{xy} := \frac{\log 1/ \abs{x-y}}{\log 1 / \eps}.
\end{equation}
Because $\abs{x-y} > \eps^a$, note that $p_{xy}$ is bounded away from 1: $0 < p_{xy} < 1-a.$
We first remark that by \eqref{eq:hitting probabilities rough} we have
\begin{gather*}
\forall z \in \partial D(x,\eps), \PROB{z}{\tau_{\partial D(y, \eps)} < \tau} = p_{xy} + O (1/ \log \eps),\\
\forall z \in \partial D(y,\eps), \PROB{z}{\tau_{\partial D(x, \eps)} < \tau} = p_{xy} + O (1/ \log \eps),\\
\forall z \in \partial D(x,R), \PROB{z}{\tau_{\partial D(x, \eps)} < \tau} = p_{xy} + O (1/ \log \eps), \PROB{z}{\tau_{\partial D(y, \eps)} < \tau} = p_{xy} + O (1/ \log \eps).
\end{gather*}
Here and in the following of the proof, the $O$'s may depend on $\eps_0, M,a$. By Lemma \ref{lem:hitting probability 2 circles} it thus implies that for all $z \in \partial D(x,R)$,
\begin{gather}
\label{eq:proof bdd in L2 transition proba1}
\PROB{z}{\tau_{\partial D(x, \eps)} < \tau \wedge \tau_{\partial D(y, \eps)}} = \frac{p_{xy} + O (1/ \log \eps) - \left( p_{xy} + O (1/ \log \eps) \right)^2}{1 - \left(p_{xy} + O (1/ \log \eps)\right)^2}
= \frac{p_{xy}}{1+p_{xy}} + O \left( \frac{1}{\log \eps} \right) ,\\
\label{eq:proof bdd in L2 transition proba2}
\PROB{z}{\tau_{\partial D(y, \eps)} < \tau \wedge \tau_{\partial D(x, \eps)}} = \frac{p_{xy}}{1+p_{xy}} + O \left( \frac{1}{\log \eps} \right).
\end{gather}
Of course, for any $z \in \partial D(x, \eps)$,
\begin{equation}
\label{eq:proof bdd in L2 transition proba3}
\PROB{z}{\tau_{\partial D(x, R)} < \tau \wedge \tau_{\partial D(y, \eps)}} = 1
\end{equation}
and \eqref{eq:hitting probabilities rough} implies that for all $z \in \partial D(y,\eps)$
\begin{equation}
\label{eq:proof bdd in L2 transition proba4}
\PROB{z}{\tau_{\partial D(x, R)} < \tau \wedge \tau_{\partial D(x, \eps)}} = \PROB{z}{\tau_{\partial D(x, R)} < \tau} = 1 - O \left( \frac{1}{\log \abs{x-y}} \right).
\end{equation}
To summarise, despite the apparent asymmetry between $x$ and $y$, the circle $\partial D(x,R)$ plays a similar role for $\partial D(x,\eps)$ and $\partial D(y,\eps)$ and the transition probabilities between those three circles are given by \eqref{eq:proof bdd in L2 transition proba1}, \eqref{eq:proof bdd in L2 transition proba2}, \eqref{eq:proof bdd in L2 transition proba3} and \eqref{eq:proof bdd in L2 transition proba4}.

We now move on to the study of the $L_{x,\eps}^n$'s, $L_{y,\eps}^n$'s, $L_{x,R}^n$'s.
Starting from any point of $\partial D(x, \eps)$, $L_{x,\eps}(\tau_{\partial D(x,R)})$ is an exponential variable with mean given by (see \eqref{eq:Green estimate circle} in Lemma \ref{lem:Green function estimate})
\[
2 \eps \log (R/\eps) = 2 (1-p_{xy}) \eps \log \frac{1}{\eps} \left( 1 + O \left( \frac{1}{\log \eps} \right) \right).
\]
Starting from any point of $\partial D(y,\eps)$, Lemma \ref{lem:approx local times exponential} allows us to approximate $L_{y,\eps}(\tau \wedge \tau_{\partial D(x,R)} )$ by an exponential variable with mean equal to (see Lemma \ref{lem:compute excursion measure} applied with $A \leftarrow \partial D(x,R)$ and see \eqref{eq:Green estimate domain} in Lemma \ref{lem:Green function estimate})
\[
\left( 1 + O \left( \tfrac{\eps}{R} \right) \right) \left( 1 - p_{xy} + O \left( \tfrac{1}{\log \eps} \right)  \right) 2 \eps \left( \log \tfrac{1}{\eps} + O(1) \right)
=
2 (1-p_{xy}) \eps \log \tfrac{1}{\eps} \left( 1 + O \left( \tfrac{1}{\log \eps} \right) \right).
\]
Similarly, starting from any point of $\partial D(x,R)$, we can approximate $L_{x,R}(\tau \wedge \tau_{\partial D(x,\eps)} \wedge \tau_{\partial D(x,\eps)})$ by an exponential variable with mean equal to (we apply Lemma \ref{lem:compute excursion measure} with $A \leftarrow \partial D(y,\eps), \eps \leftarrow R, \delta \leftarrow \eps$)
\[
\left( 1 \pm C \tfrac{R}{\abs{x-y}} \right) 
\left( 1 - 2 \tfrac{p_{xy}}{1+p_{xy}} + O \left( \tfrac{1}{\log \abs{x-y}} \right) \right) 2 R \left( \log \tfrac{1}{R} + O(1) \right) = \left( 1 \pm \tfrac{C_1}{M} \right) \tfrac{1-p_{xy}}{1+p_{xy}} 2R \log \tfrac{1}{R}
\]
for some universal constants $C,C_1$. In the following we will denote $\hat{\gamma} = \widetilde{\gamma} / \sqrt{1-C_1/M}$. As we can take $M$ as large as we want, we will be able to require $\hat{\gamma}$ to be as close to $\gamma$ as we want.

Finally, to use Lemma \ref{lem:approx local times exponential} to approximate either $L_{x,\eps}^n$, $L_{y,\eps}^n$ or $L_{x,R}^n$ by exponential variables independent of the exit point, we need to control the error we make in estimating the harmonic measure (what was written $u$ in Lemma \ref{lem:approx local times exponential}). For this, we use \eqref{eq:lem harmonic measure} of Lemma \ref{lem:compute excursion measure} which tells us that the hypothesis of Lemma \ref{lem:approx local times exponential} (used in our three above cases) is satisfied with $u = \Cc/M$ for some $\Cc>0$.
\medskip

\emph{Step 3.} We are now ready to start to compute the probability \eqref{eq:proof bdd in L2 THE PROBA}. We will denote $\Gamma(n,1), \Gamma(n',1)$ independent Gamma variables with shape parameter $n,n'$ and scale parameter 1. We recall the following elementary fact: for any $n,n' \geq 1$ and $t \geq 0$,
\begin{align*}
\Prob{ \Gamma(n,1), \Gamma(n',1) \geq t }
& = e^{-2t} \left( \sum_{i=0}^{n-1} \frac{t^i}{i!} \right) \left( \sum_{j=0}^{n'-1} \frac{t^j}{j!} \right)
= e^{-2t} \sum_{k=0}^{n+n'-2} t^k \sum_{\substack{0 \leq i \leq n-1 \\ 0 \leq j \leq n'-1 \\ i+j = k}} \frac{1}{i!j!} \\
& \leq e^{-2t} \sum_{k=0}^{n+n'-2} t^k \sum_{\substack{i,j \geq 0 \\ i+j = k}} \frac{1}{i!j!}
= e^{-2t} \sum_{k=0}^{n+n'-2} \frac{(2t)^k}{k!}.
\end{align*}
By \eqref{eq:proof_bddL2}, we have:
\begin{align*}
& \PROB{x_0}{ \sqrt{\frac{1}{\eps}L_{x,\eps}}, \sqrt{\frac{1}{\eps}L_{y,\eps}} \geq \gamma \log \frac{1}{\eps} + b, \sqrt{\frac{1}{R}L_{x,R}} \leq \widetilde{\gamma} \log \frac{1}{R}, \tau_{\partial D(x,\eps)} < \tau_{\partial D(y,\eps)} }
\\
& \leq \PROB{x_0}{\tau_{\partial D(x,\eps)} < \tau \wedge \tau_{\partial D(y,\eps)}} \sup_{x' \in \partial D(x,\eps)}
\sum_{\substack{n_x \geq 0 \\ n_{y \geq 1}}} \PROB{x'}{A_{R \to x} = n_x, A_{R \to y} = n_y} \left( 1 + \tfrac{\Cc}{M} \right)^{1+2 n_x + 2 n_y}
\\
& ~~~\times \Prob{ \Gamma(n_x + n_y, 1) \leq \frac{\hat{\gamma}^2}{2} \frac{1+p_{xy}}{1-p_{xy}} \log \frac{1}{R},
\Gamma(n_x + 1, 1), \Gamma(n_y, 1) \geq \frac{\gamma^2}{2} \frac{\log 1/ \eps}{1 - p_{xy}} \left( 1 + O \left( \tfrac{1}{\log \eps} \right) \right) }.
\end{align*}
The term $\left( 1 + \Cc/M \right)^{1+2 n_x + 2 n_y}$ in the above inequality comes from the fact that every time we approximate one of $L_{x,\eps}^n$, $L_{y,\eps}^n$, $L_{x,R}^n$ by an exponential variable independent of the last point of the excursion, we have to pay the multiplicative price $\left( 1 + \Cc/M \right)$. See Lemma \ref{lem:approx local times exponential}.
Now,
\begin{align}
& \PROB{x_0}{ \sqrt{\frac{1}{\eps}L_{x,\eps}}, \sqrt{\frac{1}{\eps}L_{y,\eps}} \geq \gamma \log \frac{1}{\eps} + b, \sqrt{\frac{1}{R}L_{x,R}} \leq \widetilde{\gamma} \log \frac{1}{R}, \tau_{\partial D(x,\eps)} < \tau_{\partial D(y,\eps)} }
\nonumber
\\
& \leq \frac{O(1)}{\log \eps} e^{-\frac{\gamma^2}{1-p_{xy}} \log \frac{1}{\eps}} \sum_{n=1}^\infty \sup_{x' \in \partial D(x,\eps)} \PROB{x'}{A_{R \to x} + A_{R \to y} = n} \left( 1 + \tfrac{\Cc}{M} \right)^{1+2 n}
\nonumber \\
& ~~~\times \Prob{ \Gamma(n, 1) \leq \frac{\hat{\gamma}^2}{2} \frac{1+p_{xy}}{1-p_{xy}} \log \frac{1}{R} }
\sum_{k=0}^{n-1} \frac{1}{k!} \left\{ \gamma^2 \frac{\log 1/ \eps}{1 - p_{xy}} \left( 1 + O \left( \tfrac{1}{\log \eps} \right) \right) \right\}^k
\nonumber \\
& \leq O(1) \frac{p_{xy}}{\log \eps} e^{-\frac{\gamma^2}{1-p_{xy}} \log \frac{1}{\eps}} \sum_{n=1}^\infty \left( \tfrac{2p_{xy}}{1+p_{xy}} + O \left( \tfrac{1}{\log \eps} \right) \right)^{n-1}  \left( 1 + \alpha \right)^{n-1}
\nonumber \\
& ~~~ \times \Prob{ \Gamma(n, 1) \leq \frac{\hat{\gamma}^2}{2} \frac{1+p_{xy}}{1-p_{xy}} \log \frac{1}{\abs{x-y}} + C_2 }
\sum_{k=0}^{n-1} \frac{1}{k!} \left\{ \gamma^2 \frac{\log 1/ \eps}{1 - p_{xy}} + C_3 \right\}^k. \label{eq:proof bdd in L2 3}
\end{align}
Here $\alpha>0$ is of order $1/M$ and can be required to be as small as necessary.
We are going to bound from above the last sum indexed by $n$. We decompose it in three parts that we will denote $S_1, S_2$ and $S_3$ respectively: by denoting
\[
n_1 := \frac{\hat{\gamma}^2}{2} \frac{1+p_{xy}}{1-p_{xy}} \log \frac{1}{\abs{x-y}} + C_2
\mathrm{~and~}
n_2 := \gamma^2 \frac{\log 1/ \eps}{1 - p_{xy}} + C_3,
\]
$S_1$ is the sum over $n = 1 \dots n_1$, $S_2$ corresponds to $n = n_1 + 1 \dots n_2$ and $S_3$ is the remaining $n \geq n_2 +1$.
Let us comment that if $\hat{\gamma}$ is close enough to $\gamma$, we have $n_1 < n_2$ because $(1+p_{xy}) \hat{\gamma}^2 / 2 \leq (1+a) \hat{\gamma}^2 / 2 < \gamma^2.$ In the sum $S_1$, it will be difficult for $L_{x,\eps}(\tau)$ and $L_{y,\eps}(\tau)$ to be large at the same time. In the sum $S_2$ it will be difficult for all the three events to happen and in the sum $S_3$, it will be unlikely for $L_{x,R}(\tau)$ to be small.

Later in the proof, we will use the two following elementary inequalities that we record here for ease of reference:
for all $n \geq 1$ and $\mu \geq 0$, we have:
\begin{align}
\label{eq:upper bound exponential rest}
\mathrm{if~} \mu \leq 1, \sum_{k=n}^\infty \frac{(\mu n)^k}{k!} \leq (\mu e)^n, \\
\label{eq:upper bound exponential sum}
\mathrm{if~} \mu \geq 1, \sum_{k=0}^{n-1} \frac{(\mu n)^k}{k!} \leq e (\mu e)^{n-1}.
\end{align}

$\bullet~ S_1:$ We bound from above the probability appearing in the sum by $1$ and we exchange the order of the summations: we first sum over $k = 0 \dots n_1-1$ and then we sum over $n \geq k+1$. The sum over $n$ being a geometric sum, it is explicit and it leads to
\[
S_1 \leq O(1) \sum_{0 \leq k \leq n_1 -1} \frac{1}{k!} \left( 2 \left( 1 + \alpha \right) \gamma^2 \frac{p_{xy}}{1-p_{xy}^2} \log \frac{1}{\eps} + C'_3 \right)^k.
\]
We now use \eqref{eq:upper bound exponential sum} with
\begin{align*}
\mu & = \left( 2 \left( 1 + \alpha \right) \gamma^2 \frac{p_{xy}}{1-p_{xy}^2} \log \frac{1}{\eps} + C'_3 \right) \Big/ n_1
= 4 \left( 1 + \alpha \right) \frac{\gamma^2}{\hat{\gamma}^2} \frac{1}{(1+p_{xy})^2} \left( 1 + O \left( \tfrac{1}{\log \abs{x-y}} \right) \right) \\
& \geq \left( \frac{2 \gamma}{(1+a) \hat{\gamma}} \right)^2 > 1
\end{align*}
if $\hat{\gamma}$ is close enough to $\gamma$. It gives
\begin{align*}
S_1 & \leq O(1) \left( 4 \left( 1 + \alpha \right) \frac{\gamma^2}{\hat{\gamma}^2} \frac{1}{(1+p_{xy})^2} \left( 1 + O \left( \tfrac{1}{\log \abs{x-y}} \right) \right) e \right)^{n_1} \\
& =
O(1) \exp \left\{ \frac{\hat{\gamma}^2}{2} \frac{1+p_{xy}}{1-p_{xy}} \log \frac{1}{\abs{x-y}} \left( 1 + 2 \log \frac{2 \sqrt{1+\alpha} \gamma}{(1+p_{xy}) \hat{\gamma}} \right) \right\}.
\end{align*}

$\bullet~ S_2:$ For $n \geq n_1 + 1$, we have (see \eqref{eq:upper bound exponential rest})
\begin{equation}
\label{eq:proof bdd in L2 2}
\Prob{\Gamma(n,1) \leq \frac{\hat{\gamma}^2}{2} \frac{1+p_{xy}}{1-p_{xy}} \log \frac{1}{\abs{x-y}} + C_2} \leq \left( \frac{\hat{\gamma}^2}{2} \frac{1+p_{xy}}{1-p_{xy}} \log \frac{O(1)}{\abs{x-y}} \frac{e}{n} \right)^{n-1} e^{- \frac{\hat{\gamma}^2}{2} \frac{1+p_{xy}}{1-p_{xy}} \log \frac{O(1)}{\abs{x-y}}}
\end{equation}
and we also have for $n \leq n_2$ (see \eqref{eq:upper bound exponential sum})
\[
\sum_{k=0}^{n-1} \left\{ \gamma^2 \frac{\log 1/ \eps}{1 - p_{xy}} + C_3 \right\}^k \leq e \left( \gamma^2 \frac{\log O(1)/ \eps}{1 - p_{xy}} \frac{e}{n} \right)^{n-1}.
\]
Recalling that $p_{xy} = \log \abs{x-y} / \log \eps$, these two inequalities show that $S_2$ is at most
\begin{align*}
& O(1) e^{- \frac{\hat{\gamma}^2}{2} \frac{1+p_{xy}}{1-p_{xy}} \log \frac{1}{\abs{x-y}}} \sum_{n=n_1+1}^{n_2} \left( (1+\alpha) \tfrac{\gamma^2 \hat{\gamma}^2}{(1-p_{xy})^2} \log \left( \tfrac{O(1)}{ \abs{x-y}} \right) \left( p_{xy} \log \tfrac{1}{\eps} + O(1) \right) \right)^{n-1} \left( \tfrac{e}{n} \right)^{2(n-1)} \\
& \leq O(1) e^{- \frac{\hat{\gamma}^2}{2} \frac{1+p_{xy}}{1-p_{xy}} \log \frac{1}{\abs{x-y}}} \sum_{n=1}^\infty \left( (1+\alpha) \frac{\gamma^2 \hat{\gamma}^2}{(1-p_{xy})^2} \left( \log \frac{O(1)}{ \abs{x-y}} \right)^2 \right)^{n-1} \left( \frac{e}{n} \right)^{2(n-1)}.
\end{align*}
By Stirling's formula, there exists $C > 0$, such that for all $n \geq 2$, $(e/n)^{2(n-1)} \leq C / ((n-1)!(n-2)!)$. Also, denoting $I_1$ the modified Bessel function of the first kind (see \eqref{eq:def Bessel I_1}) and using its asymptotic form \eqref{eq:I_1 asymptotic}, we notice that
\[
\sum_{n=2}^\infty \frac{1}{(n-1)!(n-2)!} v^n = 2 v^{5/2} I_1 (2 \sqrt{v}) \leq C v^{9/4} e^{2 \sqrt{v}}.
\]
Hence
\[
S_2 \leq O(1) \left( \log \frac{1}{\abs{x-y}} \right)^{9/2} \exp \left\{ \left( - \frac{\hat{\gamma}^2}{2} \frac{1+p_{xy}}{1-p_{xy}} + 2 \sqrt{1+\alpha} \frac{\gamma \hat{\gamma}}{1-p_{xy}} \right) \log \frac{1}{\abs{x-y}} \right\}.
\]

$\bullet~ S_3:$ We again use \eqref{eq:proof bdd in L2 2} and we simply bound
\[
\sum_{k=0}^{n-1} \frac{1}{k!} \left\{ \gamma^2 \frac{\log 1/ \eps}{1 - p_{xy}} + C_3 \right\}^k \leq O(1) e^{\gamma^2 \frac{\log 1/ \eps}{1 - p_{xy}}}
\]
to obtain
\[
S_3 \leq O(1) \exp \left\{ \gamma^2 \frac{\log 1/ \eps}{1 - p_{xy}} - \frac{\hat{\gamma}^2}{2} \frac{1+p_{xy}}{1-p_{xy}} \log \tfrac{1}{\abs{x-y}} \right\} \sum_{n \geq n_2+1} \left( (1+\alpha) \hat{\gamma}^2 \frac{p_{xy}}{1-p_{xy}} \log \tfrac{O(1)}{\abs{x-y}} \frac{e}{n} \right)^{n-1}.
\]
Again by Stirling's formula, $(e/n)^{n-1} \leq C \sqrt{n} / (n-1)!$ and with an inequality of the kind of \eqref{eq:upper bound exponential rest} we have
\begin{align*}
S_3 & \leq O(1) \exp \left\{ \gamma^2 \tfrac{\log 1/ \eps}{1 - p_{xy}} - \frac{\hat{\gamma}^2}{2} \tfrac{1+p_{xy}}{1-p_{xy}} \log \tfrac{1}{\abs{x-y}} \right\} \left( (1+\alpha) \tfrac{\hat{\gamma}^2}{\gamma^2} p_{xy} \left( p_{xy} + O \left( \tfrac{1}{\log \eps} \right) \right) e \right)^{\frac{\gamma^2}{1-p_{xy}} \log \frac{1}{\eps} } \\
& = O(1) \exp \left\{ \left( - \frac{\hat{\gamma}^2}{2} \frac{1+p_{xy}}{1-p_{xy}} + 2 \frac{\gamma^2}{1-p_{xy}} \right) \log \frac{1}{\abs{x-y}} \right\} \\
& \times \exp \left\{  \frac{\gamma^2}{p_{xy}} \left( 2 + \frac{1}{1-p_{xy}} \log  \left( (1+\alpha) \frac{\hat{\gamma}^2}{\gamma^2} p_{xy} \left( p_{xy} + O \left( \tfrac{1}{\log \eps} \right) \right) \right) \right) \log \frac{1}{\abs{x-y}} \right\}.
\end{align*}
But $\sup_{0 < p < 1-a} 1 + (\log p)/(1-p) < 0 $. Hence if $\hat{\gamma}$ is close enough to $\gamma$, $\alpha$ close enough to $0$ and if $\eps$ is small enough
\[
2 + \frac{1}{1-p_{xy}} \log  \left( (1+\alpha) \frac{\hat{\gamma}^2}{\gamma^2} p_{xy} \left( p_{xy} + O \left( \tfrac{1}{\log \eps} \right) \right) \right) < 0
\]
which implies that
\[
S_3 \leq O(1) \exp \left\{ \left( - \frac{\hat{\gamma}^2}{2} \frac{1+p_{xy}}{1-p_{xy}} + 2 \frac{\gamma^2}{1-p_{xy}} \right) \log \frac{1}{\abs{x-y}} \right\}.
\]

Finally, the worst upper bound we have is for $S_2$ and coming back to \eqref{eq:proof bdd in L2 3} we have obtained
\begin{align}
\label{eq:last_eq_24}
& \PROB{x_0}{ \sqrt{\frac{1}{\eps}L_{x,\eps}}, \sqrt{\frac{1}{\eps}L_{y,\eps}} \geq \sqrt{\frac{g}{2}}\gamma \log \frac{1}{\eps} + b, \sqrt{\frac{1}{R}L_{x,R}} \leq \sqrt{\frac{g}{2}} \widetilde{\gamma} \log \frac{1}{R}, \tau_{\partial D(x,\eps)} < \tau_{\partial D(y,\eps)} } \\
& \leq O(1) \frac{1}{(\log \eps)^2} \eps^{\gamma^2} \left( \log \frac{1}{\abs{x-y}} \right)^{11/2}
\exp \left\{ \frac{2 \sqrt{1+\alpha} \gamma \hat{\gamma} - \gamma^2 - \hat{\gamma}^2(1+p_{xy})/2}{1-p_{xy}} \log \frac{1}{\abs{x-y}} \right\}. \nonumber
\end{align}
We can ensure that the coefficient
\[
\frac{2\sqrt{1+\alpha}\gamma \hat{\gamma} - \gamma^2 - \hat{\gamma}^2(1+p_{xy})/2}{1-p_{xy}}
\]
is as close to $\gamma^2/2$ as we want. In particular, it is smaller than $2$ and we have shown that $(\log \eps)^2 \eps^{-\gamma^2/2}$ times 
\[
\PROB{x_0}{ \sqrt{\frac{1}{\eps}L_{x,\eps}}, \sqrt{\frac{1}{\eps}L_{y,\eps}} \geq \gamma \log \frac{1}{\eps} + b, \sqrt{\frac{1}{R}L_{x,R}} \leq \widetilde{\gamma} \log \frac{1}{R}, \tau_{\partial D(x,\eps)} < \tau_{\partial D(y,\eps)} }
\]
is bounded from above by a quantity independent of $\eps$ and integrable. It concludes the proof.
\end{proof}

\begin{remark}\label{rem:proof bdd in L2}
We now do a small remark that will be useful in the proof of Proposition \ref{prop:bdd in L2 exponential}.
If in the inequality \eqref{eq:proof bdd in L2 3} we had a worse estimate with an extra multiplicative error $(1 + O( 1 / \sqrt{\log (1/\eps)} ))^n$ in the sum indexed by $n$, we could have absorbed this error by increasing slightly the value of $\alpha$ and it would not have changed the final result: we would have still obtained an upper bound which is integrable over $x,y$.
\end{remark}

We now prove Proposition \ref{prop:bdd in L2 exponential}.
We are going to see that this is an easy consequence of the proof of Proposition \ref{prop:bdd in L2} and we will use the notations defined therein.

\begin{proof}[Proof of Proposition \ref{prop:bdd in L2 exponential}]
By definition of $\widetilde{\mu}_\eps^\gamma$ (see \eqref{eq:def mu tilde}), $\EXPECT{x_0}{ \widetilde{\mu}_\eps^\gamma(A)^2}$ is equal to
\begin{align*}
\log \left( \frac{1}{\eps} \right) \eps^{\gamma^2} \int_{A_{\eps_0} \times A_{\eps_0}}
\EXPECT{x_0}{e^{\gamma \sqrt{\frac{1}{\eps} L_{x,\eps}} } \mathbf{1}_{G'_\eps(x,\eps_0,M)} e^{\gamma \sqrt{\frac{1}{\eps} L_{y,\eps}} } \mathbf{1}_{G'_\eps(y,\eps_0,M)} }  dx dy.
\end{align*}
As before, if $a \in (\gamma^2/4,1)$, the contribution of points $x,y$ such that $\abs{x-y} \leq \eps^a$ is negligible: it is at most
\begin{align*}
& \log \left( \frac{1}{\eps} \right) \eps^{\gamma^2} \eps^{2a} \exp \left( \gamma^2 \log \frac{1}{\eps} + \sqrt{\frac{2}{g}} M \gamma \sqrt{\log \frac{1}{\eps}} \right) \int_A \EXPECT{x_0}{ e^{\gamma \sqrt{\frac{2}{g \eps} L_{x,\eps}} } } dx
\\
& = \eps^{-\gamma^2/2+2a-o(1)} \EXPECT{x_0}{\nu_\eps^\gamma(A)}
\end{align*}
which converges to zero thanks to the first moment estimate \eqref{eq:prop first moment exponential limit} of Proposition \ref{prop:first moment estimates exponential}. For $x,y \in A_{\eps_0}$ with $\abs{x-y} \geq \eps^a$, we proceed in the exact same way as before. In particular, we have the same description of the joint law of $(L_{x,\eps},L_{x,R},L_{y,\eps})$: starting from any point of $\partial D(x,\eps)$ and conditioning on the event that the number of excursions from $\partial D(x,R)$ to $\partial D(x,\eps)$ is $n$, we can approximate $L_{x,\eps}(\tau)/\eps$ by a Gamma random variable $\Gamma(n+1,2 \theta)$ which is the sum of $n+1$ independent exponential variables with mean $2 \theta$. Here
\[
\theta = \log R + O(1) = (1-p_{xy}) \log \frac{1}{\eps} + O(1).
\]
The only difference with the case treated in Proposition \ref{prop:bdd in L2} is that we consider
\begin{equation}
\label{eq:proof bdd in L2 exponential}
\sqrt{\log \left( \frac{1}{\eps} \right)} \eps^{\gamma^2/2} \Expect{e^{\gamma \sqrt{2\theta \Gamma(n+1,1)}} \indic{ \abs{ \sqrt{2 \theta \Gamma(n+1,1)} - \gamma \log (1/\eps)} \leq M \sqrt{\log (1/\eps)}} }
\end{equation}
rather than
\begin{equation}
\label{eq:proof bdd in L2 exponential2}
\log \frac{1}{\eps} \eps^{-\gamma^2/2} \Prob{ \Gamma(n+1,1) \geq \frac{\gamma^2}{2\theta} \log \frac{1}{\eps} }.
\end{equation}
We are actually going to see that the first quantity can be bounded from above by second one, up to an irrelevant factor. This will allow us to conclude the proof thanks to Proposition \ref{prop:first moment estimates}. With the change of variable $u = \sqrt{t} - \gamma \sqrt{\theta/2}$, we have
\begin{align*}
& \Expect{e^{\gamma \sqrt{2\theta \Gamma(n+1,1)}} \indic{ \abs{ \sqrt{2 \theta \Gamma(n+1,1)} - \gamma \log (1/\eps)} \leq M \sqrt{\log (1/\eps)}} } \\
& = \int_0^\infty e^{\gamma \sqrt{2 \theta t} - t} \frac{t^n}{n!} \indic{ \abs{ \sqrt{t} - \gamma \log (1/\eps)/ \sqrt{2\theta}} \leq M \sqrt{\log (1/\eps)} / \sqrt{2 \theta}} dt \\
& = \frac{(\gamma \sqrt{\theta})^{2n+1}}{n! 2^n} e^{\gamma^2 \theta/2} \int_{\R} e^{-u^2/2} \left( 1 + \frac{\sqrt{2}u}{\gamma \sqrt{\theta}} \right)^{2n+1} \indic{ \abs{u + \gamma \sqrt{\theta/2} -\gamma \log (1/\eps)/ \sqrt{2\theta}} \leq M \sqrt{\log (1/\eps)} / \sqrt{2 \theta}} dt.
\end{align*}
In the range of admissible $u$, we have
\[
1 + \frac{\sqrt{2}u}{\gamma \sqrt{\theta}} = \frac{1}{1-p_{xy}} + O \left( \frac{1}{\sqrt{\log (1/\eps)}} \right)
\]
and we also have
\[
u^2 = \frac{\gamma^2}{2} \frac{p_{xy}}{1-p_{xy}} \log \frac{1}{\abs{x-y}} + O(1) \sqrt{\log \frac{1}{\abs{x-y}}}.
\]
Hence 
\begin{align*}
& \Expect{e^{\gamma \sqrt{2\theta \Gamma(n+1,1)}} \indic{ \abs{ \sqrt{2 \theta \Gamma(n+1,1)} - \gamma \log (1/\eps)} \leq M \sqrt{\log (1/\eps)}} } \\
& = \left( 1 + O \left( \tfrac{1}{\sqrt{\log (1/\eps)}} \right) \right)^n \tfrac{\sqrt{\theta}}{n!} \left( \tfrac{\gamma^2 \theta}{2(1-p_{xy})^2} \right)^n e^{\gamma^2 \theta /2} \exp \left( -\tfrac{\gamma^2p_{xy}}{2(1-p_{xy})} \log \tfrac{1}{\abs{x-y}} + O(1) \sqrt{\log \tfrac{1}{\abs{x-y}}} \right)
\end{align*}
which then implies that the term in \eqref{eq:proof bdd in L2 exponential} is at most
\[
\left( 1 + O \left( \tfrac{1}{\sqrt{\log (1/\eps)}} \right) \right)^n \tfrac{1}{n!}
\log \tfrac{1}{\eps} \exp \left( -\tfrac{\gamma^2}{2(1-p_{xy})} \log \tfrac{1}{\abs{x-y}} + O(1) \sqrt{\log \tfrac{1}{\abs{x-y}}} \right) \left( \tfrac{\gamma^2}{2(1-p_{xy})} \log \tfrac{1}{\eps} \right)^n.
\]
Recalling that the term in \eqref{eq:proof bdd in L2 exponential2} is equal to
\[
\log \tfrac{1}{\eps} \exp \left( -\tfrac{\gamma^2}{2(1-p_{xy})} \log \tfrac{1}{\abs{x-y}} + O(1) \right) \sum_{k=0}^n \frac{1}{k!} \left( \tfrac{\gamma^2}{2(1-p_{xy})} \log \tfrac{1}{\eps} + O(1) \right)^n,
\]
it shows that the term in \eqref{eq:proof bdd in L2 exponential} is at most $(1+O(1/\sqrt{\log(1/\eps)})^n \exp \left( O(1) \sqrt{ \log(1/\abs{x-y})} \right) $ times the term in \eqref{eq:proof bdd in L2 exponential2}. As we mentioned in Remark \ref{rem:proof bdd in L2}, it implies that we obtain the same upper bound as in the proof of Proposition \ref{prop:bdd in L2} with an extra multiplicative error $\exp \left( O(1) \sqrt{ \log(1/\abs{x-y})} \right) $ which is still integrable over $x,y$. It concludes the proof.
\end{proof}

\section{Convergence}\label{sec:Cauchy in L2}

In this section, we will prove the following proposition:

\begin{proposition}\label{prop:Cauchy}
If $\widetilde{\gamma}$ is close enough to $\gamma$, $(\tilde{\nu}_\eps^\gamma(A \times T), \eps >0)$ is a Cauchy sequence in $L^2$ and moreover,
\begin{equation}\label{eq:prop Cauchy decomposition product}
\lim_{\eps \to 0} \EXPECT{x_0}{ \left( \widetilde{\nu}_\eps^\gamma (A \times (b,\infty)) - e^{- \gamma b} \widetilde{\nu}_\eps^\gamma (A \times (0,\infty)) \right)^2 } = 0
\end{equation}
and
\begin{equation}\label{eq:prop Cauchy exponential}
\limsup_{\eps \to 0} \EXPECT{x_0}{ \left( \frac{1}{\sqrt{2 \pi} \gamma} \widetilde{\mu}_\eps^\gamma (A) - \widetilde{\nu}_\eps^\gamma (A \times (0,\infty)) \right)^2 } \leq p(M)
\end{equation}
with $p(M) \to 0$ as $M \to \infty$. $p(M)$ may depend on $\gamma$.
\end{proposition}

As mentioned in the introduction, to use the link between the local times and the zero-dimensional Bessel process \eqref{eq:prop local times and Bessel process}, we will use the following lemma proven in Appendix \ref{sec:proof lemma independence}:

\begin{lemma}\label{lem:independence local times and exit point}
Let $k,k',n \geq 0$ with $k' \geq k+1$ and $n \geq k' - k$. Denote $\eta = e^{-k}$, $\eta' = e^{-k'}$ and for all $i=1 \dots k' - k$, $r_i = \eta e^{-i}$. Consider $0 < r_n < \dots < r_{k'-k+1} < r_{k'-k} = \eta'$ and for $i = 1 \dots n$, $T_i \in \Bc([0,\infty))$. For any $y \in \partial D(0,\eta/e)$, we have
\begin{align}
1 - p(\eta') \leq 
\frac{ \PROB{y}{ \forall i=1 \dots n, L_{0,r_i}(\tau_{\partial D(0,\eta)}) \in T_i \vert \tau_{\partial D(0,\eta')} < \tau_{\partial D(0,\eta)}, B_{\tau_{\partial D(0,\eta)}}} }{ \PROB{y}{ \forall i=1 \dots n, L_{0,r_i}(\tau_{\partial D(0,\eta)}) \in T_i \vert \tau_{\partial D(0,\eta')} < \tau_{\partial D(0,\eta)}} }
\leq 1 + p(\eta')
\end{align}
with $p(\eta') \to 0$ as $\eta' \to 0$. $p(\eta')$ may depend on $\eta$.
\end{lemma}

\begin{remark}\label{rem:indep local times and exit point}
If we had conditioned on $\tau_{\partial D(0,\eta')} < \tau_{\partial D(0,\eta)}, B_{\tau_{\partial D(0,\eta)}}, L_{0, \eta/e}(\tau_{\partial D(0,\eta)})$ rather than on $\tau_{\partial D(0,\eta')} < \tau_{\partial D(0,\eta)}, B_{\tau_{\partial D(0,\eta)}}$, the same conclusion would have held: up to a multiplicative error $1+o_{\eta' \to 0}(1)$, we can forget the conditioning on the exit point $B_{\tau_{\partial D(0,\eta)}}$. This is a direct consequence of Lemma \ref{lem:independence local times and exit point}.
\end{remark}

We now state the result that we will need on the zero-dimensional Bessel process to prove Proposition \ref{prop:Cauchy}. This lemma will be proven in Appendix \ref{sec:Appendix}.

\begin{lemma}\label{lem:Bessel Cauchy}
Let $\widetilde{\gamma} > \gamma > 0$, $b, \widetilde{b} \in \R$, $s_0 \geq 1$ an integer and for all $s \in [|1,s_0|], A_s \in \Bc(\R)$. Let $n \geq 1$ and $(R_s^{(i)}, s \geq 0), i = 1 \dots n$, independent zero-dimensional Bessel processes. Denote for all $s \geq 0$,
\[
R_s := \sqrt{ \sum_{i=1}^n \left( R^{(i)}_s \right)^2 }.
\]
Then the two following limits exist
\begin{align*}
l_1(b) &  := \lim_{t \to \infty}
t e^{\frac{\gamma^2}{2}t} \\
& \times \Brob{ \left. R_t \geq \gamma t+b, \forall s \in [|1,s_0|], R_s \in A_s, \forall s \in [|s_0, t|], R_s \leq \widetilde{\gamma} s + \widetilde{b} \right\vert \forall i=1 \dots n, R_{s_0}^{(i)} >0 }
\end{align*}
and
\begin{align*}
l_2(M) & := \lim_{t \to \infty}
\frac{1}{\sqrt{2 \pi}}
\sqrt{t} e^{-\frac{\gamma^2}{2}t} \\
& \times \ExpectB{ \left. e^{\gamma R_t} \indic{ \abs{R_t - \gamma t} \leq M \sqrt{t}} \indic{\forall s \in [|1,s_0|], R_s \in A_s, \forall s \in [|s_0, t|], R_s \leq \widetilde{\gamma} s + \widetilde{b}} \right\vert \forall i=1 \dots n, R_{s_0}^{(i)} >0 }.
\end{align*}
Moreover,
\begin{equation}\label{eq:lem Bessel Cauchy}
l_1(b) e^{b \gamma} = l_1(0) = (1 + p(M)) l_2(M)
\end{equation}
for some universal sequence $p(M)$ going to 0 as $M \to \infty$.
\end{lemma}

We now prove Proposition \ref{prop:Cauchy}.

\begin{proof}[Proof of Proposition \ref{prop:Cauchy}]
For convenience, if $u,v \in \R$, we will write $u = \pm v$ in this proof when we mean $-v \leq u \leq v$.

We start by proving that $(\tilde{\nu}_\eps^\gamma(A \times T), \eps >0)$ is a Cauchy sequence in $L^2$.
We want to show that
\[
\limsup_{\eps, \delta \to 0} \EXPECT{x_0}{ \left( \tilde{\nu}_\eps^\gamma(A \times T) - \tilde{\nu}_\delta^\gamma(A \times T) \right)^2 } = 0.
\]
By expanding the product, we notice that it is enough to show that
\[
\limsup_{\eps, \delta \to 0} \EXPECT{x_0}{ \tilde{\nu}_\eps^\gamma(A \times T) \tilde{\nu}_\eps^\gamma(A \times T) } - \EXPECT{x_0}{ \tilde{\nu}_\eps^\gamma(A \times T) \tilde{\nu}_\delta^\gamma(A \times T) } \leq 0.
\]
Take $\eps, \delta >0$. In this proof, we will denote $f_{\eps, \delta}(x,y) := \abs{\log \delta} \abs{\log \eps} (\delta\eps)^{-\gamma^2/2}$ times
\[
\PROB{x_0}{\sqrt{\frac{L_{x,\eps}(\tau)}{\eps}} \geq \gamma \log \frac{1}{\eps} + b , \sqrt{\frac{L_{y,\delta}(\tau)}{\delta}} \geq \gamma \log \frac{1}{\delta} + b, G_\eps(x,\eps_0), G_\delta(y,\eps_0) }.
\]
Take $\eta \in \{\bar{r}, r < \eps_0\}$ and denote $(A \times A)_\eta$ the subset of $A \times A$ made of "good points":
\begin{equation}\label{eq:proof def (AxA)_eta}
(A \times A)_\eta := \left\{ (x,y) \in A \times A: D(y, \eta) \cap \bigcup_{r \leq \eps_0} \partial D(x,\bar{r}) = \varnothing \right\}.
\end{equation}
If $(x,y) \in (A \times A)_\eta$, the two sequences of circles $(\partial D(x,\bar{r}), r \leq \eps_0)$ and $(\partial D(y,\bar{r}), r \leq \eps_0)$ will not interact between each other inside $D(y,\eta)$. Since the Lebesgue measure of $(A \times A) \backslash (A \times A)_\eta$ goes to $0$ when $\eta \to 0$, Proposition \ref{prop:bdd in L2}, or more precisely \eqref{eq:last_eq_24}, implies that
\begin{align*}
&\int_{(A \times A) \backslash (A \times A)_\eta} f_{\eps, \eps}(x,y) dx dy \leq \int_{(A \times A) \backslash (A \times A)_\eta} \sup_{\eps} f_{\eps, \eps}(x,y) dx dy = o_{\eta \to 1}(1).
\end{align*}
$\EXPECT{x_0}{ \tilde{\nu}_\eps^\gamma(A \times T) \tilde{\nu}_\eps^\gamma(A \times T) } - \EXPECT{x_0}{ \tilde{\nu}_\eps^\gamma(A \times T) \tilde{\nu}_\delta^\gamma(A \times T) }$ is thus at most
\[
\leq o_{\eta \to 1}(1) + \int_{(A \times A)_\eta} (f_{\eps,\eps}(x,y) - f_{\eps, \delta}(x,y)) dx dy.
\]
Our objective is now to bound from above $f_{\eps, \eps}(x,y) - f_{\eps, \delta}(x,y)$ for $(x,y) \in (A \times A)_\eta$. The two probabilities in $f_{\eps, \eps}(x,y)$ and in $f_{\eps, \delta}(x,y)$ differ only from what is required around $y$. We are thus going to focus around $y$. We consider the excursions from $\partial D(y,\eta/e)$ to $\partial D(y,\eta)$: define $\sigma_0^{(2)} := 0$ and for all $i \geq 1$,
\[
\sigma_i^{(1)} := \inf \left\{ t > \sigma_{i-1}^{(2)}: B_t \in \partial D(y, \eta/e) \right\}
\mathrm{~and~}
\sigma_i^{(2)} := \inf \left\{ t > \sigma_i^{(1)}: B_t \in \partial D(y, \eta) \right\}.
\]
We denote by $N := \max \{i \geq 1: \sigma_i^{(2)} < \tau\}$ the number of excursions. The local times of circles centred at $y$ inside $D(y,\eta/e)$ accumulated during the $i$-th excursion, that we will denote by $(L_{y,r}^{(i)}, r \leq \eta/e)$, depend on the starting point $B_{\sigma_i^{(1)}}$ and on the exit point $B_{\sigma_i^{(2)}}$. But this dependence is weak if the excursion goes deeply inside $D(y,\eta/e)$: this is the content of Lemma \ref{lem:independence local times and exit point}. This is why we consider $\eta' \in (\bar{r}, r < \eps_0)$ much smaller than $\eta$ and for all $i \geq 1$, we consider the random variable $v_i$
\[
v_i =
\begin{cases}
1 \mathrm{~if~} B \left[ \sigma_i^{(1)}, \sigma_i^{(2)} \right] \cap D(y,\eta'/e) \neq \varnothing, \\
0 \mathrm{~otherwise}.
\end{cases}
\]

We claim that there exists $N_\eta$ independent of $x,y,\eps$ such that
\begin{equation}
\label{eq:proof Cauchy claim postponed}
\PROB{x_0}{ \sqrt{ \frac{1}{\eps} L_{x,\eps}(\tau)}, \sqrt{ \frac{1}{\eps} L_{y,\eps}(\tau)} \geq \gamma \log \frac{1}{\eps} + b, N \geq N_\eta}
\leq \eta \frac{1}{(\log \eps)^2} \eps^{\gamma^2}.
\end{equation}
This is in the same spirit as what we did in Section \ref{sec:uniform integrability}. To not interrupt the flow of the proof, we postpone the justification of this claim at the very end of the proof.

It is thus actually enough to bound from above $g_{\eps,\eps}(x,y) - g_{\eps,\delta}(x,y)$ where $g_{\eps, \delta}(x,y)$ is the modification of $f_{\eps, \delta}(x,y)$: $g_{\eps, \delta}(x,y) := \log (1/\delta) \log (1/\eps) (\delta\eps)^{-\gamma^2/2}$ times
\[
\PROB{x_0}{\sqrt{\frac{L_{x,\eps}(\tau)}{\eps}} \geq \gamma \log \frac{1}{\eps} + b , \sqrt{\frac{L_{y,\delta}(\tau)}{\delta}} \geq \gamma \log \frac{1}{\delta} + b, G_\eps(x,\eps_0), G_\delta(y,\eps_0), N < N_\eta }.
\]
We are going to condition on the whole trajectory except the excursions from $\partial D(y,\eta/e)$ to $\partial D(y,\eta)$ which visit $D(y,\eta'/e)$. The only randomness remaining will come from $L_{y,r}(\tau)$ for $r < \eta$. We have:
\begin{align*}
& \frac{1}{\log (\delta) \log (\eps)} (\delta \eps)^{\gamma^2/2} g_{\eps,\delta}(x,y)
= \expect_{x_0} \Bigg[ \indic{\sqrt{\frac{L_{x,\eps}(\tau)}{\eps}} \geq \gamma \log \frac{1}{\eps} + b , G_\eps(x,\eps_0), G_\eta(y,\eps_0), N < N_\eta}  \\
& ~~~~~ \times
\prob_{x_0} \Bigg( \forall r \in [\eta', \eta/e], \sum_{i = 1}^N \indic{v_i = 1} L_{y, \bar{r}}^{(i)} \leq \tilde{\gamma} \bar{r} \left( \log \bar{r} \right)^2 - \sum_{i = 1}^N \indic{v_i = 0} L_{y, \bar{r}}^{(i)} \\
&  \sqrt{\tfrac{L_{y,\delta}(\tau)}{\delta}} \geq \gamma \log \tfrac{1}{\delta} + b, G_\delta(y,\eta'), \Bigg\vert N, B_{\sigma_i^{(1)}}, B_{\sigma_i^{(2)}}, v_i, \left( \indic{v_i = 0} L_{y, \bar{r}}^{(i)}, r \in [\eta',\eta/e] \right), \forall i = 1 \dots N \Bigg)
\Bigg].
\end{align*}
We are interested in this last conditional probability. For a given $i \geq 1$, Lemma \ref{lem:independence local times and exit point} (or more precisely Remark \ref{rem:indep local times and exit point} following Lemma \ref{lem:independence local times and exit point}) tells us that there exists $p(\eta')$ which may depend on $\eta$ and which goes to 0 as $\eta' \to 0$, such that for any sequence $(T_r, r < \eta/e)$ of Borel subsets of $\R$,
\begin{align*}
& \Prob{L_{y,\delta}^{(i)} \in T_\delta, \forall r \in [\delta, \eta/e), L_{y, \bar{r}}^{(i)} \in T_{\bar{r}} \left\vert B_{\sigma_i^{(1)}}, B_{\sigma_i^{(2)}}, v_i = 1, L_{y,\eta/e}^{(i)} \right.} \\
& =
(1 \pm p(\eta')) \Prob{L_{y,\delta}^{(i)} \in T_\delta, \forall r \in [\delta, \eta/e), L_{y, \bar{r}}^{(i)} \in T_{\bar{r}} \left\vert v_i = 1, L_{y,\eta/e}^{(i)} \right.}.
\end{align*}
Now, \eqref{eq:prop local times and Bessel process} tells us that, conditioned on $v_i = 1$ and $L_{y,\eta/e}^{(i)}$,
\[
\left( L_{y,r_s}^{(i)}, r_s = \frac{\eta}{e} e^{-s}, s \geq 0 \right)
\overset{\mathrm{law}}{=} \left( \left( R_s^{(i)} \right)^2, s \geq 0 \right)
\]
where $R^{(i)}$ is a zero-dimensional Bessel process starting from $\sqrt{e L_{y,\eta/e}^{(i)} / \eta}$ conditioned to be positive at time $s_0 = \log \eta / (e \eta')$. By denoting for all $s \geq 0$,
\[
R_s := \sqrt{ \sum_{i=1}^N v_i \left( R^{(i)}_s \right)^2 },
\]
we thus have
\begin{align*}
& (1 \pm p(\eta'))^{-N_\eta} \frac{1}{\log (\delta) \log (\eps)} (\delta \eps)^{\gamma^2/2} g_{\eps,\delta}(x,y) \\
& = \expect_{x_0} \Bigg[ \indic{\sqrt{\frac{L_{x,\eps}(\tau)}{\eps}} \geq \gamma \log \frac{1}{\eps} + b , G_\eps(x,\eps_0), G_\eta(y,\eps_0), N < N_\eta} \prob_{x_0} \Bigg( R_{\log \frac{\eta}{e \delta}} \geq \gamma \log \frac{1}{\delta} + b, \\
& ~~~~~ \forall s \in \left[ \left\vert \log \frac{\eta}{\eta'},  \log \frac{\eta}{e \delta} \right\vert \right], R_s \leq \tilde{\gamma} s + \tilde{\gamma} \log \frac{e}{\eta},
\mathrm{~and~}
\forall s \in \left[ \left\vert 1, \log \frac{\eta}{\delta} \right\vert \right],\\
& ~~~~~ R_s^2 \leq \tilde{\gamma}^2 \left( s + \log \frac{e}{\eta} \right)^2 - \sum_{i = 1}^N \indic{v_i = 0} L_{y, r_s}^{(i)}
\Bigg\vert N, v_i, \left( \indic{v_i = 0} L_{y, \bar{r}}^{(i)} , r \in [\eta', \eta/e] \right), \forall i = 1 \dots N \Bigg)
\Bigg].
\end{align*}
By Lemma \ref{lem:Bessel Cauchy}, the conditional probability times $\log(1/\delta) \delta^{-\gamma^2/2}$ converges as $\delta \to 0$. Hence
\begin{align*}
\limsup_{\eps, \delta \to 0} \left\{ (1+p(\eta'))^{-N_\eta} g_{\eps, \eps}(x,y) - (1-p(\eta'))^{-N_\eta} g_{\eps, \delta}(x,y) \right\} \leq 0.
\end{align*}
$g_{\eps,\delta}(x,y)$ and $N_\eta$ being independent of $\eta'$ it yields
\[
\limsup_{\eps, \delta \to 0} g_{\eps, \eps}(x,y) - g_{\eps, \delta}(x,y) \leq 0.
\]
This concludes the proof of the fact that $(\tilde{\nu}_\eps^\gamma(A \times T), \eps >0)$ is a Cauchy sequence in $L^2$, assuming the veracity of the claim \eqref{eq:proof Cauchy claim postponed}. To prove \eqref{eq:prop Cauchy decomposition product}, we notice that
\begin{align*}
& \EXPECT{x_0}{ \left( \widetilde{\nu}_\eps^\gamma (A \times (b,\infty)) - e^{- \gamma b} \widetilde{\nu}_\eps^\gamma (A \times (0,\infty)) \right)^2 } \\
& = \left\{ \EXPECT{x_0}{ \nu_\eps(A \times (b, \infty))^2} - e^{- \gamma b} \EXPECT{x_0}{ \nu_\eps(A \times (b, \infty)) \widetilde{\nu}_\eps^\gamma (A \times (0,\infty)) } \right\} \\
& + e^{- \gamma b} \left\{ e^{- \gamma b} \EXPECT{x_0}{ \nu_\eps(A \times (0, \infty))^2} - \EXPECT{x_0}{ \nu_\eps(A \times (b, \infty)) \widetilde{\nu}_\eps^\gamma (A \times (0,\infty)) } \right\}
\end{align*}
and we want to show that the two terms in brackets go to zero. We proceed in the exact same way as before. We have to control the difference of two probabilities of events which differ only around one point. Around this point, the local times behave as a zero-dimensional squared Bessel process and we use the first equality of claim \eqref{eq:lem Bessel Cauchy} of Lemma \ref{lem:Bessel Cauchy}. The proof of \eqref{eq:prop Cauchy exponential} is similar with the use of the second equality of claim \eqref{eq:lem Bessel Cauchy} of Lemma \ref{lem:Bessel Cauchy} and a claim similar to \eqref{eq:proof Cauchy claim postponed} (we omit the details).

We now finish the proof by proving \eqref{eq:proof Cauchy claim postponed}. As this is a similar reasoning as the ones we saw in Section \ref{sec:uniform integrability}, we will be brief. Conditioned on $B_{\sigma_i^{(1)}}$, $B_{\sigma_i^{(2)}}$ and on the fact that the $i$-th excursion visits $\partial D(y,\eps)$, the local time $L_{y,\eps}^{(i)}$ of $\partial D(y,\eps)$ accumulated during the $i$-th excursion is approximatively an exponential variable with mean $2 \log (O(1)/\eps)$ (see Lemma \ref{lem:approx local times exponential} for a precise statement). Moreover, conditioned on the starting and ending points of the excursion, the probability for the excursion to visit $\partial D(y,\eps)$ is at most $O(1) / \log (1/\eps)$. Hence, conditioned on the number of excursions $N$, $L_{y,\eps}(\tau)$ can be stochastically dominated by a Gamma random variable with scale parameter $1 / (2 \log (C/\eps))$ and shape parameter having the law of a binomial variable: the sum of $N$ independent Bernouilli random variables with success probability $C/\log(1/\eps)$. By increasing the value of $C$ if necessary, the same is true for $L_{x,\eps}(\tau)$ with $N$ replaced by $N+1$ (we could visit $\partial D(x,\eps)$ before $\partial D(y,\eta/e)$). Hence
\begin{align}
& \PROB{x_0}{ \sqrt{ \frac{1}{\eps} L_{x,\eps}(\tau)}, \sqrt{ \frac{1}{\eps} L_{y,\eps}(\tau)} \geq \gamma \log \frac{1}{\eps} + b, N \geq N_\eta} \nonumber \\
& \leq \sum_{n \geq N_\eta-1} \PROB{x_0}{N = n-1} \left\{ \sum_{k=1}^n \binom{n}{k} \left( \frac{C}{\log (1/\eps)} \right)^k \eps^{\gamma^2/2} \sum_{l=0}^{k-1} \frac{1}{l!} \left( \frac{\gamma^2}{2} \log \frac{C}{\eps} \right)^l \right\}^2 \nonumber \\
& = \left( \frac{1}{\log \eps} \right)^2 \eps^{\gamma^2} \sum_{n \geq N_\eta-1} \PROB{x_0}{N = n-1} \left\{ \sum_{k=1}^n \binom{n}{k} C'^k \sum_{l=0}^{k-1} \frac{1}{l!} \left( \log \frac{1}{\eps} \right)^{l-k+1} \right\}^2. \label{eq:proof Cauchy claim postponed2}
\end{align}
Noticing that the last sum over $l = 0 \cdots k-1$ is at most (by decomposing it into the sums over $l= 0 \dots \floor{k/2}-1$ and $l = \floor{k/2} \dots k-1$ for instance)
\[
k \left( \log \frac{1}{\eps} \right)^{\floor{k/2}-k} + \frac{k}{\floor{k/2}!},
\]
we see that for any $a >0$, there exists $C = C(a)>0$ such that the sum over $k=1 \dots n$ in brackets is at most $C(a) (1+a)^n$. Moreover, $\PROB{x_0}{N=n-1} \leq p^{n-1}$ for some $p <1$ depending on $\eta$. Hence, by considering $a$ small enough so that $(1+a)p < 1$, the sum over $n$ in \eqref{eq:proof Cauchy claim postponed2} is at most $C (p(1+a))^{N_\eta} \leq \eta$ if $N_\eta$ is large enough. This proves the claim \eqref{eq:proof Cauchy claim postponed} and finishes the proof.
\end{proof}

\section{Vague convergence, identification of the limits and properties of \texorpdfstring{$\mu^\gamma$}{TEXT}}\label{sec:proof of theorems}

Our proof of Theorems \ref{th:convergence measures} and \ref{th:identification limits} relies on the following:

\begin{proposition}\label{prop:convergence L1}
The sequences $(\nu_\eps^\gamma(A \times T), \eps >0)$ and $(\mu_\eps^\gamma(A), \eps >0)$ converge in $L^1$. Moreover,
\begin{equation}\label{eq:prop convergence L1 identification limits}
\lim_{\eps \to 0} e^{\gamma b} \nu_\eps^\gamma(A \times (b, \infty))
=
\lim_{\eps \to 0} \nu_\eps^\gamma(A \times (0, \infty))
=
\frac{1}{\sqrt{2\pi}\gamma} \lim_{\eps \to 0} \mu_\eps^\gamma(A)
\quad \quad \prob_{x_0}-\mathrm{a.s.}
\end{equation}
\end{proposition}

The proof is straightforward from Propositions \ref{prop:first moment estimates}, \ref{prop:first moment estimates exponential} and \ref{prop:Cauchy}:

\begin{proof}
By \eqref{eq:prop first moment good event}, for any $\eps, \delta >0$ small enough,
\[
\EXPECT{x_0}{ ~ \abs{\nu_\eps^\gamma(A \times T) - \nu_\delta^\gamma(A \times T)} ~} \leq 2p(\eps_0) + \EXPECT{x_0}{ ~ \abs{\widetilde{\nu}_\eps^\gamma(A \times T) - \widetilde{\nu}_\delta^\gamma(A \times T)} ~}.
\]
Proposition \ref{prop:Cauchy} giving
\[
\limsup_{\eps, \delta \to 0} \EXPECT{x_0}{ ~ \abs{\widetilde{\nu}_\eps^\gamma(A \times T) - \widetilde{\nu}_\delta^\gamma(A \times T)} ~}
\leq
\limsup_{\eps, \delta \to 0} \EXPECT{x_0}{ \left(\widetilde{\nu}_\eps^\gamma(A \times T) - \widetilde{\nu}_\delta^\gamma(A \times T) \right)^2 }^{1/2} = 0,
\]
it implies that
\[
\limsup_{\eps, \delta \to 0} \EXPECT{x_0}{ ~ \abs{\nu_\eps^\gamma(A \times T) - \nu_\delta^\gamma(A \times T)} ~}
\leq 2 p(\eps_0).
\]
Since the left hand side term does not depend on $\eps_0$ and since $p(\eps_0) \to 0$ as $\eps_0 \to 0$, it finally implies that
\[
\limsup_{\eps, \delta \to 0} \EXPECT{x_0}{ ~ \abs{\nu_\eps^\gamma(A \times T) - \nu_\delta^\gamma(A \times T)} ~}
\leq 0
\]
which proves the convergence in $L^1$ of $(\nu_\eps^\gamma(A \times T), \eps >0)$. Using \eqref{eq:prop first moment good event} and \eqref{eq:prop Cauchy decomposition product}, respectively \eqref{eq:prop first moment good event}, \eqref{eq:prop first moment exponential good event} and \eqref{eq:prop Cauchy exponential}, we can show in the same way that
\[
\limsup_{\eps \to 0} \EXPECT{x_0}{ ~ \abs{ e^{\gamma b} \nu_\eps^\gamma(A \times (b,\infty)) - \nu_\eps^\gamma(A \times (0,\infty))} ~}
= 0,
\]
respectively
\[
\limsup_{\eps \to 0} \EXPECT{x_0}{ ~ \abs{ \frac{1}{\sqrt{2\pi}\gamma} \mu_\eps^\gamma(A) - \nu_\eps^\gamma(A \times (0,\infty))} ~}
= 0.
\]
As $(\nu_\eps^\gamma(A \times (0,\infty)), \eps>0)$ converges, this shows the convergence of $(\mu_\eps^\gamma(A), \eps >0)$ and the identification of the limits \eqref{eq:prop convergence L1 identification limits}.
\end{proof}

\begin{proof}[Proof of Theorems \ref{th:convergence measures} and \ref{th:identification limits}]
By Proposition \ref{prop:convergence L1}, $(\nu_\eps^\gamma(A \times T), \eps >0)$ and $(\mu_\eps^\gamma(A), \eps >0)$ converge in probability for any $A \in \Bc(D)$ and $T$ of the form $(b, \infty)$ with $b \in \R$. From this, we obtain the convergence in probability for the vague topology of the random measures $(\nu_\eps^\gamma, \eps >0)$ and $(\mu_\eps^\gamma, \eps >0)$ through classical arguments which can be found in \cite{berestycki2017}, Section 6 (the reasoning therein is for the topology of weak convergence but there is no difficulty to adapt it to the topology of vague convergence). This proves Theorem \ref{th:convergence measures}. We now turn to the proof of Theorem \ref{th:identification limits}. We will abusively denote the measure $A \in \Bc(D) \mapsto \nu^\gamma(A \times (0,\infty))$ by $\nu^\gamma(dx \times (0,\infty))$ and we consider the measure $\bar{\nu}^\gamma$ on $D \times \R$
\[
\bar{\nu}^\gamma(dx, dt) :=
\nu^\gamma(dx \times (0,\infty)) e^{- \gamma t} \gamma dt.
\]
The first equality of \eqref{eq:prop convergence L1 identification limits} shows that $\prob_{x_0}$-a.s. the measures $\nu^\gamma$ and $\bar{\nu}^\gamma$ coincide on the countable $\pi$-system of subsets of $D \times \R$ of the form $[x_1,y_1) \times [x_2,y_2) \times (b,\infty)$ with $x_1,x_2,y_1,y_2,b \in \Q$. This $\pi$-system generates the Borel $\sigma$-field on $D \times \R$ and the measures $\nu^\gamma$ and $\bar{\nu}^\gamma$ are $\prob_{x_0}$-a.s. $\sigma$-finite. Hence $\nu^\gamma = \bar{\nu}^\gamma$ $\prob_{x_0}$-a.s. The same reasoning and the second equality of \eqref{eq:prop convergence L1 identification limits} shows that the measures $\nu^\gamma(dx \times (0,\infty))$ and $\mu^\gamma(dx) / (\sqrt{2\pi} \gamma)$ are $\prob_{x_0}$-a.s. equal. This finishes to prove Theorem \ref{th:identification limits}.
\end{proof}

We now explain how we obtain the links between the work of Bass, Burdzy and Koshnevisan \cite{bass1994} and the one of Aïdékon, Hu and Shi \cite{AidekonHuShi2018} with ours. For this small part, we are going to use their notations that we recall: if $z \in \partial D$ is a nice boundary point, i.e. a point where the boundary $\partial D$ is locally an analytic curve, and $x \in D$,
\begin{enumerate}
\item[-]
$\prob_D^{x_0,z}$ denotes the probability measure of Brownian motion starting from $x_0$ and conditioned to exit $D$ through $z$ (see \cite{AidekonHuShi2018}, Notation 2.1 (i)),
\item[-]
$\Q_{x,D}^{x_0,a}$ is the probability measures of trajectories consisting of, first a Brownian motion starting from $x_0$ and conditioned to hit $x$ before exiting the domain, second a Poisson point process of excursions from $x$, and third a Brownian motion starting from $x$ and killed when it exits for the first time the domain (written $\Q_a^x$ in \cite{bass1994}, p.606),
\item[-]
$\Q_{x,D}^{x_0,z,a}$ is similar to $\Q_{x,D}^{x_0,a}$ except that the last part of the trajectory is a Brownian motion conditioned to exit $D$ through $z$ (see \cite{AidekonHuShi2018}, Proposition 3.5).
\end{enumerate}
We will also denote $C_*[0,\infty)$ the set of all parametrised continuous planar curves $c$ defined on a finite interval $[0,t_c]$ with $t_c \in (0,\infty)$. $C_*[0,\infty)$ is equipped with the Skorokhod topology.
For any event $C \in \Bc(C_*[0,\infty))$, we have
\begin{equation}
\label{eq:relations different probability measures}
\prob_D^{x_0,z} (C) = \lim_{r \to 0} \frac{\PROB{x_0}{C, \abs{B_\tau - z} \leq r}}{\PROB{x_0}{\abs{B_\tau - z} \leq r}}
\mathrm{~and~}
\Q_{x,D}^{x_0,z,a} (C) = \lim_{r \to 0} \frac{\Q_{x,D}^{x_0,a}(C, \abs{B_\tau - z} \leq r)}{\Q_{x,D}^{x_0,a}(\abs{B_\tau - z} \leq r)}.
\end{equation}

The following proposition characterises the measures $\mu^\gamma$. Let us emphasise that we only assume that the domain $D$ is bounded and simply connected here.

\begin{proposition}\label{prop:formula}
For every $\gamma \in (0,2)$ and every non-negative measurable function $f$ on $\R^2 \times C_*[0,\infty)$, we have with $a = \gamma^2/2$,
\begin{equation}
\label{eq:prop formula}
\EXPECT{x_0}{ \int_{\R^2} f(x,B) \mu^\gamma(dx)} = \sqrt{2\pi}\gamma \int_D \expect_{\Q_{x,D}^{x_0,a}} [f(x,B)] R(x,D)^{\gamma^2/2} G_D(x_0,x) dx.
\end{equation}
\end{proposition}

\begin{proof}[Proof of Proposition \ref{prop:formula}]
Proposition 5.1 of \cite{bass1994} states that if the domain $D$ is the unit disc and if the starting point $x_0$ is the origin, for any $x \in D$, the distribution of Brownian motion conditioned on
\[
\left\{ \frac{1}{\eps} L_{x,\eps}(\tau) \geq \gamma^2 (\log \eps)^2 - 2 \abs{\log \eps} \log \abs{\log \eps} \right\}
\]
converges to $\Q_{x,D}^{x_0,a}$ as $\eps \to 0$. No restriction on the value of $\gamma$ is required here and their proof actually works in a general setting of a bounded open simply connected domain and a starting point $x_0 \in D$. Moreover, we notice that if we had conditioned rather on
\begin{equation}
\label{eq:proof th formula 1}
\left\{ \frac{1}{\eps} L_{x,\eps}(\tau) \geq \gamma^2 (\log \eps)^2 \right\},
\end{equation}
we would have obtained the same result: this can be seen in their equation (5.7) where the term $2 \abs{\log \eps} \log \abs{\log \eps}$ is killed by bigger order terms. Hence, we also have: the distribution of Brownian motion starting from $x_0$ and conditioned on \eqref{eq:proof th formula 1} converges to $\Q_{x,D}^{x_0,a}$ as $\eps \to 0$. We can now conclude as in \cite{bass1994}, Theorem 5.2: by standard monotone class argument, it is enough to prove \eqref{eq:prop formula} for $f$ of the form $f(x,B) = \mathbf{1}_A(x) \mathbf{1}_C(B)$ for some $A \in \Bc(D)$ and $C \in \Bc(C_*[0,\infty))$. In that case
\begin{align*}
& \abs{ \EXPECT{x_0}{ \int_{\R^2} f(x,B) \nu^\gamma(dx,(0,\infty)) } - \EXPECT{x_0}{ \int_{\R^2} f(x,B) \nu_\eps^\gamma(dx,(0,\infty)) } } \\
& ~~~~ \leq \EXPECT{x_0}{\mathbf{1}_C(B) \abs{ \nu^\gamma(A,(0,\infty)) - \nu_\eps^\gamma(A,(0,\infty))}}
\leq \EXPECT{x_0}{ \abs{ \nu^\gamma(A,(0,\infty)) - \nu_\eps^\gamma(A,(0,\infty))}}
\end{align*}
which goes to 0 as $\eps \to 0$ by Proposition \ref{prop:convergence L1}. Hence
\begin{align*}
& \EXPECT{x_0}{ \int_{\R^2} f(x,B) \nu^\gamma(dx,(0,\infty)) }
= \lim_{\eps \to 0} \EXPECT{x_0}{ \int_{\R^2} f(x,B) \nu_\eps^\gamma(dx,(0,\infty)) } \\
& ~~~~ = \lim_{\eps \to 0} \abs{ \log \eps } \eps^{-\gamma^2/2} \int_A \PROB{x_0}{C \left\vert \frac{1}{\eps} L_{x,\eps}(\tau) \geq \gamma^2 (\log \eps)^2 \right. } \PROB{x_0}{\frac{1}{\eps} L_{x,\eps}(\tau) \geq \gamma^2 (\log \eps)^2} dx \\
& ~~~~ = \int_A \Q_{x,D}^{x_0,a}(C) R(x,D)^{\gamma^2/2} G_D(x_0,x) dx
\end{align*}
by Proposition \ref{prop:first moment estimates}, \eqref{eq:prop first moment limit}. Recalling that Theorem \ref{th:identification limits} shows that $\mu^\gamma(dx) = \sqrt{2\pi} \gamma \nu^\gamma(dx,(0,\infty))$ $\prob_{x_0}$-a.s., this finishes to prove \eqref{eq:prop formula}.
\end{proof}

From Proposition \ref{prop:formula}, Corollary \ref{cor:mu=beta} is immediate:

\begin{proof}[Proof of Corollary \ref{cor:mu=beta}]
When $D$ is the unit disc and $x_0 = 0$, $R(x,D) = 1- \abs{x}^2$ and $G_D(x_0,x) = - \log \abs{x}$ (see \cite{lawler2005conformally}, Section 2.4). Hence by Proposition \ref{prop:formula} and by \cite{bass1994}, Theorem 5.2, $\mu^\gamma$ and $\sqrt{2\pi}\gamma \beta_a$ both satisfy \eqref{eq:prop formula}. Moreover, these two measures are measurable with respect to the Brownian path.
As noticed in \cite{bass1994}, Remark 5.2 (i), there is only one measure satisfying these two conditions implying that $\prob_{x_0}$-a.s. $\mu^\gamma = \sqrt{2\pi}\gamma \beta_a$.
\end{proof}

The proof of Corollary \ref{cor:mu=Mc} is quite similar:

\begin{proof}[Proof of Corollary \ref{cor:mu=Mc}]
For the same reason as before, it is enough to show that for all non-negative measurable function $f$,
\begin{equation}
\label{eq:proof thm mu=Mc}
\expect_{\prob_D^{x_0,z}} \left[ \int_{\R^2} f(x,B) \mu^\gamma(dx) \right] =
\sqrt{2\pi}\gamma \expect_{\prob_D^{x_0,z}} \left[ \int_{\R^2} f(x,B) \frac{H_D(x_0,z)}{H_D(x,z)} \Mc_\infty^a(dx) \right]
\end{equation}
and we can assume that $f$ is of the form $f(x,B) = \mathbf{1}_A(x) \mathbf{1}_C(B)$ for some $A \in \Bc(D)$ and $C \in \Bc(C_*[0,\infty))$. By \cite{AidekonHuShi2018}, Proposition 5.1, the right hand side term of \eqref{eq:proof thm mu=Mc} is equal to
\[
\sqrt{2\pi}\gamma \int_D \expect_{\Q_{x,D}^{x_0,z,a}} \left[ f(x,B) \right] R(x,D)^a G_D(x_0,x) dx.
\]
On the other hand, by \eqref{eq:relations different probability measures}, \eqref{eq:prop formula} and by dominated convergence theorem, the left hand side term of \eqref{eq:proof thm mu=Mc} is equal to 
\begin{align*}
& \lim_{r \to 0} \left. \expect_{x_0} \left[ \int_{\R^2} f(x,B) \mu^\gamma(dx) \indic{\abs{B_\tau - z} \leq r} \right] \right/ \PROB{x_0}{\abs{B_\tau - z} \leq r} \\
& = \lim_{r \to 0} \sqrt{2\pi}\gamma \int_A \frac{\Q_{x,D}^{x_0,a}(B \in A, \abs{B_\tau - z} \leq r)}{\PROB{x_0}{\abs{B_\tau - z} \leq r}} R(x,D)^{\gamma^2/2} G_D(x_0,x) dx \\
& = \sqrt{2\pi}\gamma \int_A \Q_{x,D}^{x_0,z,a}(B \in A) R(x,D)^{\gamma^2/2} G_D(x_0,x) dx.
\end{align*}
This shows \eqref{eq:proof thm mu=Mc} and concludes the proof.
\end{proof}

We finish this section by proving Corollary \ref{cor:properties mu}. We are basically going to collect properties in \cite{AidekonHuShi2018}.

\begin{proof}[Proof of Corollary \ref{cor:properties mu}]
Let $a = \gamma^2/2$. For any nice point $z \in \partial D$, the first three properties are satisfied by $\Mc_\infty^a$ under $\prob_D^{x_0,z}$ (see \cite{AidekonHuShi2018}, Proposition 5.4 and Theorem 1.1). By Corollary \ref{cor:mu=Mc} it is thus also the case for $\mu^\gamma$. To change the probability measure $\prob_D^{x_0,z}$ to $\prob_{x_0}$, we notice that
\[
\PROB{x_0}{ \cdot }
=
\int_{\partial D} \prob_D^{x_0,z}( \cdot ) H_D(x_0,z) dz.
\]
So if an event $E$ satisfies $\prob_D^{x_0,z}(E) = 0$ for all nice point $z \in \partial D$, $\PROB{x_0}{E} = 0$.
This concludes the proof of \textit{(i)-(iii)}. We now turn to the proof of the claim \textit{(iv)}. It is enough to show that for all non-negative measurable function $f$,
\begin{equation}
\label{eq:proof corollary}
\EXPECT{x_0}{ \int_{D'} f(x,B) \left( \mu^{\gamma,D} \circ \phi^{-1} \right)(dx) }
=
\EXPECT{\phi(x_0)}{ \int_{D'} f(x,B) \abs{\phi'(\phi^{-1}(x))}^{2+\gamma^2/2} \mu^{\gamma,D'}(dx) }.
\end{equation}
To help us to do the change of variable $z' = \phi(z)$ in the computations below, we recall that for any $y \in D$ and $z \in \partial D$, $H_{D'}(\phi(y), \phi(z)) = \abs{\phi'(z)}^{-1} H_D(y,z)$ (see \cite{lawler2005conformally}, Section 5.2).
By Corollary \ref{cor:mu=Mc}, and by the conformal invariance of $\Mc_\infty^a$ (\cite{AidekonHuShi2018}, Proposition 5.3), the left hand side term of \eqref{eq:proof corollary} is then equal to
\begin{align*}
& \int_{\partial D} dz ~H_D(x_0,z) \expect_{\prob_D^{x_0,z}} \left[ \int_{D'} f(x,B) \left( \mu^{\gamma,D} \circ \phi^{-1} \right)(dx) \right] \\
& = \sqrt{2\pi} \gamma \int_{\partial D} dz ~H_D(x_0,z)^2 \expect_{\prob_D^{x_0,z}} \left[ \int_{D'} \frac{f(x,B)}{H_D(\phi^{-1}(x),z)} \left( \Mc_\infty^a \circ \phi^{-1} \right)(dx) \right] \\
& = \sqrt{2\pi} \gamma \int_{\partial D} dz ~H_D(x_0,z)^2 \expect_{\prob_{D'}^{\phi(x_0),\phi(z)}} \left[ \int_{D'} \frac{f(x,B)}{H_D(\phi^{-1}(x),z)} \abs{\phi'(\phi^{-1}(x))}^{2+\gamma^2/2} \Mc_\infty^a(dx) \right] \\
& = \sqrt{2\pi} \gamma \int_{\partial D} dz ~\abs{\phi'(z)} H_{D'}(\phi(x_0),\phi(z)) \\
& ~~~~~~~~~\times \expect_{\prob_{D'}^{\phi(x_0),\phi(z)}} \left[ \int_{D'} f(x,B) \abs{\phi'(\phi^{-1}(x))}^{2+\gamma^2/2} \frac{H_{D'}(\phi(x_0),\phi(z))}{H_{D'}(x,\phi(z))} \Mc_\infty^a(dx) \right] \\
& = \int_{\partial D} dz ~\abs{\phi'(z)} H_{D'}(\phi(x_0),\phi(z)) \expect_{\prob_{D'}^{\phi(x_0),\phi(z)}} \left[ \int_{D'} f(x,B) \abs{\phi'(\phi^{-1}(x))}^{2+\gamma^2/2} \mu^{\gamma,D'}(dx) \right] \\
& = \int_{\partial D'} dz' ~ H_{D'}(\phi(x_0),z') \expect_{\prob_{D'}^{\phi(x_0),z'}} \left[ \int_{D'} f(x,B) \abs{\phi'(\phi^{-1}(x))}^{2+\gamma^2/2} \mu^{\gamma,D'}(dx) \right] \\
& = \expect_{\phi(x_0)} \left[ \int_{D'} f(x,B) \abs{\phi'(\phi^{-1}(x))}^{2+\gamma^2/2} \mu^{\gamma,D'}(dx) \right].
\end{align*}
This shows \eqref{eq:proof corollary}.
\end{proof}

\appendix

\section{Proof of Lemma \ref{lem:independence local times and exit point}}\label{sec:proof lemma independence}

We now prove Lemma \ref{lem:independence local times and exit point}.

\begin{proof}[Proof of Lemma \ref{lem:independence local times and exit point}]
To ease notations, we will denote $\tau_\eta := \tau_{\partial D(0,\eta)}, \tau_{\eta'} := \tau_{\partial D(0,\eta')}$ and for all $i=1 \dots n, L_{r_i} := L_{0,r_i}(\tau_\eta)$. Take $C \in \Bc \left( \partial D(0,\eta) \right)$. We will denote $\mathrm{Leb}(C)$ for the Lebesgue measure on $\partial D(0,\eta)$ of $C$. It is enough to show that
\begin{align}
\label{eq:proof indep1}
& \PROB{y}{B_{\tau_\eta} \in C ,\tau_{\eta'} < \tau_\eta, \forall i=1 \dots n, L_{r_i} \in T_i} \\
& ~~~~ = (1+o_{\eta' \to 0}(1)) \frac{\PROB{y}{B_{\tau_\eta} \in C, \tau_{\eta'} < \tau_\eta }}{\PROB{y}{\tau_{\eta'} < \tau_\eta}} \PROB{y}{\tau_{\eta'} < \tau_\eta, \forall i=1 \dots n, L_{r_i} \in T_i}. \nonumber
\end{align}
Moreover, establishing \eqref{eq:proof indep1} can be reduced to show that
\begin{align}
\label{eq:proof indep2}
& \PROB{y}{B_{\tau_\eta} \in C, \tau_{\eta'} < \tau_\eta, \forall i=1 \dots n, L_{r_i} \in T_i} \\
& ~~~~~~~ = (1+o_{\eta' \to 0}(1)) \frac{\mathrm{Leb}(C)}{2 \pi \eta} ~ \PROB{y}{\tau_{\eta'} < \tau_\eta, \forall i=1 \dots n, L_{r_i} \in T_i}. \nonumber
\end{align}
Indeed, applying \eqref{eq:proof indep2} to $T_i = [0,\infty)$ for all $i$ gives (which was already contained in Lemma \ref{lem:compute excursion measure})
\[
\PROB{y}{B_{\tau_\eta} \in C, \tau_{\eta'} < \tau_\eta }
= (1+o_{\eta' \to 0}(1)) \PROB{y}{\tau_{\eta'} < \tau_\eta} \frac{\mathrm{Leb}(C)}{2 \pi \eta},
\]
which combined with \eqref{eq:proof indep2} leads to \eqref{eq:proof indep1}.
Finally, after reformulation of \eqref{eq:proof indep2}, to finish the proof we only need to prove that
\begin{equation}
\label{eq:proof prop local times indep exit point}
\PROB{y}{B_{\tau_\eta} \in C \vert \tau_{\eta'} < \tau_\eta, \forall i=1 \dots n, L_{r_i} \in T_i}
= (1+o_{\eta' \to 0}(1)) \frac{\mathrm{Leb}(C)}{2 \pi \eta}.
\end{equation}

The skew-product decomposition of Brownian motion (see \cite{kallenberg2002foundations}, Corollary 16.7 for instance) tells us that we can write
\[
(B_t,t \geq 0) \overset{\mathrm{(d)}}{=} (\abs{B_t} e^{i \theta_t}, t \geq 0)
\mathrm{~with~}
(\theta_t, t \geq 0) = (w_{\sigma_t}, t \geq 0)
\]
where $(w_t, t \geq 0)$ is a one-dimensional Brownian motion independent of the radial part $(\abs{B_t}, t \geq 0)$ and $(\sigma_t, t \geq 0)$ is a time-change that is adapted to the filtration generated by $(\abs{B_t}, t \geq 0)$:
\[
\sigma_t = \int_0^t \frac{1}{\abs{B_s}^2} ds.
\]
In particular, under $\prob_y$, we have the following equality in law
\begin{equation}
\label{eq:proof prop local times indep exit point2}
\left( \tau_\eta, \abs{B_t}, t < \tau_\eta, B_{\tau_\eta} \right)
\overset{\mathrm{(d)}}{=}
\left( \tau_\eta, \abs{B_t}, t < \tau_\eta, \eta e^{i \theta_0 + i \varsigma \Nc} \right)
\end{equation}
where $\theta_0$ is the argument of $y$, $\Nc$ is a standard normal random variable independent of the radial part $(\abs{B_t}, t \geq 0)$ and
\[
\varsigma = \sqrt{\int_0^{\tau_\eta} \frac{1}{\abs{B_s}^2} ds}.
\]

We now investigate a bit the distribution of $e^{i \theta_0 + it \Nc}$ for some $t>0$. More precisely, we want to give a quantitative description of the fact that if $t$ is large, the previous distribution should approximate the uniform distribution on the unit disc. Using the probability density function of $\Nc$ and then using Poisson summation formula, we find that the probability density function $f_t(\theta)$ of $e^{i \theta_0 + it \Nc}$ at a given angle $\theta$ is given by
\begin{align*}
f_t(\theta) & = \frac{1}{\sqrt{2 \pi t}} \sum_{n \in \Z} e^{-(\theta - \theta_0 + 2\pi n)^2/(2t)} = \frac{1}{2 \pi} \sum_{p \in \Z} e^{ip(\theta - \theta_0)} e^{- p^2t/2} \\
& = \frac{1}{2 \pi} \left( 1 + 2 \sum_{p=1}^\infty \cos(p (\theta - \theta_0)) e^{- p^2 t / 2} \right).
\end{align*}
In particular, we can control the error in the approximation mentioned above by: for all $\theta \in [0,2 \pi]$,
\[
\abs{f_t(\theta) - \frac{1}{2\pi} } \leq \frac{1}{\pi} \sum_{p=1}^\infty e^{-p^2t/2} \leq C_1 \max \left( 1, \frac{1}{\sqrt{t}} \right) e^{-t/2}
\]
for some universal constant $C_1>0$.

We now come back to the objective \eqref{eq:proof prop local times indep exit point}. Using the identity \eqref{eq:proof prop local times indep exit point2} and because the local times $L_{r_i}$ are measurable with respect to the radial part of Brownian motion, we have by triangle inequality
\begin{align*}
& \abs{ \PROB{y}{B_{\tau_\eta} \in C \vert \tau_{\eta'} < \tau_\eta, \forall i=1 \dots n, L_{r_i} \in T_i} - \frac{\mathrm{Leb}(C)}{2 \pi \eta} } \\
& \leq
\EXPECT{y}{ \left. \int_0^{2\pi} \abs{ f_\varsigma(\theta) - \frac{1}{2 \pi} } \indic{ \eta e^{i \theta} \in C} d \theta \right\vert \tau_{\eta'} < \tau_\eta, \forall i =1 \dots n, L_{r_i} \in T_i } \\
& \leq
C_1 \frac{\mathrm{Leb}(C)}{\eta} \EXPECT{y}{ \left. \max \left( 1, \frac{1}{\sqrt{\varsigma}} \right) e^{-\varsigma/2} \right\vert \tau_{\eta'} < \tau_\eta, \forall i =1 \dots n, L_{r_i} \in T_i } \\
& \leq
C_1 \frac{\mathrm{Leb}(C)}{\eta} \EXPECT{y}{ \left. \max \left( 1, \frac{1}{\sqrt{\varsigma'}} \right) e^{-\varsigma'/2} \right\vert \tau_{\eta'} < \tau_\eta, \forall i =1 \dots n, L_{r_i} \in T_i }
\end{align*}
where
\[
\varsigma' := \sqrt{\int_{\tau_{r_n}}^{\tau_\eta} \frac{1}{\abs{B_s}^2} ds}.
\]
To conclude the proof, we want to show that
\[
\EXPECT{y}{ \left. \max \left( 1, \frac{1}{\sqrt{\varsigma'}} \right) e^{-\varsigma'/2} \right\vert \tau_{\eta'} < \tau_\eta, \forall i =1 \dots n, L_{r_i} \in T_i } = o_{\eta' \to 0}(1).
\]
By conditioning on the trajectory up to $\tau_{\eta'}$, it is enough to show that for any $T_i' \in \Bc([0,\infty)), i = 1 \dots n$, for any $z \in \partial D(0,\eta')$,
\begin{equation}
\label{eq:proof indep3}
\EXPECT{z}{ \left. \max \left( 1, \frac{1}{\sqrt{\varsigma'}} \right) e^{-\varsigma'/2} \right\vert \forall i =1 \dots n, L_{r_i} \in T_i' } = o_{\eta' \to 0}(1).
\end{equation}
In the following, we fix such $T_i'$ and such a $z$.

Consider the sequence of stopping times defined by: $\sigma_0^{(2)} :=0$ and for all $i = 1 \dots k' + k$,
\[
\sigma_i^{(1)} := \inf \left\{ t > \sigma_{i-1}^{(2)}: \abs{B_t} = \eta' e^{i-1/2} \right\}
\mathrm{~and~}
\sigma_i^{(2)} := \inf \left\{ t > \sigma_i^{(1)}: \abs{B_t} \in \{ \eta' e^{i}, \eta' e^{i-1} \} \right\}.
\]
We only keep track of the portions of trajectories during the intervals $\left[ \sigma_i^{(1)}, \sigma_i^{(2)} \right]$ by bounding from below $\varsigma'$ by
\[
(\varsigma')^2 \geq \sum_{i=1}^{k'-k} \frac{\sigma_i^{(2)} - \sigma_i^{(1)}}{(\eta' e^i)^2} =: L.
\]
By Markov property, conditioning on $\{ \forall i =1 \dots n, L_{r_i} \in T_i' \}$ impacts the variables $\sigma_i^{(2)} - \sigma_i^{(1)}$ only through $\abs{B_{\sigma_i^{(2)}}}$. But one has that there exists $c >0$ such that for all $i=1 \dots k' - k$,
\[
\EXPECT{z}{ \left. \frac{1}{ \sigma_i^{(2)} - \sigma_i^{(1)} } \right\vert \abs{B_{\sigma_i^{(2)}}}  } \leq \frac{c}{ (\eta' e^i)^2 }.
\]
Then for all $S>0$ we have by Markov's inequality and then by Jensen's inequality applied to $u \mapsto 1/u$:
\begin{align*}
\PROB{z}{L < S \vert \forall i =1 \dots n, L_{r_i} \in T_i'}
& \leq S \EXPECT{z}{\left. \frac{1}{L} \right\vert \forall i =1 \dots n, L_{r_i} \in T_i'}
\\
& \leq \frac{S}{(k'-k)^2} \sum_{i=1}^{k'-k} \EXPECT{z}{ \left. \frac{(\eta' e^i)^2}{\sigma_i^{(2)} - \sigma_i^{(1)}} \right\vert \forall i =1 \dots n, L_{r_i} \in T_i' }
\\
& \leq c \frac{S}{k'-k}.
\end{align*}
In particular, $\PROB{z}{L < S \vert \forall i =1 \dots n, L_{r_i} \in T_i'} \leq o_{\eta' \to 0}(1) S$ and it implies that
\begin{align*}
& \EXPECT{z}{ \left. \max \left( 1, \frac{1}{\sqrt{\varsigma'}} \right) e^{-\varsigma'/2} \right\vert \forall i =1 \dots n, L_{r_i} \in T_i' } \\
& \leq
\EXPECT{z}{ \left. \max \left( 1, \frac{1}{L^{1/4}} \right) e^{-\sqrt{L}/2} \right\vert \forall i =1 \dots n, L_{r_i} \in T_i' } \\
& \leq \sum_{p = -\infty}^\infty \max \left( 1,  2^{(p+1)/4} \right) e^{- 2^{-(p+1)/2} / 2} \PROB{z}{2^{-p-1} \leq L < 2^{-p} \vert \forall i =1 \dots n, L_{r_i} \in T_i' } \\
& = o_{\eta' \to 0}(1) \sum_{p = -\infty}^\infty 2^{-p} \max \left( 1,  2^{(p+1)/4} \right) e^{- 2^{-(p+1)/2} / 2} = o_{\eta' \to 0}(1).
\end{align*}
This shows \eqref{eq:proof indep3} which finishes the proof.
\end{proof}

\section{Proofs of lemmas on the zero-dimensional Bessel process}\label{sec:Appendix}

This appendix is dedicated to the proofs of the properties we have collected on the zero-dimensional Bessel process throughout the article. Because those properties are fairly classical, we will sometimes be brief.
Recall that we denote by $\brob_r$ the law under which $(R_s)_{s \geq 0}$ is a zero-dimensional Bessel process starting from $r$.

In this section, we will denote $q_s(x,y)$ the transition probability of $(R_s)_{s \geq 0}$. It satisfies the following explicit formula (see Proposition 2.5 of \cite{LawlerBessel} for instance)
\begin{equation}\label{eq:Bessel transition probability}
q_s(x,y) = \frac{x}{s} e^{-\frac{x^2 + y^2}{2s}} I_1 \left( \frac{xy}{s} \right)
\end{equation}
where $I_1$ is a modified Bessel function of the first kind:
\begin{equation}
\label{eq:def Bessel I_1}
I_1(u) = \sum_{n \geq 0} \frac{1}{n!(n+1)!} \left( \frac{u}{2} \right)^{2m+1}.
\end{equation}
We also recall (see \cite{ParisBessel}) that for all $v > u >0$,
\begin{equation}\label{eq:I_1 inequality}
I_1(v) \leq \frac{v}{u}e^{v-u} I_1(u)
\end{equation}
and that $I_1$ has the well-known asymptotic form:
\begin{equation}
\label{eq:I_1 asymptotic}
I_1(u) \underset{u \to \infty}{\sim} \frac{1}{\sqrt{2 \pi u}} e^u.
\end{equation}

We start by proving Lemma \ref{lem:Bessel tail asymptotic}.

\begin{proof}[Proof of Lemma \ref{lem:Bessel tail asymptotic}]
Take $t, \lambda$ and $r$ as in the statement of the lemma.
We have
\[
\PROB{r}{R_t \geq \lambda} = \int_{\lambda}^\infty q_t(r,x)dx.
\]
For all $x \geq \lambda$, we have $r x / t \geq a$.
Hence, by \eqref{eq:Bessel transition probability} and \eqref{eq:I_1 asymptotic}, there exists $C = C(a) >0$ such that for all $x \geq \lambda$,
\[
q_t(r,x) \leq C \frac{r}{t} e^{- \frac{r^2 + x^2}{2t}} \frac{1}{\sqrt{r x /t}} e^{\frac{rx}{s}} \leq C \sqrt{\frac{r}{\lambda t}} e^{-\frac{(x-r)^2}{2t}}.
\]
Using tail estimates of normal random variable, this leads to
\[
\PROB{r}{R_t \geq \lambda} \leq C' \frac{\sqrt{r}}{\lambda} e^{- \frac{(\lambda-r)^2}{2t}} \leq C' \sqrt{r} e^{\frac{\lambda r}{t}} \frac{1}{\lambda} e^{-\frac{\lambda^2}{2t}}.
\]
This proves the first claim.
The second claim follows from the first one and from
\[
\EXPECTB{r}{e^{\gamma R_t}} \leq e^{\frac{\gamma^2t}{4}} + \int_{e^{\gamma^2 t/4}}^\infty \BROB{r}{R_t \geq \frac{\log \mu}{\gamma}} d \mu.
\]
We omit the details.
\end{proof}

We now move on to the proof of Lemma \ref{lem:Bessel first moment}.

\begin{proof}[Proof of Lemma \ref{lem:Bessel first moment}]
For ease of notation, we will assume that $\widetilde{b} = 0$ in the proof.
We are going to show that there exist $c = c(\gamma, \widetilde{\gamma})>0$ and $s_0 = s_0(\gamma, \widetilde{\gamma}, r_0, b) > 0$ such that for all $r \in (0, r_0)$ and $t > s \geq s_0$,
\begin{align}
\label{eq:proof Bessel process close to line1}
\BROB{r}{R_t \geq \gamma t+b, R_s \geq \widetilde{\gamma}s} & \leq \frac{1}{c} e^{-cs} \BROB{r}{R_t \geq \gamma t+b}, \\
\label{eq:proof Bessel process close to line2}
\EXPECTB{r}{ e^{\gamma R_t} \indic{R_s \geq \widetilde{\gamma} s} } & \leq \frac{1}{c} e^{-cs} \EXPECTB{r}{ e^{\gamma R_t} }.
\end{align}
Lemma \ref{lem:Bessel first moment} is then an easy consequence of these estimates.

Define $\eps = (\widetilde{\gamma}-\gamma)/4>0$. Assume that $t$ is large enough so that $\eps t >b$. Take $s < t$ and $\lambda < (\gamma + \eps)t$. We are going to show that
\begin{equation}
\label{eq:proof lemma Bessel}
\BROB{r}{R_t \in [\lambda,(\gamma+\eps)t], R_s \geq \widetilde{\gamma}s} \leq \frac{1}{c} e^{-cs} \BROB{r}{R_t \geq \lambda}
\end{equation}
for some $c = c(\gamma,\widetilde{\gamma})>0$. We will then see that we can conclude with a proof of \eqref{eq:proof Bessel process close to line1} and \eqref{eq:proof Bessel process close to line2} quite quickly.
We have:
\begin{align*}
\BROB{r}{R_t \in [\lambda,(\gamma+\eps)t], R_s \geq \widetilde{\gamma} s}
=
\frac{\widetilde{\gamma}}{\gamma} \int_{\gamma s}^\infty q_s \left( r, \widetilde{\gamma} x/\gamma \right) \BROB{\widetilde{\gamma}x/\gamma}{R_{t-s} \in [\lambda,(\gamma+\eps)t]} dx.
\end{align*}
But by \eqref{eq:I_1 inequality}
\begin{align*}
\BROB{\widetilde{\gamma}x/\gamma}{R_{t-s} \in [\lambda,(\gamma+\eps)t]}
& = \int_{\lambda}^{(\gamma+\eps)t} \frac{\widetilde{\gamma}x/\gamma }{t-s} \exp \left( - \frac{(\widetilde{\gamma}x/\gamma)^2 + y^2}{2(t-s)} \right) I_1 \left( \frac{\widetilde{\gamma}xy/\gamma }{t-s} \right) dy\\
\leq \left( \frac{\widetilde{\gamma}}{\gamma} \right)^2 \int_{\lambda}^{(\gamma+\eps)t} & \frac{x}{t-s} \exp \left( - \frac{(\widetilde{\gamma}x/\gamma)^2 + y^2}{2(t-s)} + \left( \frac{\widetilde{\gamma}}{\gamma} - 1 \right) \frac{xy}{t-s} \right) I_1 \left( \frac{xy}{t-s} \right) dy
\end{align*}
and
\begin{equation}\label{eq:proof Bessel probability density}
q_s(r,\widetilde{\gamma}x/\gamma) \leq \frac{\widetilde{\gamma}}{\gamma} \frac{r}{s} \exp \left( - \frac{r^2 + (\widetilde{\gamma}x/\gamma)^2}{2s} + \left( \frac{\widetilde{\gamma}}{\gamma} - 1 \right) \frac{r x}{s} \right) I_1 \left( \frac{r x}{s} \right).
\end{equation}
After elementary simplifications, we find that $\BROB{r}{R_t \in [\lambda,(\gamma+\eps)t], R_s \geq \widetilde{\gamma}s}$ is at most
\begin{align*}
\left( \frac{\widetilde{\gamma}}{\gamma} \right)^4 \int_{\gamma s}^\infty dx ~q_s(r,x) \int_{\lambda}^{(\gamma+\eps)t} dy ~q_{t-s}(x,y) \exp \left( - \frac{(\widetilde{\gamma}/\gamma-1)x}{s(t-s)} \left( \frac{ \widetilde{\gamma}/\gamma + 1}{2} xt - ys - r(t-s) \right) \right).
\end{align*}
We have chosen $\eps < (\widetilde{\gamma}-\gamma)/2$ so that for all $x \geq \gamma s$ and $y \in [\lambda, (\gamma+\eps)t]$,
\[
\frac{ \widetilde{\gamma}/\gamma + 1}{2} xt - ys \geq cts
\]
for some $c=c(\gamma,\widetilde{\gamma})>0$. Hence if $s$ and $t$ are large enough (depending on $\gamma,\widetilde{\gamma}$ and $r_0$),
\[
\frac{ \widetilde{\gamma}/\gamma + 1}{2} xt - ys - r(t-s) \geq c'ts
\]
for some $c'=c'(\gamma,\widetilde{\gamma})>0$. This implies \eqref{eq:proof lemma Bessel}.

This finishes almost entirely the proof. Indeed, to prove \eqref{eq:proof Bessel process close to line1} we use \eqref{eq:proof Bessel process close to line2} with $\lambda = \gamma t + b$ and we notice that \eqref{eq:proof Bessel probability density} (used with $s = t$ and $\widetilde{\gamma} = \gamma+\eps$) implies that
\begin{align*}
\BROB{r}{R_t \geq (\gamma+\eps)t}
& \leq \left( \frac{\gamma+\eps}{\gamma} \right)^2 \int_{\gamma t}^\infty q_t(r,x) \exp \left( - \frac{((\gamma+\eps)/\gamma - 1)x}{t} \left( \frac{(\gamma+\eps)/\gamma + 1}{2} x -r \right) \right) \\
& \leq \frac{1}{c} e^{-ct} \BROB{r}{R_t \geq \gamma t}
\end{align*}
for some $c=c(\gamma,\widetilde{\gamma})>0$ and if $t$ is large enough. This shows \eqref{eq:proof Bessel process close to line1}.
For \eqref{eq:proof Bessel process close to line2}, we see that \eqref{eq:proof lemma Bessel} gives
\begin{align*}
\EXPECTB{r}{e^{\gamma R_t} \indic{R_s \geq \widetilde{\gamma} s} \indic{R_t \leq (\gamma + \eps)} }
& = \int_0^\infty \BROB{r}{ e^{\gamma R_t} \indic{R_s \geq \widetilde{\gamma} s} \indic{R_t \leq (\gamma + \eps)} \geq \lambda} d \lambda \\
& = \int_0^\infty \BROB{r}{ \frac{\log \lambda}{\gamma} \leq R_t \leq (\gamma+\eps)t, R_s \geq  \widetilde{\gamma} s} d \lambda \\
& \leq \frac{1}{c} e^{-cs}  \int_0^\infty \BROB{r}{\frac{\log \lambda}{\gamma} \leq R_t \leq (\gamma+\eps)t} d \lambda
\leq \frac{1}{c} e^{-cs} \EXPECTB{r}{e^{\gamma R_t}  }.
\end{align*}
On the other hand, we have by \eqref{eq:lem Bessel expectation exponential}
\[
\EXPECTB{r}{e^{\gamma R_t} \indic{R_t \geq (\gamma + \eps)} } \leq \frac{1}{c} e^{-ct} \EXPECTB{r}{e^{\gamma R_t}  }
\]
which concludes the proof of \eqref{eq:proof Bessel process close to line2}. This finishes the proof.
\end{proof}

We finish this appendix by proving Lemma \ref{lem:Bessel Cauchy}.

\begin{proof}[Proof of Lemma \ref{lem:Bessel Cauchy}]
The sum of $n$ independent zero-dimensional squared Bessel processes is still a zero-dimensional squared Bessel process. Hence, by conditioning on $(R_s^{(i)}, s \leq s_0), i = 1 \dots n$, we have
\begin{align*}
& \Brob{ \left. R_t \geq \gamma t+b, \forall s \in [|1,s_0|], R_s \in A_s, \forall s \in [|s_0, t|], R_s \leq \widetilde{\gamma} s + \widetilde{b} \right\vert \forall i=1 \dots n, R_{s_0}^{(i)} >0 } \\
& = \expectB \Bigg[ \BROB{\sqrt{ \sum_{i=1}^n \left( R^{(i)}_{s_0} \right)^2 }}{R_{t-s_0} \geq \gamma t+b, \forall s \in [|1, t-s_0|], R_s \leq \widetilde{\gamma} (s + s_0) + \widetilde{b}} \\
& ~~~~~~~~~~~~~~~~~~~~~~~~~~~~~~~~~~~~~~~~~~~~~~~~~~~~~~~\times \indic{\forall s \in [|1,s_0|], R_s \in A_s}  \Bigg\vert \forall i=1 \dots n, R_{s_0}^{(i)} >0 \Bigg].
\end{align*}
Now we focus on the asymptotic of
\[
\BROB{r}{R_{t-s_0} \geq \gamma t+b, \forall s \in [|1, t-s_0|], R_s \leq \widetilde{\gamma} (s + s_0) + \widetilde{b}}
\]
for a given $r \geq 0$. Take $\eps>0$. By \eqref{eq:lem Bessel first moment good event} of Lemma \ref{lem:Bessel first moment}, there exists $s_0' >0$ such that for all $t \geq s_0' + s_0$,
\begin{align*}
0 & \leq \BROB{r}{R_{t-s_0} \geq \gamma t+b, \forall s \in [|1, s_0'|], R_s \leq \widetilde{\gamma} (s + s_0) + \widetilde{b}} \\
& ~~~~~~ - \BROB{r}{R_{t-s_0} \geq \gamma t+b, \forall s \in [|1, t-s_0|], R_s \leq \widetilde{\gamma} (s + s_0) + \widetilde{b}} \leq \eps.
\end{align*}
But \begin{align*}
& \BROB{r}{R_{t-s_0} \geq \gamma t+b, \forall s \in [|1, s_0'|], R_s \leq \widetilde{\gamma} (s + s_0) + \widetilde{b}} \\
& = \EXPECTB{r}{\indic{\forall s \in [|1, s_0'|], R_s \leq \widetilde{\gamma} (s + s_0) + \widetilde{b}} \BROB{R_{s_0'}}{R_{t - s_0 - s_0'} \geq \gamma t + b}}.
\end{align*}
We could have done the same reasoning with the expectation of $e^{\gamma R_t} \indic{\abs{R_t - \gamma t} \leq M\sqrt{t}}$: the only difference is that we would have to replace
\[
\BROB{R_{s_0'}}{R_{t - s_0 - s_0'} \geq \gamma t + b}
\mathrm{~by~}
\EXPECTB{R_{s_0'}}{e^{\gamma R_{t-s_0-s_0'}} \indic{\abs{R_{t-s_0-s_0'} - \gamma (t-s_0-s_0')} \leq M\sqrt{t-s_0-s_0'}} }
\]
(see also claim \eqref{eq:lem Bessel first moment good event exponential} of Lemma \ref{lem:Bessel first moment}). To conclude the proof, we thus only need to show that for a given $r \geq 0$ and $t_0 \geq 0$,
\[
t e^{\frac{\gamma^2}{2}t} \BROB{r}{R_{t - t_0} \geq \gamma t + b}
\mathrm{~and~}
\frac{1}{\gamma \sqrt{2\pi}}\sqrt{t} e^{- \frac{\gamma^2}{2}t} \EXPECTB{r}{e^{\gamma R_{t-t_0}}\indic{\abs{R_{t-t_0} - \gamma (t-t_0)} \leq M\sqrt{t-t_0}}}
\]
converge and that the limits satisfy \eqref{eq:lem Bessel Cauchy}. This is a simple computation:
\begin{align*}
\BROB{r}{R_t \geq \gamma t + b} = \frac{r}{t} e^{-\frac{r^2}{2t}} \int_{\gamma t + b}^\infty e^{-\frac{x^2}{2t}} I_1 \left( \frac{r x}{t} \right) dx
\underset{t \to \infty}{\sim}
\frac{r I_1(\gamma r)}{\gamma} e^{-b \gamma} \frac{1}{t} e^{-\frac{\gamma^2}{2} t}
\end{align*}
implying that
\[
t e^{\frac{\gamma^2}{2}t} \BROB{r}{R_{t - t_0} \geq \gamma t + b}
\underset{t \to \infty}{\sim}
\frac{r I_1(\gamma r)}{\gamma} e^{-b \gamma - \frac{ \gamma^2}{2} t_0 },
\]
whereas
\begin{align*}
\EXPECTB{r}{e^{\gamma R_t} \indic{\abs{R_t - \gamma t} \leq M \sqrt{t}} } 
& = \frac{r}{t} e^{-\frac{r^2}{2t}} \int_{\gamma t - M \sqrt{t}}^{\gamma t + M \sqrt{t}} e^{-\frac{x^2}{2t} + \gamma x} I_1 \left( \frac{rx}{t} \right) dx \\
& \underset{t \to \infty}{\sim}
r I_1(\gamma r) \frac{1}{\sqrt{t}} e^{\frac{\gamma^2}{2} t}
\int_{-M}^M e^{-y^2/2} dy
\end{align*}
implying that 
\[
\frac{1}{\gamma \sqrt{2\pi}}\sqrt{t} e^{- \frac{\gamma^2}{2}t} \EXPECTB{r}{e^{\gamma R_{t-t_0}}}
\underset{t \to \infty}{\sim}
\frac{r I_1(\gamma r)}{\gamma} e^{- \frac{ \gamma^2}{2} t_0 } (1-p(M))
\]
for $p(M) = 1 - \int_{-M}^M e^{-y^2/2} dy/ \sqrt{2\pi}$.
This concludes the proof.
\end{proof}

\section{Continuity of the local times. Proof of Proposition \ref{prop:continuity}}\label{sec:Appendix continuity local times}

Consider any norm $\norme{\cdot}$ on $\R^2 \times \R$.
By Kolmogorov's continuity theorem, to prove Proposition \ref{prop:continuity}, it is enough to show:

\begin{lemma}\label{lem:continuity local times}
For all $p \geq 1$ and $\eta > \eta' >0$, there exists $C = C(p,\eta, \eta')>0$ such that for all $x,y \in D$ and $0 < \eps, \delta < \eta'$ such that $D(x,\eta) \cup D(y, \eta) \subset D$,
\begin{equation}
\label{eq:prop continuity}
\EXPECT{x_0}{ \abs{ L_{x,\eps}(\tau) - L_{y,\delta}(\tau)}^p}
\leq C \norme{ (x,\eps) - (y,\delta) }^{p/3} \abs{\log \norme{ (x,\eps) - (y,\delta) } }^p.
\end{equation}
\end{lemma}

Let us emphasise that the previous lemma considers the local times $L_{x,\eps}(\tau)$ rather than their normalised versions $L_{x,\eps}(\tau)/\eps$.
Before proving this lemma, we collect one more time a property on the zero-dimensional Bessel process:

\begin{lemma}\label{lem:Bessel continuity}
For all integer $p \geq 1$, there exists $C=C(p)>0$ such that for all $0<s<1$ and for all starting point $r>0$,
\begin{equation}
\label{eq:lem Bessel continuity}
\EXPECTB{r}{\abs{R_s^2 - r^2}^p} \leq C s^{p/2} \max (1, r^{2p}).
\end{equation}
\end{lemma}

\begin{proof}[Proof of Lemma \ref{lem:Bessel continuity}]
Take $\lambda>0$. We are going to bound from above $\BROB{r}{ \abs{R_s^2 - r^2} > \lambda}$. Denoting $T_\lambda := \inf \left\{ t >0 : \abs{R_t^2 - r^2} > \lambda \right\}$, we have:
\begin{align*}
\BROB{r}{ \abs{R_s^2 - r^2} > \lambda}
& \leq \BROB{r}{\sup_{0 \leq t \leq s} \abs{R_t^2 - r^2} > \lambda}
= \BROB{r}{T_\lambda \leq s}
= \BROB{r}{T_\lambda \leq s, \abs{R_{T_\lambda}^2 - r^2} \geq \lambda}.
\end{align*}
And recalling that (see \cite{LawlerBessel})
\[
d(R_t^2) = 2 R_t dW_t
\]
where $(W_t)_{t \geq 0}$ is a standard one-dimensional Brownian motion, we see that $(R_t^2)_{t \geq 0}$ is a local martingale whose quadratic variation is given by
\[
\forall T \geq 0, ~\scalar{R^2}_T = 4 \int_0^T R_t^2 dt.
\]
In particular, $\scalar{R^2}_{T_\lambda} \leq 4 (r^2 + \lambda) T_\lambda$ a.s. Also, because $(R_{t \wedge T_\lambda}^2 - r^2, t \geq 0)$ is bounded, for all $u>0$,
\[
\left( e^{u(R_{t \wedge T_\lambda}^2 - r^2) - u^2 \scalar{R^2}_{t \wedge T_\lambda} /2}, t \geq 0 \right)
\]
is a martingale uniformly integrable.
We thus have by Markov's inequality: for all $u >0$,
\begin{align*}
\BROB{r}{T_\lambda \leq s, R_{T_\lambda}^2 - r^2 \geq \lambda}
& \leq \BROB{r}{T_\lambda \leq s, e^{u(R_{T_\lambda}^2 - r^2) - u^2 \scalar{R^2}_{T_\lambda} /2} \geq e^{u \lambda - u^2 \scalar{R^2}_{T_\lambda} /2}} \\
& \leq \BROB{r}{e^{u(R_{T_\lambda}^2 - r^2) - u^2 \scalar{R^2}_{T_\lambda} /2} \geq e^{u \lambda - 2u^2 (r^2 + \lambda) s }} \\
& \leq e^{-u \lambda + 2 u^2 (r^2 + \lambda)s} = \exp \left( - \frac{\lambda^2}{8(r^2 + \lambda)s} \right)
\end{align*}
with the choice of $u = \lambda / (4(r^2 + \lambda)s)$. The same reasoning works for $\BROB{r}{T_\lambda \leq s, R_{T_\lambda}^2 - r^2 \leq -\lambda}$ and we have found
\[
\forall \lambda >0, ~\BROB{r}{ \abs{R_s^2 - r^2} > \lambda \sqrt{s}} \leq 2 \exp \left( - \frac{\lambda^2}{8(r^2 + \lambda \sqrt{s})} \right).
\]
It then implies that
\begin{align*}
\EXPECTB{r}{ \abs{R_s^2 - r^2}^p }
& = s^{p/2} \int_0^\infty \BROB{r}{\abs{R_s^2 - r^2} > \lambda^{1/p} \sqrt{s}} d \lambda \\
& \leq 2 s^{p/2} \left( \int_0^1 \exp \left( - \frac{\lambda^{2/p}}{8(r^2 + \sqrt{s})} \right) d \lambda + \int_1^\infty \exp \left( - \frac{\lambda^{1/p}}{8(r^2 + \sqrt{s})} \right) d \lambda \right) \\
& \leq C s^{p/2} \left( (r^2 + \sqrt{s})^{p/2} + (r^2 + \sqrt{s})^p \right)
\end{align*}
which yields \eqref{eq:lem Bessel continuity} recalling that $s \leq 1$.
\end{proof}

We are now ready to prove Lemma \ref{lem:continuity local times}.

\begin{proof}[Proof of Lemma \ref{lem:continuity local times}]
The proof will be decomposed in two steps. The first one will bound from above the left hand side term of \eqref{eq:prop continuity} when $x=y$ whereas the second one will treat the case $\delta = \eps$. In the first part, we will be able to transfer all the computations from the local times to the zero-dimensional Bessel process. For the second part, we will use the result of the first step to compare the local time $L_{x,\eps}(\tau)$ with the occupation measure of a narrow annulus around the circle $\partial D(x,\eps)$. Then an elementary argument of monotonicity will allow us to conclude.

In the entire proof, we will consider $p \geq 1$, $\eta > \eta' >0$, $x,y \in D$, $0 < \eps, \delta < \eta'$ such that $D(x,\eta') \cup D(y, \eta') \subset D$. Without loss of generality, we will assume that $\eps \geq \delta$. All the constants appearing in the proof may depend on $p$, $\eta$ and $\eta'$. Before starting off, let us notice that if we fix $K>0$, the result \eqref{eq:prop continuity} is clear if $\abs{x-y} \vee \abs{\eps-\delta} \geq \eps^{3/2}/K$. Indeed, in that case we have:
\begin{align*}
\EXPECT{x_0}{ \abs{ L_{x,\eps}(\tau) - L_{y,\delta}(\tau)}^p}
& \leq 2^{p-1} \EXPECT{x_0}{ L_{x,\eps}(\tau)^p + L_{y,\delta}(\tau)^p}
\leq C \eps^p \abs{\log \eps}^p \\
& \leq C K^{p/3} \eps^{p/2} \abs{\log \eps}^p \left( \abs{x-y} \vee \abs{\eps-\delta} \right)^{p/3} \\
& \leq C' \norme{ (x,\eps) - (y,\delta) }^{p/3}.
\end{align*}
In the rest of the proof, we will thus assume that $\abs{x-y} \vee \abs{\eps-\delta} \leq \eps^{3/2}/K$. It will be convenient for us in particular because it forces $\eps - \abs{x-y}^{2/3} - \abs{x-y}$ to be positive (if $K$ is larger than $2^{3/2}$ say).

\emph{Step 1.} In this step, we assume that $x=y$. To use the links between the local times and the zero-dimensional Bessel process, we consider the different excursions from $\partial D(x,\eps)$ to $\partial D(x,\eta)$: we define $\sigma_0^{(2)}:=0$ and for all $i \geq 1$,
\[
\sigma_i^{(1)} := \inf \left\{ t > \sigma_{i-1}^{(2)}, B_t \in \partial D(x,\eps) \right\}
\mathrm{~and~}
\sigma_i^{(2)} := \inf \left\{ t > \sigma_{i}^{(1)}, B_t \in \partial D(x,\eta) \right\}.
\]
We also denote $N := \max \left\{ i \geq 0: \sigma_i^{(2)} < \tau \right\}$ the number of excursions before exiting the domain $D$ and for all $i \geq 1$, we denote $L_{x,\eps}^i$ and $L_{x,\delta}^i$ the local times of $\partial D(x,\eps)$ and $\partial D(x,\delta)$ accumulated during the $i$-th excursion. To avoid to condition on $N$, we do the following rough bound which follows from Jensen's inequality: for $N_0 \geq 1$, $\EXPECT{x_0}{ \abs{L_{x,\eps}(\tau) - L_{x,\delta}(\tau)}^p }$ is equal to
\begin{align}
& \EXPECT{x_0}{ \abs{ \sum_{i=1}^N L_{x,\eps}^i - L_{x,\delta}^i }^p \indic{N \leq N_0} } + \EXPECT{x_0}{ \abs{L_{x,\eps}(\tau) - L_{x,\delta}(\tau)}^p \indic{N > N_0}} \nonumber \\
& \leq (N_0)^{p-1} \sum_{i=1}^{N_0} \EXPECT{x_0}{ \abs{ L_{x,\eps}^i - L_{x,\delta}^i }^p } + \EXPECT{x_0}{ \abs{L_{x,\eps}(\tau) - L_{x,\delta}(\tau)}^{2p} }^{1/2} \PROB{x_0}{N>N_0}^{1/2} \nonumber \\
& \leq (N_0)^{p} \max_{x_0' \in \partial D(x,\eps)} \EXPECT{x_0'}{ \abs{L_{x,\eps}(\tau_{\partial D(x,\eta)}) - L_{x,\delta}(\tau_{\partial D(x,\eta)})}^p } + C \eps^p \abs{\log \eps}^p \left( \frac{C'}{\abs{\log \eps}} \right)^{N_0/2}. \label{eq:proof continuity a}
\end{align}
If we choose $N_0$ to be the first integer larger than
\[
2p \left. \log \left( \frac{\eps \abs{\log \eps}}{\abs{\eps-\delta}^{1/2}} \right) \right/ \log \left( \frac{\abs{\log \eps}}{C'} \right),
\]
the second term of \eqref{eq:proof continuity a} is at most $C\abs{\eps-\delta}^{p/2}$. Thanks to \eqref{eq:prop local times and Bessel process} and Lemma \ref{lem:Bessel continuity}, the first term of \eqref{eq:proof continuity a} can be easily controlled: denoting $s = \log (\eps / \delta)$ and $R_0 = \sqrt{L_{x,\eps}(\tau_{\partial D(x,\eta)})/\eps}$, for any $x_0' \in \partial D(x,\eps)$, $\EXPECT{x_0'}{ \abs{L_{x,\eps}(\tau_{\partial D(x,\eta)}) - L_{x,\delta}(\tau_{\partial D(x,\eta)})}^p }$ is at most
\begin{align*}
& 2^{p-1} \EXPECT{x_0'}{ \abs{L_{x,\eps}(\tau_{\partial D(x,\eta)}) - \frac{\eps}{\delta} L_{x,\delta}(\tau_{\partial D(x,\eta)})}^p } + 2^{p-1} \abs{\eps - \delta}^p \EXPECT{x_0'}{ \left( \frac{1}{\delta} L_{x,\delta}(\tau_{\partial D(x,\eta)}) \right)^p} \\
& \leq 2^{p-1} \eps^p \EXPECT{x_0'}{ \EXPECTB{R_0}{ \abs{R_s^2 - R_0^2}^p }} + C \abs{\eps-\delta}^p \abs{ \log \delta}^p \\
& \leq C \eps^p (\log (\eps/\delta))^{p/2} \EXPECT{x_0'}{ \max \left( 1, \left( \frac{1}{\eps} L_{x,\eps}(\tau_{\partial D(x,\eta)}) \right)^p \right)} + C \abs{\eps-\delta}^p \abs{ \log \delta}^p \\
& \leq C \eps^p (\log (\eps/\delta))^{p/2} \abs{ \log \eps}^p + C \abs{\eps-\delta}^p \abs{ \log \delta}^p.
\end{align*}
Recalling that $\abs{\eps - \delta} \leq \eps^{3/2}/K$, it leads to
\[
\EXPECT{x_0'}{ \abs{L_{x,\eps}(\tau_{\partial D(x,\eta)}) - L_{x,\delta}(\tau_{\partial D(x,\eta)})}^p }
\leq C \abs{\eps- \delta}^{p/2}.
\]
Coming back to \eqref{eq:proof continuity a}, we have just proved that
\begin{equation}
\label{eq:proof continuity b}
\EXPECT{x_0}{ \abs{L_{x,\eps}(\tau) - L_{x,\delta}(\tau)}^p }
\leq C \abs{ \log \abs{\eps-\delta} }^p \abs{\eps-\delta}^{p/2}.
\end{equation}

\emph{Step 2.} Thanks to the first step we can now assume that $\eps = \delta$. In this step, we will denote for $u \in \R$, $\{u\}_+^p := \max(u,0)^p$. It will be convenient because it is a non-decreasing and convex function. We will also denote $\alpha = \abs{x-y}^{2/3}$. By taking $K$ large enough and decreasing (resp. increasing) slightly the value of $\eta$ (resp. $\eta'$) if necessary, we will be able to use the results of the first part for the circles 
\[
\left\{ \partial D(x,r), \eps - \alpha - \abs{x-y} < r < \eps + \alpha + \abs{x-y} \right\}
\mathrm{~and~}
\left\{ \partial D(y,r), \eps - \alpha < r < \eps + \alpha \right\}.
\]
Recall that $\eps - \alpha - \abs{x-y} >0$ thanks to the assumption $\abs{x-y} \leq \eps^{3/2} / K$.
We notice that $\prob_{x_0}$-a.s.
\begin{equation}
\label{eq:proof continuity c}
I_x := \int_{\eps - \alpha - \abs{x-y}}^{\eps+\alpha+\abs{x-y}} L_{x,r}(\tau) dr
\mathrm{~and~}
I_y := \int_{\eps - \alpha}^{\eps+\alpha} L_{y,r}(\tau) dr
\end{equation}
are equal to the occupation measures up to time $\tau$ of the annuli
\[
D(x,\eps + \alpha + \abs{x-y}) \backslash D(x,\eps - \alpha - \abs{x-y})
\mathrm{~and~}
D(y,\eps + \alpha) \backslash D(y,\eps - \alpha)
\]
respectively. As the first annulus contains the second one, $I_x \geq I_y$ $\prob_{x_0}$-a.s. We have
\begin{align*}
\EXPECT{x_0}{ \left\{ L_{y,\eps}(\tau) - L_{x,\eps}(\tau) \right\}_+^p }
& \leq C \EXPECT{x_0}{ \left\{ \tfrac{1}{2\alpha} I_y - \tfrac{1}{2(\alpha + \abs{x-y})} I_x \right\}_+^p } \\
&  + C \EXPECT{x_0}{ \left\{ L_{y,\eps}(\tau) - \tfrac{1}{2\alpha} I_y \right\}_+^p }
+ C \EXPECT{x_0}{ \left\{ \tfrac{1}{2(\alpha + \abs{x-y})} I_x - L_{x,\eps}(\tau) \right\}_+^p }.
\end{align*}
By our previous observation, the first term on the right hand side is at most
\[
C \left( \frac{\abs{x-y}}{\alpha} \right)^p \EXPECT{x_0}{ \left\{ \frac{1}{2(\alpha + \abs{x-y})} I_x \right\}_+^p } \leq C \abs{x-y}^{p/3}
\]
thanks to our choice of $\alpha$. The two other terms can be controlled thanks to \eqref{eq:proof continuity b}: by Jensen's inequality
\begin{align*}
\EXPECT{x_0}{ \left\{ L_{y,\eps}(\tau) - \frac{1}{2\alpha} I_y \right\}_+^p }
& = \EXPECT{x_0}{ \left\{ \frac{1}{2 \alpha} \int_{\eps -\alpha}^{\eps +\alpha} \left( L_{y,\eps}(\tau) - L_{y,r}(\tau) \right) dr \right\}_+^p } \\
& \leq \frac{1}{2 \alpha} \int_{\eps-\alpha}^{\eps+\alpha} \EXPECT{x_0}{ \left\{ L_{y,\eps}(\tau) - L_{y,r}(\tau) \right\}_+^p } dr
\leq C \alpha^{p/2} \abs{ \log \alpha }^p
\end{align*}
and the third term satisfies a similar upper bound. We have thus obtained:
\[
\EXPECT{x_0}{ \left\{ L_{y,\eps}(\tau) - L_{x,\eps}(\tau) \right\}_+^p } \leq C \abs{x-y}^{p/3} \abs{ \log \abs{x-y}}^p.
\]
By symmetry, the same thing is true for $\EXPECT{x_0}{ \left\{ L_{x,\eps}(\tau) - L_{y,\eps}(\tau) \right\}_+^p }$ which concludes the proof.
\end{proof}

\paragraph*{Acknowledgement}

I am grateful to Nathanaël Berestycki for having supported me during the whole process of writing this paper and for numerous interesting discussions and to the referee for a careful reading of the manuscript and suggestions which helped improve the presentation. I am also grateful to Hugo Falconet and Guillaume Baverez for interesting remarks on a first version of the article.

\bibliographystyle{alpha}
\bibliography{bibliography}

\end{document}